\documentclass[12pt,reqno]{article}

\usepackage[usenames]{color}
\usepackage{amssymb}
\usepackage{amsmath}
\usepackage{amsthm}
\usepackage{amsfonts}
\usepackage{amscd}
\usepackage{graphicx}

\usepackage[colorlinks=true,
linkcolor=webgreen,
filecolor=webbrown,
citecolor=webgreen]{hyperref}

\definecolor{webgreen}{rgb}{0,.5,0}
\definecolor{webbrown}{rgb}{.6,0,0}

\usepackage{color}
\usepackage{fullpage}
\usepackage{float}

\usepackage{graphics}
\usepackage{latexsym}

\begin{document}

\theoremstyle{plain}
\newtheorem{theorem}{Theorem}
\newtheorem{corollary}[theorem]{Corollary}
\newtheorem{proposition}{Proposition}
\newtheorem{lemma}{Lemma}
\newtheorem{example}{Example}
\newtheorem{remark}{Remark}

\numberwithin{equation}{section}
\numberwithin{example}{section}
\numberwithin{theorem}{section}
\numberwithin{remark}{section}
\numberwithin{proposition}{section}
\numberwithin{lemma}{section}

\newcommand{\braces}{\genfrac{\lbrace}{\rbrace}{0pt}{}}

\begin{center}

\vskip 1cm{\large\bf  
Binomial Transforms and the Binomial Convolution of Sequences}
\vskip 1cm
\large
Kunle Adegoke \\ 
Department of Physics and Engineering Physics\\
Obafemi Awolowo University\\
220005 Ile-Ife \\
Nigeria \\
\href{mailto:adegoke00@gmail.com }{\tt adegoke00@gmail.com}\\

\end{center}

\vskip .2 in

\begin{abstract}
Given any two sequences of complex numbers, we establish simple relations between their binomial convolution and the binomial convolution of their individual binomial transforms. We employ these relations to derive new identities involving Fibonacci numbers, Bernoulli numbers,  Catalan numbers, harmonic numbers, odd harmonic numbers, Stirling numbers of the second kind, and binomial coefficients. In addition, we present several results which allow the construction of new binomial-transform pairs from existing ones. Many new relations concerning self-inverse sequences are also derived.
\end{abstract}



\section{Introduction}

Let $n$ be a non-negative integer. Let $(s_k)$ and $(\sigma_k)$, $k=0,1,2,\ldots$, be sequences of complex numbers. It is known that~\cite[p.4]{boyadzhiev18} (see also~\cite[p.136]{knuth3} and~\cite{poblete}):
\begin{equation}\label{first}
s_n  = \sum_{k = 0}^n {( - 1)^k \binom{{n}}{k}\sigma _k }\iff \sigma_n  = \sum_{k = 0}^n {( - 1)^k \binom{{n}}{k}s_k } .
\end{equation}
Such sequences $(s_k)$ and $(\sigma_k)$, as given in~\eqref{first}, will be called a binomial-transform pair of the first kind.

Consider two binomial-transform pairs of the first kind, say $\{(s_k), (\sigma_k)\}$ and $\{(t_k), (\tau_k)\}$, $k=0,1,2,\ldots$. We will establish the following binomial convolution identity:
\begin{equation*}
\sum_{k = 0}^n {( - 1)^k \binom nks_{n-k} t_k }  = \sum_{k = 0}^n {( - 1)^k \binom{{n}}{k}\sigma _k \tau _{n-k} } .
\end{equation*}

Let $(\bar s_k)$ and $(\bar\sigma_k)$, $k=0,1,2,\ldots$, be sequences of complex numbers. It is also known that
\begin{equation}\label{second}
\bar s_n  = \sum_{k = 0}^n {( - 1)^{n-k} \binom{{n}}{k}\bar\sigma _k }\iff \bar\sigma_n  = \sum_{k = 0}^n { \binom{{n}}{k}\bar s_k } .
\end{equation}
Such sequences $(\bar s_k)$ and $(\bar\sigma_k)$, as given in~\eqref{second}, will be called a binomial-transform pair of the second kind. In this case we also call $(\bar\sigma_k)$ the binomial transform of $(\bar s_k)$.

Given two binomial-transform pairs of the second kind, say $\{(\bar s_k), (\bar\sigma_k)\}$ and $\{(\bar t_k), (\bar\tau_k)\}$, $k=0,1,2,\ldots$, we will derive the following binomial convolution identity:
\begin{equation*}
\sum_{k = 0}^n {( - 1)^k \binom{{n}}{k}\bar s_k \bar t_{n-k} }  = \sum_{k = 0}^n {( - 1)^k \binom{{n}}{k}\bar\sigma _k \bar\tau _{n-k} } .
\end{equation*}
Throughout this paper, we will distinguish binomial-transform pairs of the second kind with a bar on the letter representing each sequence. We will use a Greek alphabet to denote each binomial transform and the corresponding Latin alphabet to represent the original sequence.

We will also derive the following binomial convolution of one sequence and the binomial transform of the other:
\begin{equation*}
\sum_{k = 0}^n {\binom{{n}}{k}\bar s_k \bar\tau_{n-k} }  = \sum_{k = 0}^n {\binom{{n}}{k}\bar\sigma _k \bar t _{n-k} }.
\end{equation*}
Given a binomial-transform pair of the first kind, $\{(s_k),(\sigma_k)\}$ and a binomial-transform pair of the second kind, $\{(\bar t_k),(\bar\tau_k)\}$, $k=0,1,2,\ldots$, we will derive the following binomial convolution identity:
\begin{equation*}
\sum_{k = 0}^n {\binom{{n}}{k}s_k \bar t_{n-k} }  = \sum_{k = 0}^n {(-1)^k\binom{{n}}{k} \sigma_k \bar \tau _{n-k} } .
\end{equation*}
We will give a short proof of the known result~\cite{chen07,gould14}:
\begin{equation*}
\sum_{k = 0}^n {( - 1)^k \binom{{n}}{k}s_{k + m} }  = \sum_{k = 0}^m {( - 1)^k \binom{{m}}{k}\sigma _{k + n} } ,
\end{equation*}
and, based partly on it, establish the following identity:
\begin{align*}
&\sum_{k = 0}^n {( - 1)^k \binom{{n}}{k}s_{n - k + m} \sum_{q = 0}^r {( - 1)^q \binom{{r}}{q}\tau _{k + q} } }\\
&\qquad  = \sum_{k = 0}^n {( - 1)^k \binom{{n}}{k}t_{n - k + r} \sum_{p = 0}^m {( - 1)^p \binom{{m}}{p}\sigma _{k + p} } } .
\end{align*}
In fact, we will prove that
\begin{equation*}
\sum_{k=0}^n{(-1)^k\binom nk\sum_{p = 0}^r {( - 1)^p \binom{{r}}{p}s _{k + p + m} }}=\sum_{k = 0}^m {( - 1)^k \binom mk\sigma _{n + k + r} },
\end{equation*}
and hence establish
\begin{align*}
&\sum_{k = 0}^n {( - 1)^k \binom{{n}}{k}\sum_{p = 0}^r {( - 1)^p \binom{{r}}{p}s_{n - k + p + m} } \sum_{j = 0}^u {( - 1)^j \binom{{u}}{j}t_{k + j + v} } }\\
&\qquad  = \sum_{k = 0}^n {( - 1)^k \binom{{n}}{k}\sum_{p = 0}^m {( - 1)^p \binom{{m}}{p}\sigma _{k + p + r} } \sum_{j = 0}^v {( - 1)^j \binom{{v}}{j}\tau _{n - k + j + u} } } ,
\end{align*}
valid for non-negative integers $m$, $n$, $r$, $u$ and $v$ and any binomial-transform pairs $\{(s_k),(\sigma_k\}$ and $\{(t_k),(\tau_k\}$, $k=0,1,2,\ldots,$ of the first kind.

Binomial transform of products of sequences will be computed in Section~\ref{sec.products}. The results will complement those of Boyadzhiev~\cite{boyadzhiev16}.

In Section~\ref{sec.relations}, we will derive several results which allow one to construct new binomial-transform pairs from existing ones. For example we will show that if $(t_k)$ and $(\tau_k)$, $k=0,1,2,\dots$, are a binomial-transform pair of the first kind, then so are the sequences $(a_k)$ and $(\alpha_k)$, $k=0,1,2,\dots$, where
\begin{equation*}
a_k=\frac{1}{{\left( {k + 1} \right)\left( {k + 2} \right)}}\sum\limits_{j = 0}^k {t_j }\text{ and } \alpha_k=\frac{1}{{\left( {k + 1} \right)\left( {k + 2} \right)}}\sum\limits_{j = 0}^k {\tau _j } .
\end{equation*}
If, in~\eqref{first}, $\sigma_k=s_k$ for every $k=0,1,2,\ldots$, then the sequence $(s_k)$ is called an invariant sequence or a self-inverse sequence. If $\sigma_k=-s_k$ for every $k=0,1,2,\ldots$, then the sequence $(s_k)$ will be called an anti-self-inverse sequence. We will show in Section~\ref{sec.invariant} that every invertible sequence has a self-inverse sequence and an anti-self-inverse sequence associated with it. In particular, we will show that if the sequence $(s_k)$, $k=0,1,2,\ldots$, is a self-inverse sequence, then the sequence $(w_k)$ given by
\begin{equation*}
w_k  = \frac1{k+1}\sum_{j = 0}^{k-1} s_j,\quad k=0,1,2,\ldots, 
\end{equation*}
is an anti-self-inverse sequence, and vice versa.
\begin{remark}
Although conversion from one kind of binomial transform to the other is, in general, very straightforward (see Remark~\ref{rem.r7dv221}), it is convenient to retain the distinction between the two kinds because binomial coefficient identities occur naturally in either of the two kinds. Thus, binomial convolution identities can be written directly without the need to convert from one kind to the other.
\end{remark}

The binomial convolution identities developed in this paper will facilitate the discovery of an avalanche of new combinatorial identities. We now present a couple of identities, selected from our results, to whet the reader's appetite for reading on.

We derived various general identities involving binomial transform pairs, such as

\begin{equation*}
\sum_{k = 0}^n {( - 1)^{n - k} \binom{{n}}{k}\binom{{y - k}}{x}t_{n - k} }  = \sum_{k = 0}^n {( - 1)^k \binom{{n}}{k}\binom{{y - k}}{{y - x}}\tau _{n - k} }, 
\end{equation*}

\begin{equation*}
\sum_{k = 0}^n {( - 1)^{n - k} \binom{{n}}{k}\frac{{m\,H_{k + m} }}{{k + m}}\,t_{n - k} } = \sum_{k = 0}^n {( - 1)^k \binom{{n}}{k}\binom{{k + m}}{m}^{ - 1} \left( {H_{k + m}  - H_k } \right)\tau _{n - k} },
\end{equation*}
and
\begin{equation*}
\sum_{k = 0}^n {\binom{{n}}{k}H_{k + m} \bar t_{n - k} }  = H_m \bar \tau _n  - \sum_{k = 1}^n {( - 1)^k \binom{{n}}{k}\binom{{k + m}}{m}^{ - 1}\frac1k\, \bar \tau _{n - k} } ,
\end{equation*}
where $H_k$ is a harmonic number and $m$ is a complex number that is not a negative integer.

We discovered the following binomial convolution of harmonic and odd harmonic numbers:
\begin{equation*}
\sum_{k = 0}^n {( - 1)^{n - k} \binom{{n}}{k}H_k O_{n - k} }= \sum_{k = 1}^{n - 1} {( - 1)^k \binom{{n}}{k}\frac{{2^{2\left( {n - k} \right) - 1} }}{{k\left( {n - k} \right)}}\binom{{2\left( {n - k} \right)}}{{n - k}}^{ - 1} } .
\end{equation*}
We obtained some identities involving Bernoulli numbers and Bernoulli polynomials, including the following:
\begin{equation*}
\sum_{k = 0}^n {\binom nkkB_k }  
= \begin{cases}
 nB_n,&\text{if $n$ is an even integer or $n=1$} ; \\ 
  - nB_{n - 1},&\text{if $n>1$ is an odd integer}; 
 \end{cases} 
\end{equation*}
\begin{equation*}
\sum_{k = 0}^n {( - 1)^k \binom{{n}}{k}\left( {\frac{y}{w}} \right)^k B_{n - k} (y)B_k (w)}  = \frac{{ny}}{{2w}}\left( {1 - \left( {\frac{y}{w}} \right)^{n - 2} } \right)B_{n - 1},\quad\text{$n$ an odd integer}, 
\end{equation*}
and
\begin{equation*}
\sum_{k = 0}^n {( - 1)^{n - k} \binom{{n}}{k}\binom{{x}}{k}B_{n - k} }  = \sum_{k = 0}^n {\binom{{n}}{k}\binom{{x + k}}{k}B_{n - k} } ,
\end{equation*}
where $x$, $y$ and $w$ are complex numbers.

We found the following binomial convolution of Fibonacci and Bernoulli numbers, valid for $n$ an even integer:
\begin{equation*}
\sum_{k = 0}^n {\binom{{n}}{k}F_k B_{n - k} }  = 0.
\end{equation*}
We obtained the following polynomial identities involving, respectively, harmonic numbers and odd harmonic numbers:
\begin{equation*}
\sum_{k = 0}^n {\binom{{n}}{k}H_k t^k }  =\sum_{k = 1}^n {( - 1)^{k - 1} \binom{{n}}{k}\frac{1}{k}\,t^k \left( {1 + t} \right)^{n - k} } 
\end{equation*}
and
\begin{equation*}
\sum_{k = 0}^n {\binom{{n}}{k}O_k t^k }  =\sum_{k = 1}^n {( - 1)^{k - 1} \binom{{n}}{k}2^{2k-1}\binom{2k}k^{-1}\frac{1}{k}\,t^k \left( {1 + t} \right)^{n - k} } ,
\end{equation*}
with the special values
\begin{equation*}
\sum_{k = 0}^n {\binom{{n}}{k}H_k }  = \sum_{k = 1}^n {( - 1)^{k - 1} 2^{n - k} \binom{{n}}{k}\frac{1}{k}} 
\end{equation*}
and
\begin{equation*}
\sum_{k = 0}^n {\binom{{n}}{k}O_k }  = \sum_{k = 1}^n {( - 1)^{k - 1} 2^{n + k - 1} \binom{{n}}{k}\binom{2k}k^{-1}\frac{1}{k}}. 
\end{equation*}
Let
\begin{equation*}
C_n=\frac1{n+1}\binom{2n}n,\quad n=0,1,2,\ldots,
\end{equation*}
be the Catalan numbers. By employing a result of Miki\'c~\cite{mikic}, we derived the following identity:
\begin{equation*}
\sum_{k = 0}^n {( - 1)^k \binom{{n}}{k}\binom{{k}}{{\left\lfloor {k/2} \right\rfloor }}\binom{{n - k}}{{\left\lfloor {\left( {n - k} \right)/2} \right\rfloor }}}  =
\begin{cases}
 C_{n/2} \binom n{n/2},&\text{$n$ even}; \\ 
 0,&\text{$n$ odd}; 
 \end{cases}   
\end{equation*}
where here and throughout this paper, $\lfloor x\rfloor$ is the greatest integer less than or equal to $x$.

We derived a couple of binomial coefficient identities, including the following:
\begin{equation*}
\sum_{k = 0}^n {\binom{{n}}{k}\binom{{n - k + m}}{u}\binom{{r}}{{j - k}}}  = \sum_{k = 0}^n {\binom{{n}}{k}\binom{{n - k + r}}{j}\binom{{m}}{{u - k}}},
\end{equation*}
which holds for complex numbers $j$, $m$, $r$, and $u$.

We derived many, mostly new, binomial transform identities involving special numbers, to serve as a pool of binomial-transform pairs. These are presented in Section~\ref{sec.table}.

We close this section by giving a short description of each of the special numbers from which illustrations will be drawn.

Binomial coefficients are defined, for non-negative integers $i$ and $j$, by
\begin{equation*}
\binom ij=
\begin{cases}
\dfrac{{i!}}{{j!(i - j)!}}, & \text{$i \ge j$};\\
0, & \text{$i<j$};
\end{cases}
\end{equation*}
the number of distinct sets of $j$ objects that can be chosen from $i$ distinct objects.

Generalized binomial coefficients are defined for complex numbers $r$ and $s$ by
\begin{equation}\label{y89d722}
\binom rs= \frac{{\Gamma (r + 1)}}{{\Gamma (s + 1)\Gamma (r - s + 1)}},
\end{equation}
where the Gamma function, $\Gamma(z)$, is defined for $\Re(z)>0$ by
\begin{equation*}
\Gamma (z) = \int_0^\infty  {e^{ - t} t^{z - 1}dt}  = \int_0^\infty  {\left( {\log (1/t)} \right)^{z - 1}dt},
\end{equation*}
and is extended to the rest of the complex plane, excluding the non-positive integers, by analytic continuation.

Harmonic numbers, $H_n$, and odd harmonic numbers, $O_n$, are defined for non-negative integers $n$ by
\begin{equation*}
H_n=\sum_{k=1}^n\frac1k,\quad O_n=\sum_{k=1}^n\frac1{2k-1},\quad H_0=0=O_0.
\end{equation*}

The Bernoulli numbers, $B_k$, are defined by the generating function
\begin{equation*}
\frac{z}{{e^z  - 1}} = \sum_{k = 0}^\infty  {B_k \frac{{z^k }}{{k!}}},\quad z<2\pi\,,
\end{equation*}
and the Bernoulli polynomials by the generating function
\begin{equation*}
\frac{{ze^{xz} }}{{e^z  - 1}} = \sum_{k = 0}^\infty  {B_k (x)\frac{{z^k }}{{k!}}},\quad|z|<2\pi\,.
\end{equation*}
Clearly, $B_k=B_k(0)$.

The first few Bernoulli numbers are
\begin{equation*}
B_0  = 1,\,B_1  =  - \frac{1}{2},\,B_2  = \frac{1}{6},\,B_3  = 0,\,B_4  =  - \frac{1}{{30}},\,B_5  = 0,\,B_6  = \frac{1}{{42}},\,B_7  = 0,\,\ldots\,,
\end{equation*}
while the first few Bernoulli polynomials are
\begin{equation*}
\begin{split}
&B_0 (x)= 1,\quad B_1 (x) = x - \frac{1}{2},\quad B_2 (x) = x^2  - x + \frac{1}{6},\quad B_3 (x) = x^3  - \frac{3}{2}x^2  + \frac{1}{2}x.
\end{split}
\end{equation*}
An explicit formula for the Bernoulli polynomials is
\begin{equation*}
\frac{B_n (x)}{x^n} = \sum_{k = 0}^n {\binom nk\frac{B_k}{x^k} }\,,
\end{equation*}
while a recurrence formula for them is
\begin{equation}\label{eq.v8egj4b}
B_n (x + 1) = \sum_{k = 0}^n {\binom nkB_k (x)}.
\end{equation}
The Fibonacci numbers, $F_n$, and the Lucas numbers, $L_n$, are defined, for \mbox{$n\in\mathbb Z$}, through the recurrence relations 
\begin{equation*}
F_n=F_{n-1}+F_{n-2}, \mbox{($n\ge 2$)},\quad\mbox{$F_0=0$, $F_1=1$};
\end{equation*}
and
\begin{equation*}
L_n=L_{n-1}+L_{n-2}, \mbox{($n\ge 2$)},\quad\mbox{$L_0=2$, $L_1=1$};
\end{equation*}
with
\begin{equation}
F_{-n}=(-1)^{n-1}F_n\,,\quad L_{-n}=(-1)^nL_n\,.
\end{equation}

Explicit formulas (Binet formulas) for the Fibonacci and Lucas numbers are
\begin{equation*}
F_n  = \frac{{\alpha ^n  - \beta ^n }}{{\alpha  - \beta }},\quad L_n  = \alpha ^n  + \beta ^n,\quad n\in\mathbb Z,
\end{equation*}
where $\alpha=(1+\sqrt 5)/2$ is the golden ratio and $\beta=(1-\sqrt 5)/2=-1/\alpha$.

The generalized Fibonacci sequence, $(G_n)$, $n=0,1,2,\ldots$, the so-called gibonacci sequence, is a generalization of $F_n$ and $L_n$; having the same recurrence relation as the Fibonacci sequence but with arbitrary initial values. Thus
\begin{equation*}
G_n  = G_{n - 1}  + G_{n - 2},\quad (n \ge 2);
\end{equation*}
with
\begin{equation*}
\quad G_{-n}=G_{-n+2}-G_{-n+1}.
\end{equation*}
where $G_0$ and $G_1$ arbitrary numbers (usually integers) not both zero. 

Stirling numbers of the second kind are defined for non-negative integers $m$ and $n$ through the generating function
\begin{equation*}
\sum_{k = 0}^n {\binom{{n}}{k}\braces{{ m}}{k}k!}  = n^m ,
\end{equation*}
and hence by binomial inversion, they have the explicit representation
\begin{equation}\label{ib2djg2}
\braces{{ m}}{n} = \frac{{( - 1)^n }}{{n!}}\sum_{k = 0}^n {( - 1)^k \binom{{n}}{k}k^m } .
\end{equation}
The Stirling numbers of the second kind obey the following recursion relation~\cite[Equation (6.3), p.259]{graham}:
\begin{equation}\label{qkejzsh}
\braces{{ n}}{m} = \braces{{ n - 1}}{{m - 1}} + m\braces{{ n - 1}}{m},\quad m=1,2,\ldots,n,
\end{equation}
and have the following special values:
\begin{equation}
\braces{{ n}}{0} = \delta _{n0} ,\quad\braces{{ n}}{1} = 1 = \braces{{ n}}{n},\quad\braces{{ n}}{{n - 1}} = \binom{{n}}{2},\quad\braces{{ n}}{2} = 2^{n - 1}  - 1,
\end{equation}
and
\begin{equation}\label{urlbi68}
\braces nm=0 \text{ for $n<m$}.
\end{equation}

\section{Identities involving binomial-transform pairs of the first kind}\label{sec.first}
Our first main result is stated in Theorem~\ref{thm.main1}; we require the result stated in the next lemma for its proof.
\begin{lemma}\label{lem.qqdca04}
Let $\left\{(a_k),(\alpha_k)\right\}$, $k=0,1,2,\ldots$, be a binomial-transform pair of the first kind. Let $\mathcal L_x$ be a linear operator defined by $\mathcal L_x(x^j)=a_j$ for every complex number $x$ and every non-negative integer~$j$. Then $\mathcal L_x((1-x)^j)=\alpha_j$.
\end{lemma}
\begin{proof}
We have
\begin{align*}
\mathcal L_x\left( {\left( {1 - x} \right)^j } \right) = \mathcal L_x\left( {\sum_{k = 0}^j {( - 1)^k \binom{{j}}{k}x^k } } \right) &= \sum_{k = 0}^j {( - 1)^k \binom{{j}}{k}\mathcal L_x\left( {x^k } \right)} \\
& = \sum_{k = 0}^j {( - 1)^k \binom{{j}}{k}a_k }\\ 
&=\alpha_j.
\end{align*}
\end{proof}
\begin{theorem}\label{thm.main1}
Let $n$ be a non-negative integer. If $\{(s_k), (\sigma_k)\}$ and $\{(t_k), (\tau_k)\}$, $k=0,1,2,\ldots$, are binomial-transform pairs of the first kind, then
\begin{equation}\label{main1}
\sum_{k = 0}^n {( - 1)^{n - k} \binom{{n}}{k}s_k t_{n-k} }  = \sum_{k = 0}^n {( - 1)^k \binom{{n}}{k}\sigma _k \tau _{n-k} } .
\end{equation}
\end{theorem}

\begin{proof}
Let $n$ be a non-negative integer. Let $x$ and $y$ be complex numbers. Consider the following identity whose veracity is readily established by the binomial theorem:
\begin{equation}\label{eq.fltsd62}
\sum_{k = 0}^n {( - 1)^k \binom{{n}}{k} y^{n - k}x^k }  = \sum_{k = 0}^n {( - 1)^{n - k} \binom{{n}}{k} \left( {1 - y} \right)^{n - k} \left( {1 - x} \right)^k} .
\end{equation}
Let $(t_j)$, $j=0,1,2,\ldots$, be a sequence of complex numbers. Let $\mathcal L_x(x^j)=t_j$ .

Operate on both sides of~\eqref{eq.fltsd62} with $\mathcal L_x$ to obtain
\begin{equation*}
\sum_{k = 0}^n {( - 1)^k \binom{{n}}{k} y^{n - k}t_k }  = \sum_{k = 0}^n {( - 1)^{n - k} \binom{{n}}{k} \left( {1 - y} \right)^{n - k} \tau _k} ,
\end{equation*}
where
\begin{equation*}
\tau _k  = \mathcal L_x((1-x)^k)=\sum_{i = 0}^k {( - 1)^i \binom{{k}}{i}t_i }.
\end{equation*}
Thus,
\begin{equation}\label{eq.renl1hy}
\sum_{k = 0}^n {( - 1)^{n - k} \binom{{n}}{k}y^kt_{n - k}  }  = \sum_{k = 0}^n {( - 1)^k \binom{{n}}{k}\left( {1 - y} \right)^k\tau _{n - k} } .
\end{equation}
Let $(s_j)$, $j=0,1,2,\ldots$, be a sequence of complex numbers. Let $\mathcal L_y(y^j)=s_j$. The action of $\mathcal L_y$ on~\eqref{eq.renl1hy} produces
\begin{equation*}
\sum_{k = 0}^n {( - 1)^{n - k} \binom{{n}}{k}s_kt_{n - k}  }  = \sum_{k = 0}^n {( - 1)^k \binom{{n}}{k}\sigma_k\tau _{n - k} } ,
\end{equation*}
where
\begin{equation*}
\sigma _k  = \mathcal L_y((1-y)^k)=\sum_{i = 0}^k {( - 1)^i \binom{{k}}{i}s_i },
\end{equation*}
and the proof is complete.
\end{proof}
In the next example we illustrate Theorem~\ref{thm.main1}.
\begin{example}
Consider the following identity~\cite[Equation (10.7)]{boyadzhiev18}, where $x$ and $y$ are complex numbers and $n$ is a non-negative integer:
\begin{equation}\label{eq.zw4rawn}
\sum_{k = 0}^n {\binom{{n}}{k}( - 1)^k \binom{{y - k}}{x} }  = \binom{{y - n}}{{y - x}}.
\end{equation}
We can choose
\begin{equation*}
s_k  = \binom{{y - k}}{x},\quad\sigma _k  = \binom{{y - k}}{{y - x}},
\end{equation*}
and use these in~\eqref{main1} to obtain the following identity
\begin{equation}\label{eq.o20tqr7}
\sum_{k = 0}^n {( - 1)^{n - k} \binom{{n}}{k}\binom{{y - k}}{x}t_{n - k} }  = \sum_{k = 0}^n {( - 1)^k \binom{{n}}{k}\binom{{y - k}}{{y - x}}\tau _{n - k} }, 
\end{equation}
which holds for every binomial-transform pair of the first kind $\{(t_k),(\tau_k)\}$.

As an immediate consequence of~\eqref{eq.o20tqr7}, we have the following polynomial identity in the complex variable $t$:
\begin{equation*}
\sum_{k = 0}^n {( - 1)^{n - k} \binom{{n}}{k}\binom{{y - k}}{x}t^{n - k} }  = \sum_{k = 0}^n {( - 1)^k \binom{{n}}{k}\binom{{y - k}}{{y - x}}\left( {1 - t} \right)^{n - k} } ,
\end{equation*}
which can also be written as
\begin{equation}
\sum_{k = 0}^n {( - 1)^k \binom{{n}}{k}\binom{{y + k}}{x}t^k }  = \sum_{k = 0}^n {( - 1)^{n - k} \binom{{n}}{k}\binom{{y + k}}{{y + n - x}}\left( {1 - t} \right)^k },
\end{equation}
which is valid for all complex numbers $x$,  $y$ and $t$.
\end{example}

\begin{example}\label{ex.ur405iz}
If we identify
\begin{equation*}
s_k  = \frac{{H_{k + m} }}{{k + m}},\quad\sigma _k  = \frac{{H_{k + m}  - H_k }}{m}\binom{{k + m}}{k}^{ - 1}, 
\end{equation*}
from the following identity~\cite[Equation (9.46)]{boyadzhiev18}:
\begin{equation*}
\sum_{k = 0}^n {( - 1)^k \binom{{n}}{k}\frac{{H_{k + m} }}{{k + m}}}  = \frac{{H_{n + m}  - H_n }}{m}\binom{{n + m}}{n}^{ - 1},\quad m\ne 0,
\end{equation*}
then from~\eqref{main1}, we obtain the following identity
\begin{align}\label{eq.awm96g7}
\sum_{k = 0}^n {( - 1)^{n - k} \binom{{n}}{k}\frac{{m\,H_{k + m} }}{{k + m}}\,t_{n - k} } = \sum_{k = 0}^n {( - 1)^k \binom{{n}}{k}\binom{{k + m}}{m}^{ - 1} \left( {H_{k + m}  - H_k } \right)\tau _{n - k} },
\end{align}
which holds for any binomial-transform pair of the first kind, $\{(t_k),(\tau_k)\}$, for every complex number $m$ that is not a non-positive number and every non-negative integer $n$.

In particular, since
\begin{equation}\label{q96hz7w}
\sum_{k = 0}^n {( - 1)^k \binom{{n}}{k}\frac{{G_{tk + r} }}{{L_t^k }}}  = \frac{{( - 1)^r }}{{L_t^n }}\left( {G_0 L_{tn - r}  - G_{tn - r} } \right),
\end{equation}
using
\begin{equation*}
t_k  = \frac{{G_{tk + r} }}{{L_t^k }},\quad\tau _k  = \frac{{( - 1)^r }}{{L_t^k }}\left(G_0L_{tk-r}-G_{tk-r}\right),
\end{equation*}
in~\eqref{eq.awm96g7} yields the following general Fibonacci-harmonic number identity:
\begin{align}
&\sum_{k = 0}^n {( - 1)^{n - k} \binom{{n}}{k}\frac{{m\,H_{k + m} }}{{k + m}}L_t^k G_{t(n - k) + r} },\nonumber\\
&\qquad  = \sum_{k = 0}^n {( - 1)^k \binom{{n}}{k}\binom{{k + m}}{m}^{ - 1} \left( {H_{k + m}  - H_k } \right)L_t^k \left( {G_0 L_{t(n - k) - r}  - G_{t(n - k) - r} } \right)}; 
\end{align}
and, in particular,
\begin{equation}
\sum_{k = 0}^n {( - 1)^{n - k} \binom{{n}}{k}\frac{{m }}{{k + m}}H_{k + m}F_{n - k} }  = \sum_{k = 0}^n {( - 1)^{k-1} \binom{{n}}{k}\binom{{k + m}}{m}^{ - 1} \left( {H_{k + m}  - H_k } \right)F_{n - k} } ,
\end{equation}
with the special value
\begin{equation}
\sum_{k = 0}^n {( - 1)^{n - k} \binom{{n}}{k}\frac{{H_{k + 1} F_{n - k} }}{{k + 1}}}  = \sum_{k = 0}^n {( - 1)^{k-1} \binom{{n}}{k}\frac{{F_{n - k} }}{{\left( {k + 1} \right)^2 }}} .
\end{equation}
\end{example}

\begin{example}
From~\eqref{eq.zw4rawn} and~\cite[Equation (10.1)]{boyadzhiev18}:
\begin{equation}
\sum_{k = 0}^n {( - 1)^k \binom{{n}}{k}\binom{{x}}{k}\binom{{y}}{k}^{ - 1} }  = \binom{{y - x}}{n}\binom{{y}}{n}^{ - 1},\quad y-n\not\in\mathbb Z^-,y\in\mathbb C,\label{eq.a45516v}
\end{equation}
we can choose
\begin{equation*}
s_k  = \binom{{y - k}}{x},\quad\sigma _k  = \binom{{y - k}}{{y - x}},
\end{equation*}
and
\begin{equation*}
t_k  = \binom{{u}}{k}\binom{{v}}{k}^{ - 1} ,\quad\tau _k  = \binom{{v - u}}{k}\binom{{v}}{k}^{ - 1}. 
\end{equation*}
Using these in~\eqref{main1}, we obtain the following identity
\begin{align}
&\sum_{k = 0}^n {( - 1)^{n - k} \binom{{n}}{k}\binom{{y - k}}{x}\binom{u}{{n - k}}\binom{{v}}{{n - k}}^{ - 1} }\nonumber \\
&\qquad = \sum_{k = 0}^n {( - 1)^k \binom{{n}}{k}\binom{{y - k}}{{y - x}}\binom{{v - u}}{{n - k}}\binom{{v}}{{n - k}}^{ - 1} } ,
\end{align}
which is valid for complex numbers $x$, $y$, $u$ and $v$ such that $v-n$ is not a negative integer.

In particular, at $v=n$, we get
\begin{equation}
\sum_{k = 0}^n {( - 1)^{n - k} \binom{{y - k}}{x}\binom{u}{{n - k}}}  = \sum_{k = 0}^n {( - 1)^k \binom{{y - k}}{{y - x}}\binom{{n - u}}{{n - k}}} ,
\end{equation}
with the special value
\begin{equation}
\sum_{k = 0}^n {( - 1)^k \binom nk\binom{{y - k}}{{y - x}}}=\binom{y-n}x ,
\end{equation}
which is the binomial transform of the first kind of~\eqref{eq.zw4rawn}.
\end{example}

\begin{example}\label{ex.omtqrl4}
From~\cite[Equation (9.37)]{boyadzhiev18}:
\begin{equation}
\sum_{k = 0}^n {( - 1)^k \binom{{n}}{k}H_{k + m} }  =  - \frac{1}{n}\binom{{n + m}}{m}^{ - 1},\quad n\ne0,\label{eq.wt34kp1}
\end{equation}
and~\cite[Equation (9.51)]{boyadzhiev18}:
\begin{equation}\label{eq.jupxl7y}
\sum_{k = 0}^n {( - 1)^k \binom{{n}}{k}O_k }  =  - \binom{{2n}}{n}^{ - 1} \frac{{2^{2n - 1} }}{n},
\end{equation}
we can identify
\begin{equation*}
s_k  = H_{k + m} ,\quad\sigma _k  =  \frac{{\delta _{k0}(1+H_m)-1 }}{{k + \delta _{k0} }}\binom{{k + m}}{m}^{ - 1},  
\end{equation*}
and
\begin{equation*}
t_k  = O_k ,\quad\tau _k  =  - \frac{{1 - \delta _{k0} }}{{k + \delta _{k0} }}\binom{{2k}}{k}^{ - 1} 2^{2k - 1}. 
\end{equation*}
Here and throughout this paper, $\delta_{ij}$ is Kronecker's delta having the value $1$ when $i=j$ and zero otherwise.

Using these in~\eqref{main1} leads to the following harmonic number identity:
\begin{align}
&\sum_{k = 0}^n {( - 1)^{n - k} \binom{{n}}{k}H_{k + m} O_{n - k} }\nonumber\\
&\qquad  = -\frac{H_m}n\binom{2n}n^{-1}2^{2n-1}+\sum_{k = 1}^{n - 1} {( - 1)^k \binom{{n}}{k}\frac{{2^{2\left( {n - k} \right) - 1} }}{{k\left( {n - k} \right)}}\binom{{k + m}}{m}^{ - 1} \binom{{2\left( {n - k} \right)}}{{n - k}}^{ - 1} } ,
\end{align}
which is valid for every positive integer $n$ and every complex number $m$ that is not a negative integer.

In particular, we have the following alternating binomial convolution of harmonic number and odd harmonic number:
\begin{align}
\sum_{k = 0}^n {( - 1)^{n - k} \binom{{n}}{k}H_k O_{n - k} }= \sum_{k = 1}^{n - 1} {( - 1)^k \binom{{n}}{k}\frac{{2^{2\left( {n - k} \right) - 1} }}{{k\left( {n - k} \right)}}\binom{{2\left( {n - k} \right)}}{{n - k}}^{ - 1} } .
\end{align}

\end{example}

\begin{corollary}
Let $n$ be a non-negative integer. If $\{(s_k), (\sigma_k)\}$ and $\{(t_k), (\tau_k)\}$, $k=0,1,2,\ldots$, are binomial-transform pairs of the first kind, then
\begin{equation}
\sum_{k = 0}^n {( - 1)^{n - k} \binom{{n}}{k}\sigma_k t_{n-k} }  = \sum_{k = 0}^n {( - 1)^k \binom{{n}}{k}s _k \tau _{n-k} }.
\end{equation}

\end{corollary}

\begin{corollary}
Let $n$ be a non-negative integer. If $\{(s_k), (\sigma_k)\}$, $k=0,1,2,\ldots$, is a binomial-transform pair of the first kind, then
\begin{equation}\label{eq.ci4kck8}
\sum_{k = 0}^n {( - 1)^{n - k} \binom{{n}}{k}s_k s_{n-k} }  = \sum_{k = 0}^n {( - 1)^k \binom{{n}}{k}\sigma _k \sigma _{n-k} } .
\end{equation}
\end{corollary}

\begin{remark}
We deduce from~\eqref{eq.ci4kck8} that if $n$ is a non-negative integer, then
\begin{equation}\label{eq.pt7l41w}
\sum_{k = 0}^n {( - 1)^k \binom{{n}}{k}s_k s_{n-k} }  = \sum_{k = 0}^n {( - 1)^k \binom{{n}}{k}\sigma _k \sigma _{n-k} } ,
\end{equation}
with each sum being $0$ for odd $n$.
\end{remark}

\begin{example}
In view of the following identity~\cite{mathworld}:
\begin{equation}
\sum_{k = 0}^n {( - 1)^k \binom{{n}}{k}2^{ - k} \binom{{k}}{{\left\lfloor {k/2} \right\rfloor }}}  = 2^{ - n} C_n ,
\end{equation}
we can choose
\begin{equation}
s _k  = 2^{ - k} \binom{{k}}{{\left\lfloor {k/2} \right\rfloor }}, \quad\sigma_k  = 2^{ - k} C_k ,
\end{equation}
and use in~\eqref{eq.ci4kck8} to obtain
\begin{equation}\label{ubq9lia}
\sum_{k = 0}^n {( - 1)^k \binom{{n}}{k}\binom{{k}}{{\left\lfloor {k/2} \right\rfloor }}\binom{{n - k}}{{\left\lfloor {\left( {n - k} \right)/2} \right\rfloor }}}  = ( - 1)^n \sum_{k = 0}^n {( - 1)^k \binom{{n}}{k}C_k C_{n - k} } .
\end{equation}
Miki\'c~\cite{mikic} has shown that if $n$ is a non-negative integer, then
\begin{equation}\label{cjrojsy}
\sum_{k = 0}^n {( - 1)^k \binom{{n}}{k}C_k C_{n - k} }  
= \begin{cases}
 C_{n/2} \binom n{n/2},&\text{$n$ even}; \\ 
 0,&\text{$n$ odd}. \\ 
 \end{cases} 
\end{equation}
From~\eqref{ubq9lia} and~\eqref{cjrojsy}, we obtain
\begin{equation}
\sum_{k = 0}^n {( - 1)^k \binom{{n}}{k}\binom{{k}}{{\left\lfloor {k/2} \right\rfloor }}\binom{{n - k}}{{\left\lfloor {\left( {n - k} \right)/2} \right\rfloor }}}  =
\begin{cases}
 C_{n/2} \binom n{n/2},&\text{$n$ even}; \\ 
 0,&\text{$n$ odd}. \\ 
 \end{cases}   
\end{equation}

\end{example}
\begin{example}
Plugging $s_k$ and $\sigma_k$ from Example~\ref{ex.ur405iz} into~\eqref{eq.pt7l41w} yields, for $n$ a non-negative integer and $m$ a complex number that is not a negative integer,
\begin{align}
&\sum_{k = 0}^n {( - 1)^k \binom{{n}}{k}\frac{{m^2 H_{k + m} H_{n - k + m} }}{{\left( {k + m} \right)\left( {n - k + m} \right)}}}\nonumber\\
&\qquad  = \sum_{k = 0}^n {( - 1)^k \binom{{n}}{k}\binom{{k + m}}{k}^{ - 1} \binom{{n - k + m}}{{n - k}}^{ - 1} \left( {H_{k + m}  - H_k } \right)\left( {H_{n - k + m}  - H_{n - k} } \right)} ,
\end{align}
and, in particular,
\begin{equation}
\sum_{k = 0}^n {( - 1)^k \binom{{n}}{k}\frac{{H_{k + 1} H_{n - k + 1} }}{{\left( {k + 1} \right)\left( {n - k + 1} \right)}}}  = \sum_{k = 0}^n {\frac{{( - 1)^k }}{{\left( {k + 1} \right)^2 \left( {n - k + 1} \right)^2 }}\binom{{n}}{k}} .
\end{equation}
\end{example}
\begin{example}
From~\eqref{eq.a45516v}, with
\begin{equation*}
s_k  = \binom{{x}}{k}\binom{{y}}{k}^{ - 1} ,\quad\sigma _k  = \binom{{y - x}}{k}\binom{{y}}{k}^{ - 1}, 
\end{equation*}
identity~\eqref{eq.pt7l41w} gives
\begin{align}
&\sum_{k = 0}^n {( - 1)^k \binom{{n}}{k}\binom{{x}}{k}\binom{{x}}{{n - k}}\binom{y}{k}^{ - 1} \binom{{y}}{{n - k}}^{ - 1} } \nonumber\\
&\qquad = \sum_{k = 0}^n {( - 1)^k \binom{{n}}{k}\binom{{y - x}}{k}\binom{{y - x}}{{n - k}}\binom{{y}}{k}^{ - 1} \binom{{y}}{{n - k}}^{ - 1} } ,
\end{align}
which, upon setting $x=n$ and $y=n$, in turn, yields
\begin{align}\label{eq.hiez2vp}
&\sum_{k = 0}^n {( - 1)^k \binom{{n}}{k}^3 \binom{y}{k}^{ - 1} \binom{{y}}{{n - k}}^{ - 1} }\nonumber\\
&\qquad  = \sum_{k = 0}^n {( - 1)^k \binom{{n}}{k}\binom{{y - n}}{k}\binom{{y - n}}{{n - k}}\binom{{y}}{k}^{ - 1} \binom{{y}}{{n - k}}^{ - 1} } 
\end{align}
and
\begin{equation}\label{eq.u00e6qz}
\sum_{k = 0}^n {( - 1)^k \binom{{x}}{k}\binom{{x}}{{n - k}}\binom{{n}}{k}^{ - 1} }  = \sum_{k = 0}^n {( - 1)^k \binom{{n - x}}{k}\binom{{n - x}}{{n - k}}\binom{{n}}{k}^{ - 1} } .
\end{equation}
Setting $y=-1$ in~\eqref{eq.hiez2vp} gives, after some simplification,
\begin{equation*}
\sum_{k = 0}^n {( - 1)^k \binom{{n}}{k}^3 }  = \sum_{k = 0}^n {( - 1)^{n - k} \binom{{n}}{k}\binom{{n + k}}{k}\binom{{2n - k}}{{n - k}}} ,
\end{equation*}
the left hand side of which is known, Dixon's identity~\cite{dixon03}:
\begin{equation}\label{dixon}
\sum_{k = 0}^n {( - 1)^k \binom{{n}}{k}^3 }  
= \begin{cases}
 ( - 1)^{n/2} \binom{{n}}{{n/2}}\binom{{3n/2}}{n},&\text{if $n$ is even;} \\ 
 0,&\text{if $n$ is odd}.\\ 
 \end{cases} 
\end{equation}
We therefore obtain
\begin{equation}
\sum_{k = 0}^n {( - 1)^{n - k} \binom{{n}}{k}\binom{{n + k}}{k}\binom{{2n - k}}{{n - k}}}
= \begin{cases}
 ( - 1)^{n/2} \binom{{n}}{{n/2}}\binom{{3n/2}}{n},&\text{if $n$ is even;} \\ 
 0,&\text{if $n$ is odd}.\\ 
 \end{cases}  
\end{equation}
Similarly, setting $x=-1$ in~\eqref{eq.u00e6qz} gives
\begin{equation*}
\sum_{k = 0}^n {( - 1)^{n - k} \binom{{n}}{k}^{ - 1} }  = \sum_{k = 0}^n {(-1)^k\binom{{n}}{k}\frac{{\left( {n + 1} \right)^2 }}{{\left( {n - k + 1} \right)\left( {k + 1} \right)}}} ,
\end{equation*}
from which we get
\begin{equation*}
\sum_{k = 0}^n {( - 1)^k \binom{{n}}{k}\frac{1}{{\left( {n - k + 1} \right)\left( {k + 1} \right)}}}  = \frac{{1 + ( - 1)^n }}{{\left( {n + 1} \right)\left( {n + 2} \right)}},
\end{equation*}
upon using~\cite{sury04}:
\begin{equation*}
\sum_{k = 0}^n {( - 1)^k \binom{{n}}{k}^{ - 1} }  = \left( {1 + ( - 1)^n } \right)\frac{{n + 1}}{{n + 2}},
\end{equation*}
and hence,
\begin{equation*}
\sum_{k=0}^n{(-1)^k\binom nk\frac1{k+1}}=\frac1{n+1}.
\end{equation*}

\end{example}

A sequence $(s_k)$, $k=0,1,2,\ldots$, for which
\begin{equation*}
s_n  = \sum_{k = 0}^n {( - 1)^k \binom{{n}}{k}s_k },
\end{equation*}
for every non-negative integer $n$ will be called a self-inverse sequence, following Sun~\cite{sun01}.

A sequence $(s_k)$, $k=0,1,2,\ldots$, for which
\begin{equation*}
-s_n  = \sum_{k = 0}^n {( - 1)^k \binom{{n}}{k}s_k },
\end{equation*}
for every non-negative integer $n$ will be called an anti-self-inverse sequence.
Many self-inverse sequences are known (see e.g.~Sun~\cite{sun01}, Chen~\cite{chen07}, Boyadzhiev~\cite{boyadzhiev18}). Examples include $(L_k)$, $(kF_{k-1})$, $(\binom{x/2}k/\binom xk)$, $0\ne x\not\in\mathbb Z^+$, $(\binom{2k}k/2^{2k})$. Anti-self-inverse sequences include $(F_k)$ and $(H_k/(k+1))$. Self-inverse and anti-self-inverse sequences are discussed in more detail in Section~\ref{sec.invariant}.

\begin{corollary}\label{cor.rgyk46r}
Let $(s_k)$ and $(t_k)$, $k=0,1,2,\ldots$, be two sequences of complex numbers. Let $n$ be a non-negative integer. If both sequences are invariant or both are anti-self-inverse, then
\begin{equation}
\sum_{k = 0}^n {( - 1)^k \binom{{n}}{k}s_k t_{n - k} }  = 0,\text{ if $n$ is odd;}
\end{equation}
while if one of the sequences is invariant and the other is anti-self-inverse, then
\begin{equation}
\sum_{k = 0}^n {( - 1)^k \binom{{n}}{k}s_k t_{n - k} }  = 0,\text{ if $n$ is even.}
\end{equation}

\end{corollary}

\begin{remark}
The results stated in Corollary~\ref{cor.rgyk46r} were also obtained by Wang~\cite{wang05}.
\end{remark}

\begin{example}
If $n$ is a non-negative odd integer and $x$ is a non-zero complex number such that $x-n$ is not a negative integer, then
\begin{equation}
\sum_{k = 0}^n {( - 1)^k \binom{{n}}{k}\binom{{x/2}}{k}\binom{{x}}{k}^{ - 1} L_{n - k} }  = 0.
\end{equation}
\end{example}

\section{Identities involving binomial-transform pairs of the second kind}\label{sec.second}

\begin{lemma}
Let $\left\{(\bar a_k),(\bar\alpha_k)\right\}$, $k=0,1,2,\ldots$, be a binomial-transform pair of the second kind. Let $\mathcal M_x$ be a linear operator defined by $\mathcal M_x(x^j)=\bar a_j$ for every complex number $x$ and every non-negative integer~$j$. Then $\mathcal M_x((1+x)^j)=\bar\alpha_j$.
\end{lemma}
\begin{proof}
We have
\begin{align*}
\mathcal M_x\left( {\left( {1 + x} \right)^j } \right) = \mathcal M_x\left( {\sum_{k = 0}^j { \binom{{j}}{k}x^k } } \right) &= \sum_{k = 0}^j {\binom{{j}}{k}\mathcal M_x\left( {x^k } \right)} \\
& = \sum_{k = 0}^j {\binom{{j}}{k}\bar a_k }\\ 
&=\bar\alpha_j.
\end{align*}
\end{proof}

\begin{theorem}\label{thm.slg7tm1}
Let $n$ be a non-negative integer. If $\{(\bar s_k), (\bar\sigma_k)\}$ and $\{(\bar t_k), (\bar\tau_k)\}$, $k=0,1,2,\ldots$, are binomial-transform pairs of the second kind, then
\begin{equation}\label{main2}
\sum_{k = 0}^n {( - 1)^k \binom{{n}}{k}\bar s_k \bar t_{n-k} }  = \sum_{k = 0}^n {( - 1)^k \binom{{n}}{k}\bar\sigma _k \bar\tau _{n-k} } .
\end{equation}
\end{theorem}

\begin{proof}
Write~\eqref{eq.fltsd62} as
\begin{equation}\label{eq.ch50dhg}
\sum_{k = 0}^n {( - 1)^k \binom{{n}}{k} y^{n - k}x^k }  = \sum_{k = 0}^n {( - 1)^k \binom{{n}}{k} \left( {1 + y} \right)^{n - k} \left( {1 + x} \right)^k} .
\end{equation}
Let $(\bar t_j)$, $j=0,1,2,\ldots$, be a sequence of complex numbers. Let $\mathcal M_x(x^j)=\bar t_j$ .

Operate on both sides of~\eqref{eq.ch50dhg} with $\mathcal M_x$ to obtain
\begin{equation*}
\sum_{k = 0}^n {( - 1)^k \binom{{n}}{k} y^{n - k}\bar t_k }  = \sum_{k = 0}^n {( - 1)^k \binom{{n}}{k} \left( {1 + y} \right)^{n - k} \bar\tau _k} ,
\end{equation*}
where
\begin{equation*}
\bar\tau _k  = \mathcal M_x((1+x)^k)=\sum_{i = 0}^k {\binom{{k}}{i}\bar t_i }.
\end{equation*}
Thus,
\begin{equation}\label{eq.cr8jcr4}
\sum_{k = 0}^n {( - 1)^k \binom{{n}}{k}y^k\bar t_{n - k}  }  = \sum_{k = 0}^n {( - 1)^k \binom{{n}}{k}\left( {1 + y} \right)^k\bar\tau _{n - k} } .
\end{equation}
Let $(\bar s_j)$, $j=0,1,2,\ldots$, be a sequence of complex numbers. Let $\mathcal M_y(y^j)=\bar s_j$. The action of $\mathcal M_y$ on~\eqref{eq.cr8jcr4} produces
\begin{equation*}
\sum_{k = 0}^n {( - 1)^k \binom{{n}}{k}\bar s_k\bar t_{n - k}  }  = \sum_{k = 0}^n {( - 1)^k \binom{{n}}{k}\bar\sigma_k\bar\tau _{n - k} } ,
\end{equation*}
where
\begin{equation*}
\bar\sigma _k  = \mathcal M_y((1+y)^k)=\sum_{i = 0}^k { \binom{{k}}{i}\bar s_i },
\end{equation*}
and the proof is complete.
\end{proof}

\begin{example}
Consider the general recurrence relation for Bernoulli polynomials:
\begin{equation}\label{eq.dyudhfw}
\frac{{B_n \left( {x + y} \right)}}{{y^n }} = \sum_{k = 0}^n {\binom{{n}}{k}\frac{{B_k (x)}}{{y^k }}}.
\end{equation}
If we choose
\begin{equation*}
\bar s_k  = \frac{{B_k (x)}}{{y^k }},\quad\bar \sigma _k  = \frac{{B_k \left( {x + y} \right)}}{{y^k }},
\end{equation*}
in~\eqref{main2}, we obtain
\begin{equation}\label{eq.bf4tiok}
\sum_{k = 0}^n {( - 1)^k \binom{{n}}{k}y^k B_{n - k} (x)\bar t_k }  = \sum_{k = 0}^n {( - 1)^k \binom{{n}}{k}y^k B_{n - k} \left( {x + y} \right)\bar \tau _k } ,
\end{equation}
which is valid for complex variables $x$ and $y$ and any binomial transform pair $\{(\bar s_k),(\bar\sigma_k)\}$ of the second kind.

Many results can be derived from~\eqref{eq.bf4tiok}. To begin with, we have the following polynomial identity:
\begin{equation}
\sum_{k = 0}^n {( - 1)^k \binom{{n}}{k}y^k B_{n - k} (x)t^k }  = \sum_{k = 0}^n {( - 1)^k \binom{{n}}{k}y^k B_{n - k} \left( {x + y} \right)\left( {1 + t} \right)^k } ,
\end{equation}
which holds for complex numbers $x$, $y$ and $t$.

Choosing
\begin{equation*}
\bar t_k  = \frac{{B_k (z)}}{{w^k }},\quad\bar \tau _k  = \frac{{B_k \left( {z + w} \right)}}{{w^k }},
\end{equation*}
in~\eqref{eq.bf4tiok} leads to
\begin{align}
&\sum_{k = 0}^n {( - 1)^k \binom{{n}}{k}\left( {y/w} \right)^k B_{n - k} (x)B_k (z)}\nonumber\\
&\qquad  = \sum_{k = 0}^n {( - 1)^k \binom{{n}}{k}\left( {y/w} \right)^k B_{n - k} \left( {x + y} \right)B_k \left( {z + w} \right)} ,
\end{align}
giving, in particular,
\begin{align}
\sum_{k = 0}^n {( - 1)^k \binom{{n}}{k}B_{n - k} (x)B_k (z)}  = \sum_{k = 0}^n {( - 1)^k \binom{{n}}{k}B_{n - k} \left( {x + w} \right)B_k \left( {z + w} \right)} ,
\end{align}
and
\begin{equation}\label{eq.zc4ufo6}
\sum_{k = 0}^n {( - 1)^k \binom{{n}}{k}\left( {y/w} \right)^k B_{n - k} \left( {y} \right)B_k \left( {w} \right)}=\sum_{k = 0}^n {( - 1)^k \binom{{n}}{k}\left( {y/w} \right)^k B_{n - k}B_k}.
\end{equation}

\end{example}

\begin{proposition}
If $n$ is a positive odd integer and $y$ and $w$ are complex numbers, then
\begin{equation}
\sum_{k = 0}^n {( - 1)^k \binom{{n}}{k}\left( {\frac{y}{w}} \right)^k B_{n - k} (y)B_k (w)}  = \frac{{ny}}{{2w}}\left( {1 - \left( {\frac{y}{w}} \right)^{n - 2} } \right)B_{n - 1} .
\end{equation}

\end{proposition}

\begin{proof}
This is a consequence of the fact that the right-hand side of~\eqref{eq.zc4ufo6} can be written as
\begin{equation*}
\sum_{k = 0}^{\left\lfloor {n/2} \right\rfloor } {\binom{{n}}{{2k}}(y/w)^{2k} B_{n - 2k} B_{2k} }  - \sum_{k = 1}^{\left\lceil {n/2} \right\rceil } {\binom{{n}}{{2k - 1}}(y/w)^{2k - 1} B_{n - 2k + 1} B_{2k - 1} }, 
\end{equation*}
and the fact that $B_j$ is zero for odd $j$ greater than unity.
\end{proof}

\begin{example}
Using the binomial-transform pairs
\begin{equation}\label{munb79m}
\bar s_k  = \binom{{x}}{k},\quad\bar \sigma_k  = \binom{{x + k}}{k},
\end{equation}
obtained from~\cite[Equation (10.10d)]{boyadzhiev18}:
\begin{equation}
\sum_{k = 0}^n {\binom{{n}}{k}\binom{{x}}{k}}  = \binom{{n + x}}{n}\label{eq.s21u1g9}
\end{equation}
and
\begin{equation}
\bar t_k  = B_k ,\quad\bar \tau _k  = ( - 1)^k B_k,
\end{equation}
found from the recurrence relation
\begin{equation*}
\sum_{k = 0}^n {\binom nkB_k }  = ( - 1)^n B_n,
\end{equation*}
in~\eqref{main2} gives
\begin{equation}\label{zc3ejk3}
\sum_{k = 0}^n {( - 1)^{n - k} \binom{{n}}{k}\binom{{x}}{k}B_{n - k} }  = \sum_{k = 0}^n {\binom{{n}}{k}\binom{{x + k}}{k}B_{n - k} } ,
\end{equation}
which holds for every non-negative integer $n$ and every complex number $x$.
\end{example}

\begin{proposition}
If $n$ is a non-negative integer and $x$ is a complex number such that $x-n$ is not a negative integer, then
\begin{equation}\label{numg9mq}
\sum_{k = 0}^n {( - 1)^{n - k} \binom{{n}}{k}\binom{{x}}{k}\left( {H_x  - H_{x - k} } \right)B_{n - k} }  = \sum_{k = 0}^n {\binom{{n}}{k}\binom{{x + k}}{k}\left( {H_{x + k}  - H_x } \right)B_{n - k} } .
\end{equation}
\end{proposition}
In particular,
\begin{equation}\label{s8xzyeq}
\sum_{k = 0}^n {\binom{{n}}{k}\binom{{2k}}{k}2^{ - 2k} O_k B_{n - k} }  = 0,\quad\text{$n$ an even integer}.
\end{equation}
\begin{proof}
Differentiate~\eqref{zc3ejk3} wth respect to $x$, using
\begin{equation*}
\frac{d}{{dx}}\binom{{x + a}}{k} = \binom{{x + a}}{k}\left( {H_{x + a}  - H_{x + a - k} } \right).
\end{equation*}
In deriving~\eqref{s8xzyeq}, we set $x=-1/2$ in~\eqref{numg9mq} and used
\begin{equation}\label{kblsbvx}
H_{k-1/2}=2O_k-2\ln 2,
\end{equation}
\begin{equation}\label{duwnvoj}
\binom{{ - 1/2}}{k} = ( - 1)^k \binom{{2k}}{k}2^{ - 2k} ,
\end{equation}
and
\begin{equation}\label{zx8842q}
\binom{{k - 1/2}}{k} = \binom{{2k}}{k}2^{ - 2k} .
\end{equation}
\end{proof}
\begin{corollary}
Let $n$ be a non-negative integer. If $\{(\bar s_k), (\bar\sigma_k)\}$ is a binomial-transform pair of the second kind, then
\begin{equation}\label{s12ccud}
\sum_{k = 0}^n {( - 1)^k \binom{{n}}{k}\bar s_k \bar s_{n-k} }  = \sum_{k = 0}^n {( - 1)^k \binom{{n}}{k}\bar\sigma _k \bar\sigma _{n-k} } .
\end{equation}
\end{corollary}

\begin{example}
Using, in~\eqref{s12ccud}, the $\bar s_k$ and $\bar\sigma_k$ given in~\eqref{munb79m}, we have
\begin{equation}
\sum_{k = 0}^n {( - 1)^k \binom{{n}}{k}\binom{{x}}{k}\binom{{x}}{{n - k}}}  = \sum_{k = 0}^n {( - 1)^k \binom{{n}}{k}\binom{{x + k}}{k}\binom{{x + n - k}}{{n - k}}} ;
\end{equation}
which holds for every non-negative integer $n$ and every complex number $x$.
\end{example}

\begin{example}
From~\cite[Equation (9.62)]{boyadzhiev18}:
\begin{equation}
\sum_{k = 0}^n {\binom{{n}}{k}\binom{{m}}{k}H_k }  = \binom{{n + m}}{m}\left( {H_m  + H_n  - H_{m + n} } \right)\label{eq.mwv0ueq},
\end{equation}

we identify the following binomial-transform pair of the second kind:
\begin{equation}\label{fkgvrgi}
\bar s_k  = \binom{{m}}{k}H_k ,\quad\bar \sigma _k  = \binom{{k + m}}{m}\left( {H_m  + H_k  - H_{k + m} } \right),
\end{equation}
which when used in~\eqref{s12ccud} gives
\begin{align}
&\sum_{k = 0}^n {( - 1)^k \binom{{n}}{k}\binom{{m}}{k}\binom{{m}}{{n - k}}H_k H_{n - k} }\nonumber\\
&\qquad  = \sum_{k = 0}^n {( - 1)^k \binom{{n}}{k}\binom{{k + m}}{m}\binom{{n - k + m}}{m}\left( {H_m  + H_k  - H_{k + m} } \right)\left( {H_m  + H_{n - k}  - H_{n - k + m} } \right)} ,
\end{align}
which is valid for every non-negative integer $n$ and every complex number $m$ that is not a negative integer.
\end{example}
\begin{theorem}\label{thm.njtayqa}
Let $n$ be a non-negative integer. If $\{(\bar s_k), (\bar\sigma_k)\}$ and $\{(\bar t_k), (\bar\tau_k)\}$, $k=0,1,2,\ldots$, are binomial-transform pairs of the second kind, then
\begin{equation}\label{o5suprg}
\sum_{k = 0}^n {\binom{{n}}{k}\bar s_k \bar\tau_{n-k} }  = \sum_{k = 0}^n {\binom{{n}}{k}\bar\sigma _k \bar t _{n-k} } .
\end{equation}
\end{theorem}

\begin{proof}
Consider the following variation on~\eqref{eq.ch50dhg}:
\begin{equation}\label{eq.te268n7}
\sum_{k = 0}^n {\binom{{n}}{k}y^{n - k} \left( {1 + x} \right)^k }  = \sum_{k = 0}^n {\binom{{n}}{k}\left( {1 + y} \right)^{n - k} x^k } ,
\end{equation}
and proceed as in the proof of Theorem~\ref{thm.slg7tm1}.
\end{proof}

\begin{example}
Using the binomial-transform pairs of the second kind
\begin{equation*}
\bar s_k  = \binom{{m}}{k}H_k ,\quad\bar \sigma _k  = \binom{{k + m}}{m}\left( {H_m  + H_k  - H_{k + m} } \right), \text{ Equation~\eqref{fkgvrgi}},
\end{equation*}
and
\begin{equation*}
\bar t_k  = \binom{{x}}{k},\quad\bar \tau_k  = \binom{{x + k}}{k},\text{ Equation~\eqref{munb79m} relabelled},
\end{equation*}
in~\eqref{o5suprg} yields
\begin{equation}\label{zrnqrxn}
\sum_{k = 0}^n {\binom{{n}}{k}\binom{{m}}{k}\binom{{x + n - k}}{{n - k}}H_k }  = \sum_{k = 0}^n {\binom{{n}}{k}\binom{{x}}{{n - k}}\binom{{k + m}}{m}\left( {H_m  + H_k  - H_{k + m} } \right)} .
\end{equation}
\end{example}

\begin{example}
Consider the following identity~\cite{trif00}:
\begin{equation}\label{foz9j01}
\sum_{k = 0}^n {\binom{{n}}{k}\binom{{m}}{{p + k}}^{ - 1} }  = \frac{{m + 1}}{{m - n + 1}}\binom{{m - n}}{p}^{ - 1}, 
\end{equation}
which holds for every non-negative integer $n$ and complex numbers $m$ and $p$ for which $m-n-p$ is not a negative integer.

Choosing
\begin{equation}
\bar s_k  = \binom{{m}}{{p + k}}^{ - 1} ,\quad\bar \sigma _k  = \frac{{m + 1}}{{m - k + 1}}\binom{{m - k}}{p}^{ - 1}, 
\end{equation}
in~\eqref{o5suprg} gives
\begin{equation*}
\sum_{k = 0}^n {\binom{{n}}{k}\binom{{m}}{{p + k}}^{ - 1} \bar \tau _{n - k} }  = \sum_{k = 0}^n {\binom{{n}}{k}\frac{{m + 1}}{{m - k + 1}}\binom{{m - k}}{p}^{ - 1} t_{n - k} } ,
\end{equation*}
which can also be written as
\begin{equation}\label{a2qzajx}
\sum_{k = 0}^n {\binom{{n}}{k}\binom{{m}}{{p + n - k}}^{ - 1} \bar \tau _k }  = \sum_{k = 0}^n {\binom{{n}}{k}\frac{{m + 1}}{{m - n + k + 1}}\binom{{m - n + k}}{p}^{ - 1} t_k } .
\end{equation}
\end{example}

\begin{example}
Differentiating~\eqref{foz9j01} with respect to $m$ using
\begin{equation*}
\frac{d}{{dm}}\binom{{m}}{{p + k}}^{ - 1}  = \binom{{m}}{{p + k}}^{ - 1} \left( {H_{m - p - k}  - H_m } \right)
\end{equation*}
gives
\begin{align}\label{cd4s9t9}
&\sum_{k = 0}^n {\binom{{n}}{k}\binom{{m}}{{p + k}}^{ - 1} \left( {H_{m - p - k}  - H_m } \right)}\nonumber\\ 
&\qquad = \frac{{ - n}}{{\left( {m - n + 1} \right)^2 }}\binom{{m - n}}{p}^{ - 1}  - \frac{{m + 1}}{{m - n + 1}}\binom{{m - n}}{p}^{-1}\left( {H_{m - n}  - H_{m - p - n} } \right),
\end{align}
for $m$ and $p$ complex numbers such that $m-n-p\ne\mathbb Z^{-}$, which at $p=0$ gives
\begin{equation}\label{evuuti7}
\sum_{k = 0}^n {\binom{{n}}{k}\binom{{m}}{k}^{ - 1} \left( {H_{m - k}  - H_m } \right)}  = \frac{{ - n}}{{\left( {m - n + 1} \right)^2 }},
\end{equation}
which yields
\begin{equation*}
\sum_{k = 0}^n {\binom{{n}}{k}\binom{{m}}{k}^{ - 1} H_{m - k} }  = \frac{{m + 1}}{{m - n + 1}}H_m  - \frac{n}{{\left( {m - n + 1} \right)^2 }},
\end{equation*}
of which the well-known identity
\begin{equation}\label{qqzbh28}
\sum_{k = 0}^n {H_k }  = \left( {n + 1} \right)H_n  - n,
\end{equation}
is a particular case.

\end{example}

\begin{proposition}
If $n$ is a non-negative integer, then
\begin{equation}
\sum_{k = 0}^n {( - 1)^k \binom{{n}}{k}\binom{{2k}}{k}^{ - 1} 2^{2k} O_k }  = \frac{{ - 2n}}{{\left( {2n - 1} \right)^2 }}.
\end{equation}

\end{proposition}

\begin{proof}
Set $m=-1/2$ in~\eqref{evuuti7} and use~\eqref{kblsbvx} and~\eqref{duwnvoj}.
\end{proof}

\begin{example}
Differentiating~\eqref{foz9j01} with respect to $p$ using
\begin{equation*}
\frac{d}{{dp}}\binom{{m}}{{p + k}}^{ - 1}  = \binom{{m}}{{p + k}}^{ - 1} \left( {H_{p + k}  - H_{m - p - k} } \right)
\end{equation*}
yields
\begin{equation}\label{mvuo316}
\sum_{k = 0}^n {\binom{{n}}{k}\binom{{m}}{{p + k}}^{ - 1} \left( {H_{m - p - k}  - H_{p + k} } \right)}  = \frac{{m + 1}}{{m - n + 1}}\binom{{m - n}}{p}^{ - 1} \left( {H_{m - p - n}  - H_p } \right),
\end{equation}
which can also be written as
\begin{align}
&\sum_{k = 0}^n {\binom{{n}}{k}\binom{{m}}{{p + n - k}}^{ - 1} \left( {H_{m - p - n + k}  - H_{p + n - k} } \right)}\nonumber\\
&\qquad  = \frac{{m + 1}}{{m - n + 1}}\binom{{m - n}}{p}^{ - 1} \left( {H_{m - p - n}  - H_p } \right).
\end{align}
Using
\begin{equation*}
\bar s_k  = \binom{{m}}{{p + n - k}}^{ - 1} \left( {H_{m - p - n + k}  - H_{p + n - k} } \right),\quad\bar \sigma _k  = \frac{{m + 1}}{{m - k + 1}}\binom{{m - k}}{p}\left( {H_{m - p - k}  - H_p } \right),
\end{equation*}
in~\eqref{o5suprg} leads to the following identity:
\begin{align}\label{v4unj46}
&\sum_{k = 0}^n {\binom{{n}}{k}\binom{{m}}{{p + n - k}}^{ - 1} \left( {H_{m - p - n + k}  - H_{p + n - k} } \right)\bar \tau _k }\nonumber\\
&\qquad  = \sum_{k = 0}^n {\binom{{n}}{k}\binom{{m - n + k}}{p}^{ - 1} \frac{{m + 1}}{{m - n + k + 1}}\left( {H_{m - p - n + k}  - H_p } \right)\bar t_k } ,
\end{align}
which holds for an arbitrary binomial-transform pair of the second kind and complex numbers $m$ and $p$ for which $m-n-p$ is not a negative integer.
\end{example}

\begin{proposition}
If $n$ is a non-negative integer and $\{(\bar t_k),(\bar\tau_k)\}$, $k=0,1,2,\ldots$, is a binomial-transform pair of the second kind, then
\begin{equation}
\sum_{k = 0}^n {\binom{{n + 1}}{{k + 1}}H_k \bar t_k }  = \sum_{k = 0}^n {\left( {H_k  - H_{n - k} } \right)\bar\tau_k} .
\end{equation}
In particular,
\begin{equation}
\sum_{k = 0}^n {\binom{{n + 1}}{{k + 1}}H_k B_k }  = \sum_{k = 0}^n {( - 1)^k \left( {H_k  - H_{n - k} } \right)B_k } 
\end{equation}
and
\begin{equation}
\sum_{k = 0}^n {\binom{{n + 1}}{{k + 1}}H_k x^k }  = \sum_{k = 0}^n {\left( {H_k  - H_{n - k} } \right)\left( {1 + x} \right)^k } ,
\end{equation}
for every complex number $x$.
\end{proposition}
\begin{proof}
Set $m=n$ and $p=0$ in~\eqref{v4unj46}.
\end{proof}

By considering the binomial-transform pair of the second kind $\{(1),(2^k)\}$, $k=0,1,2,\ldots$ and making use of Theorem~\ref{thm.njtayqa}, we obtain the next result expressing the binomial transform of a binomial transform.
\begin{theorem}
Let $n$ be a non-negative integer. If $\{(\bar s_k), (\bar\sigma_k)\}$ is a binomial-transform pair of the second kind, then
\begin{equation}\label{pop9ybt}
\sum_{k = 0}^n {\binom{{n}}{k}\bar \sigma _k }  = \sum_{k = 0}^n {\binom{{n}}{k}2^{n - k} \bar s_k } .
\end{equation}
\end{theorem}

\begin{remark}
Identity~\eqref{pop9ybt} corresponds to Example 4 of Boyadzhiev~\cite{boyadzhiev16}.
\end{remark}

\begin{example}
With~\cite[Equation (10.10a)]{boyadzhiev18} 
\begin{equation}\label{b0s5n16}
\sum_{k = 0}^n {\binom{{n}}{k}\binom{{x}}{{k + z}}}  = \binom{{n + x}}{{n + z}},
\end{equation}
in mind, choosing
\begin{equation}
\bar s_k  = \binom{{x}}{{k + z}},\quad\bar \sigma _k  = \binom{{k + x}}{{k + z}},
\end{equation}
in~\eqref{pop9ybt} yields
\begin{equation}
\sum_{k = 0}^n {\binom{{n}}{k}\binom{{k + x}}{{k + z}}}  = \sum_{k = 0}^n {\binom{{n}}{k}\binom{{x}}{{k + z}}2^{n - k} } .
\end{equation}
for $n$ a non-negative integer and $x$ and $z$ complex numbers.
\end{example}

\begin{proposition}
If $j$ and $n$ are non-negative integers, then
\begin{equation}\label{s4jiizc}
\sum_{k = 0}^n {( - 1)^{n-k} \binom{n-j}{k-j}t_{n-k} }  = \tau _{n - j} .
\end{equation}

\end{proposition}
\begin{proof}
Consider the identity
\begin{equation}\label{siusv41}
\sum_{k = 0}^n {( - 1)^k \binom{{n}}{k}\binom{{k}}{j}}  = ( - 1)^j \delta _{nj} ,
\end{equation}
and use $s_k=\binom kj$ and $\sigma_k=(-1)^j\delta_{kj}$ in~\eqref{main1}.
\end{proof}

\section{Identities involving mixed binomial transform pairs}\label{sec.mixed}

\begin{theorem}
Let $n$ be a non-negative integer. If $\{(s_k), (\sigma_k)\}$ is a binomial transform pair of the first kind and $\{(\bar t_k), (\bar\tau_k)\}$, $k=0,1,2,\ldots$, is a binomial-transform pair of the second kind, then
\begin{equation}\label{x1h685v}
\sum_{k = 0}^n {\binom{{n}}{k}s_k \bar t_{n-k} }  = \sum_{k = 0}^n {(-1)^k\binom{{n}}{k} \sigma_k \bar \tau _{n-k} } .
\end{equation}
\end{theorem}

\begin{proof}
Let $\mathcal M_x(x^k)=\bar t_k$ and $\mathcal L_y(y^k)=s_k$.

Act on the identity
\begin{equation*}
\sum_{k = 0}^n {\binom{{n}}{k}y^{n - k} x^k }  = \sum_{k = 0}^n {( - 1)^{n - k} \binom{{n}}{k}\left( {1 - y} \right)^{n - k} \left( {1 + x} \right)^k } ,
\end{equation*}
with $\mathcal M_x$ and $\mathcal L_y$, in succession.
\end{proof}

\begin{example}
Use of the binomial-transform pair of the first kind $\{(F_k),(-F_k)\}$, and the binomial-transform pair of the second kind $\{(B_k),((-1)^kB_k)\}$, $k=0,1,2,\ldots$, in~\eqref{x1h685v} gives 
\begin{equation}
\sum_{k = 0}^n {\binom{{n}}{k}F_k B_{n - k} }  = 0,
\end{equation}
for $n$ an even integer.

Similarly $\{(L_k),(-L_k)\}$ and $\{(B_k),((-1)^kB_k)\}$, $k=0,1,2,\ldots$, in~\eqref{x1h685v} yields
\begin{equation}
\sum_{k = 0}^n {\binom{{n}}{k}L_k B_{n - k} }  = 0,
\end{equation}
for $n$ an odd integer.

\end{example}

\begin{example}
Use of the following binomial-transform pair (see Example~\ref{ex.omtqrl4}):
\begin{equation*}
s_k  = H_{k + m} ,\quad\sigma _k  =  \frac{{\delta _{k0}(1+H_m)-1 }}{{k + \delta _{k0} }}\binom{{k + m}}{m}^{ - 1},  
\end{equation*}
in~\eqref{x1h685v} leads to the following identity
\begin{equation}\label{jy2d3um}
\sum_{k = 0}^n {\binom{{n}}{k}H_{k + m} \bar t_{n - k} }  = H_m \bar \tau _n  - \sum_{k = 1}^n {( - 1)^k \binom{{n}}{k}\binom{{k + m}}{m}^{ - 1}\frac1k\, \bar \tau _{n - k} } ,
\end{equation}
which is valid for every binomial-transform pair of the second kind $\{(\bar t_k),(\bar\tau_k)\}$, $k=0,1,2,\ldots$ and every complex number $m$ that is not a negative integer.

In particular, the following polynomial identity holds:
\begin{equation*}
\sum_{k = 0}^n {\binom{{n}}{k}H_{k + m} t^{n - k} }  = (1+t)^n H_m  - \sum_{k = 1}^n {( - 1)^k \binom{{n}}{k}\binom{{k + m}}{m}^{ - 1}\frac{(1+t)^{n - k}}k },
\end{equation*}
and can be recast as stated in proposition~\ref{prop.ai6n4y9} by writing $1/t$ for $t$.
\end{example}

\begin{proposition}\label{prop.ai6n4y9}
If $n$ is a non-negative integer, $m$ is a complex number that is not a negative integer and $t$ is a complex variable, then
\begin{equation}\label{t7tu7xu}
\sum_{k = 0}^n {\binom{{n}}{k}H_{k + m} t^k }  = \left( {1 + t} \right)^n H_m  + \sum_{k = 1}^n {( - 1)^{k - 1} \binom{{n}}{k}\binom{{k + m}}{m}^{ - 1} \frac{1}{k}\,t^k \left( {1 + t} \right)^{n - k} } .
\end{equation}
\end{proposition}
In particular, evaluation at $m=0$ and $m=-1/2$, respectively, gives
\begin{equation}
\sum_{k = 0}^n {\binom{{n}}{k}H_k t^k }  =\sum_{k = 1}^n {( - 1)^{k - 1} \binom{{n}}{k}\frac{1}{k}\,t^k \left( {1 + t} \right)^{n - k} } 
\end{equation}
and
\begin{equation}\label{mdlphw5}
\sum_{k = 0}^n {\binom{{n}}{k}O_k t^k }  =\sum_{k = 1}^n {( - 1)^{k - 1} \binom{{n}}{k}2^{2k-1}\binom{2k}k^{-1}\frac{1}{k}\,t^k \left( {1 + t} \right)^{n - k} } ,
\end{equation}
with the special values
\begin{equation}
\sum_{k = 0}^n {\binom{{n}}{k}H_k }  = \sum_{k = 1}^n {( - 1)^{k - 1} 2^{n - k} \binom{{n}}{k}\frac{1}{k}} 
\end{equation}
and
\begin{equation}
\sum_{k = 0}^n {\binom{{n}}{k}O_k }  = \sum_{k = 1}^n {( - 1)^{k - 1} 2^{n + k - 1} \binom{{n}}{k}\binom{2k}k^{-1}\frac{1}{k}}. 
\end{equation}
In deriving~\eqref{mdlphw5}, we used~\eqref{kblsbvx} and~\eqref{zx8842q}.

\begin{corollary}
Let $n$ be a non-negative integer. If $\{(s_k), (\sigma_k)\}$ is a binomial transform pair of the first kind and $\{(\bar t_k), (\bar\tau_k)\}$, $k=0,1,2,\ldots$, is a binomial-transform pair of the second kind, then
\begin{equation}
\sum_{k = 0}^n {\binom{{n}}{k}\sigma_k \bar t_{n-k} }  = \sum_{k = 0}^n {(-1)^k\binom{{n}}{k} s_k \bar \tau _{n-k} } .
\end{equation}
\end{corollary}

\section{Binomial transform of products}\label{sec.products}
Consider a binomial-transform pair of the second kind, say $\{(\bar t_k), (\bar\tau_k)\}$ and an arbitrary sequence of complex numbers $(s_k)$, $k=0,1,2,\ldots$. In this section we evaluate the binomial transform of the sequence $(s_k\bar t_k)$ in terms of $(s_k)$ and $(\bar\tau_k)$, by establishing the result stated in the next theorem.
\begin{theorem}
Let $n$ be a non-negative integer. If $\{(\bar t_k), (\bar\tau_k)\}$, $k=0,1,2,\ldots$, is a binomial-transform pair of the second kind and $(s_k)$, $k=0,1,2,\ldots$, is a sequence of complex numbers, then
\begin{equation}\label{wjz3q4q}
\sum_{k = 0}^n {\binom{{n}}{k}s_k \bar t_k }  = \sum_{k = 0}^n {\binom{{n}}{k}\bar \tau _k \sum_{j = 0}^{n - k} {(-1)^j\binom{{n - k}}{j}s_{k+j} } } .
\end{equation}
\end{theorem}

\begin{proof}
The elementary identity
\begin{equation}\label{jsam0sy}
\sum_{k = 0}^n {\binom{{n}}{k}x^k y^k }  = \sum_{k = 0}^n {\binom{{n}}{k}\left( {1 + y} \right)^k \sum_{j = 0}^{n - k} {( - 1)^j \binom{{n - k}}{j}x^{k + j} } }
\end{equation}
can be checked using the binomial theorem. Operating on both sides of~\eqref{jsam0sy} with $\mathcal M_x\mathcal M_y$ or $\mathcal L_x\mathcal M_y$ yields~\eqref{wjz3q4q}.
\end{proof}

\begin{remark}
In terms of a binomial-transform pair of the first kind, $\{( t_k), (\tau_k)\}$, $k=0,1,2,\ldots$, identity~\eqref{wjz3q4q} reads
\begin{equation}\label{rb8niyz}
\sum_{k = 0}^n {\binom{{n}}{k} s_k  t_k }  = \sum_{k = 0}^n {(-1)^k\binom{{n}}{k} \tau _k \sum_{j = 0}^{n - k} {\binom{{n - k}}{j} s_{k+j} } } .
\end{equation}

\end{remark}

\begin{theorem}
Let $n$ be a non-negative integer. If $\{(\bar t_k), (\bar\tau_k)\}$, $k=0,1,2,\ldots$, is a binomial-transform pair of the second kind and $(s_k)$, $k=0,1,2,\ldots$, is a sequence of complex numbers, then
\begin{equation}\label{rrn382p}
\sum_{k = 0}^n {(-1)^k\binom nks_k \bar t_k }  = \sum_{k = 0}^n {(-1)^k\binom nk\bar \tau _k \sum_{j = 0}^{n - k} {\binom{{n - k}}{j}s_{k+j} } } .
\end{equation}
\end{theorem}

\begin{proof}
Replace $s_k$ with $(-1)^ks_k$ in~\eqref{wjz3q4q} since $(s_k)$ is an arbitrary sequence.
\end{proof}

\begin{remark}
Written for a binomial-transform pair of the first kind, identity~\eqref{rrn382p} becomes
\begin{equation}\label{i752h6k}
\sum_{k = 0}^n {(-1)^k\binom nks_kt_k }  = \sum_{k = 0}^n {\binom nk\tau _k \sum_{j = 0}^{n - k} {(-1)^j\binom{{n - k}}{j}s_{k+j} } } .
\end{equation}

\end{remark}

\begin{remark}
Boyadzhiev~\cite{boyadzhiev16} studied binomial transform of products. His results connect binomial-transform pairs and are expressed in terms of the forward difference operator.
\end{remark}

\begin{example}
Setting $s_k=G_k$ in~\eqref{wjz3q4q} and making use of~\eqref{q96hz7w} gives
\begin{equation}
\sum_{k = 0}^n {\binom{{n}}{k}G_k \bar t_k }  = \sum_{k = 0}^n {( - 1)^k \binom{{n}}{k}\bar \tau _k \left( {G_0 L_{n - 2k}  - G_{n - 2k} } \right)} .
\end{equation}
The choice $s_k=G_k$ in~\eqref{rrn382p} gives
\begin{equation}\label{r638l94}
\sum_{k = 0}^n {( - 1)^k \binom{{n}}{k}G_k \bar t_k }  = \sum_{k = 0}^n {( - 1)^k \binom{{n}}{k}\bar \tau _kG_{2n - k}  } ;
\end{equation}
since~\cite[Equation (49)]{vajda}:
\begin{equation*}
\sum_{j = 0}^{n - k} {\binom{{n - k}}{j}G_{j + k} }  = G_{2n - k} .
\end{equation*}
Identity~\eqref{r638l94} corresponds to, but is slightly more general than, Boyadzhiev's identity~\cite[Equation (1.11)]{boyadzhiev16}.
\end{example}

\section{Symmetry properties and generalizations}\label{symmetry}
In this section we prove two symmetry properties and, based upon them, derive two generalizations of~\eqref{main1}.
\begin{theorem}
Let $\{(s_k),(\sigma_k)\}$, $k=0,1,2,\ldots$, be a binomial-transform pair of the first kind. If $m$ and $n$ are non-negative integers, then
\begin{equation}\label{hl4i644}
\sum_{k = 0}^n {( - 1)^k \binom{{n}}{k}s_{k + m} }  = \sum_{k = 0}^m {( - 1)^k \binom{{m}}{k}\sigma _{k + n} } .
\end{equation}

\end{theorem}
\begin{proof}
Let $\mathcal L_x(x^k)=s_k$ for every complex number $x$ and every non-negative integer $k$. Operate on both sides of the following variation on the binomial theorem with $\mathcal L_x$:
\begin{equation*}
\sum_{k = 0}^n {( - 1)^k \binom{{n}}{k}x^{k + m} }  = \sum_{k = 0}^m {( - 1)^k \binom{{m}}{k}\left( {1 - x} \right)^{_{k + n} } } .
\end{equation*}
\end{proof}

\begin{remark}
The version of~\eqref{hl4i644} involving a binomial-transform pair of the second kind, namely,
\begin{equation}\label{uwonf8d}
\sum_{k = 0}^n {\binom{{n}}{k}\bar s_{k + m} }  = \sum_{k = 0}^m {( - 1)^{k + m} \binom{{m}}{k}\bar \sigma _{k + n} } ,
\end{equation}
corresponds to Chen~\cite[Proof of Theorem 3.1, generalized in Theorem 3.2]{chen07}. See also Gould and Quaintance~\cite{gould14} for an alternative proof of~\eqref{uwonf8d}.
\end{remark}
We now present a generalization of~\eqref{main1}.
\begin{theorem}\label{thm.hr1s9kh}
Let $m$, $n$ and $r$ be non-negative integers. If $\{(s_k), (\sigma_k)\}$ and $\{(t_k), (\tau_k)\}$, $k=0,1,2,\ldots$, are binomial-transform pairs of the first kind, then
\begin{align}\label{uq1kbw0}
&\sum_{k = 0}^n {( - 1)^k \binom{{n}}{k}s_{n - k + m} \sum_{q = 0}^r {( - 1)^q \binom{{r}}{q}\tau _{k + q} } }\nonumber\\
&\qquad  = \sum_{k = 0}^n {( - 1)^k \binom{{n}}{k}t_{n - k + r} \sum_{p = 0}^m {( - 1)^p \binom{{m}}{p}\sigma _{k + p} } } .
\end{align}

\end{theorem}

\begin{proof}
From~\eqref{hl4i644} it is clear that $s_{k+m}$ and $\sum_{p = 0}^m {( - 1)^p \binom{{m}}{p}\sigma _{k + p} } $ are a binomial-transform pair of the first kind. That is, if  $\{(s_k), (\sigma_k)\}$, $k=0,1,2,\ldots$, is a binomial-transform pair of the first kind, then so is $\{(a_k), (\alpha_k)\}$, $k=0,1,2,\ldots$,  where
\begin{equation}\label{w95ejl1}
a_k  = s_{k + m} \text{ and }\alpha _k  = \sum_{p = 0}^m {( - 1)^p \binom{{m}}{p}\sigma _{k + p} }. 
\end{equation}
Similarly, $\{(b_k), (\beta_k)\}$, $k=0,1,2,\ldots$ is a binomial transform pair of the first kind, where
\begin{equation}\label{h1u6s0i}
b_k  = t_{k + r} \text{ and }\beta _k  = \sum_{q = 0}^r {( - 1)^q \binom{{r}}{q}\tau _{k + q} } .
\end{equation}
Now write~\eqref{main1} as
\begin{equation}\label{jv6g76k}
\sum_{k = 0}^n {( - 1)^k \binom nka_{n-k} b_k }  = \sum_{k = 0}^n {( - 1)^k \binom{{n}}{k}\alpha _k \beta _{n-k} } ,
\end{equation}
and substitute~\eqref{w95ejl1} and~\eqref{h1u6s0i} to obtain
\begin{align}
&\sum_{k = 0}^n {( - 1)^{n - k} \binom{{n}}{k}s_{k + m} t_{n - k + r} }\nonumber\\
&  = \sum_{k = 0}^n {( - 1)^k \binom{{n}}{k}\sum_{p = 0}^m {( - 1)^p \binom{{m}}{p}\sigma _{k + p} \sum_{q = 0}^r {( - 1)^q \binom{{r}}{q}\tau _{n - k + q} } } } ,
\end{align}
which can also be written as~\eqref{uq1kbw0}.
\end{proof}
We illustrate Theorem~\ref{thm.hr1s9kh} with the result stated in Proposition~\ref{prop.ofdjag5}.
\begin{proposition}\label{prop.ofdjag5}
If $j$, $m$, $r$ and $u$ are complex numbers and $n$ is a non-negative integer, then
\begin{equation}
\sum_{k = 0}^n {\binom{{n}}{k}\binom{{n - k + m}}{u}\binom{{r}}{{j - k}}}  = \sum_{k = 0}^n {\binom{{n}}{k}\binom{{n - k + r}}{j}\binom{{m}}{{u - k}}} .
\end{equation}
\end{proposition}

\begin{proof}
Use
\begin{equation*}
s_k  = \binom{{k}}{u},\quad\sigma _k  = ( - 1)^u \delta _{ku} ,
\end{equation*}
and
\begin{equation*}
t_k  = \binom{{k}}{j},\quad\tau_k  = ( - 1)^j \delta _{kj},
\end{equation*}
in~\eqref{uq1kbw0}.
\end{proof}

\begin{theorem}
Let $m$ and $r$ be non-negative integers. If $(s_k)$ and $(\sigma_k)$, $k=0,1,2,\ldots$, are a binomial-transform pair of the first kind, so are the sequences $(b_k)$ and $(\beta_k)$, $k=0,1,2,\dots$, where
\begin{equation}
b_k=\sum_{p = 0}^r {( - 1)^p \binom{{r}}{p}s _{k + p + m} } \text{ and } \beta_k=\sum_{p = 0}^m {( - 1)^p \binom{{m}}{p}\sigma _{k + p + r} }.
\end{equation}
Thus,
\begin{equation}
\sum_{k=0}^n{(-1)^k\binom nk\sum_{p = 0}^r {( - 1)^p \binom{{r}}{p}s _{k + p + m} }}=\sum_{k = 0}^m {( - 1)^k \binom mk\sigma _{n + k + r} }.
\end{equation}
\end{theorem}

\begin{proof}
Let $(s_k)$ and $(\sigma_k)$, $k=0,1,2,\ldots$, be a binomial transform pair of the first kind. From~\eqref{hl4i644}, for an arbitrary non-negative integer $m$, we recognize 
\begin{equation}\label{ngubu44}
s_{k+m} \text{ and } \sum_{p = 0}^m {( - 1)^p \binom{{m}}{p}\sigma _{k + p} },\quad k=0,1,2,\ldots,
\end{equation}
as a binomial-transform pair of the first kind.

Let $(t_k)$ and $(\tau_k)$, $k=0,1,2,\ldots$, be a binomial transform pair of the first kind. By the same token, for an arbitrary non-negative integer $r$,
\begin{equation}\label{f4dkuss}
t_{k+r} \text{ and } \sum_{p = 0}^r {( - 1)^p \binom{{r}}{p}\tau _{k + p} },\quad k=0,1,2,\ldots,
\end{equation}
are a binomial-transform pair of the first kind.

In~\eqref{ngubu44}, let
\begin{equation}\label{a5y59pa}
\tau_k=s_{k+m} \text{ and } t_k=\sum_{p = 0}^m {( - 1)^p \binom{{m}}{p}\sigma _{k + p} },\quad k=0,1,2,\ldots
\end{equation}
Thus, using~\eqref{a5y59pa}, the binomial transform pair given in~\eqref{f4dkuss} consists of
\begin{equation}
t_{k+r}=\sum_{p = 0}^m {( - 1)^p \binom{{m}}{p}\sigma _{k + p + r} }
\end{equation}
and
\begin{equation}
\sum_{p = 0}^r {( - 1)^p \binom{{r}}{p}\tau _{k + p} }=\sum_{p = 0}^r {( - 1)^p \binom{{r}}{p}s _{k + p + m} }.
\end{equation}
The proof is complete.
\end{proof}
In the next theorem, we present another generalization of~\eqref{main1}.
\begin{theorem}
Let $m$, $n$, $r$, $u$ and $v$ be non-negative integers. If $\{(s_k), (\sigma_k)\}$ and $\{(t_k), (\tau_k)\}$, $k=0,1,2,\ldots$, are binomial-transform pairs of the first kind, then
\begin{align}
&\sum_{k = 0}^n {( - 1)^k \binom{{n}}{k}\sum_{p = 0}^r {( - 1)^p \binom{{r}}{p}s_{n - k + p + m} } \sum_{j = 0}^u {( - 1)^j \binom{{u}}{j}t_{k + j + v} } }\nonumber\\
&\qquad  = \sum_{k = 0}^n {( - 1)^k \binom{{n}}{k}\sum_{p = 0}^m {( - 1)^p \binom{{m}}{p}\sigma _{k + p + r} } \sum_{j = 0}^v {( - 1)^j \binom{{v}}{j}\tau _{n - k + j + u} } } .
\end{align}
\end{theorem}

\begin{proof}
In~\eqref{jv6g76k} substitute
\begin{equation*}
a_k=\sum_{p = 0}^r {( - 1)^p \binom{{r}}{p}s _{k + p + m} },\quad \alpha_k=\sum_{p = 0}^m {( - 1)^p \binom{{m}}{p}\sigma _{k + p + r} },
\end{equation*}
and
\begin{equation*}
b_k=\sum_{j = 0}^u {( - 1)^j \binom{{u}}{j}t _{k + j + v} },\quad \beta_k=\sum_{j = 0}^v {( - 1)^j \binom{{v}}{j}\tau _{k + j + u} }.
\end{equation*}
\end{proof}

\section{Relations between binomial transforms}\label{sec.relations}
We derive various relations between binomial-transform pairs of both kinds. We develop relations which allow one to construct new binomial-transform pairs from known ones. Theorems~\ref{thm.v41e7ah} to~\ref{thm.u5nhiva} are derived from the binomial theorem.

\begin{theorem}\label{thm.v41e7ah}
If $(t_k)$ and $(\tau_k)$, $k=0,1,2,\dots$, are a binomial-transform pair of the first kind, then so are the sequences $(a_k)$ and $(\alpha_k)$, $k=r,r+1,r+2,\dots$, where
\begin{equation*}
a_k  = \binom{{k}}{r}t_{k - r} \text{ and }\alpha _k  = (-1)^r\binom{{k}}{r}\tau _{k - r} .
\end{equation*}

\end{theorem}
In particular, $(kt_{k-1})$ and $(-k\tau_{k-1})$, $k=1,2,3,\ldots$, constitute a binomial-transform pair of the first kind. 

\begin{proof}
Differentiate the binomial theorem 
\begin{equation}\label{binomial}
\sum_{k = 0}^n {( - 1)^k \binom{{n}}{k}x^k }  = \left( {1 - x} \right)^n ,
\end{equation}
$r$ times with respect to $x$ to obtain
\begin{equation*}
\sum_{k = r}^n {( - 1)^k \binom{{n}}{k}\binom{{k}}{r}x^{k - r} }  = ( - 1)^r \binom{{n}}{r}\left( {1 - x} \right)^{n - r} 
\end{equation*}
and operate on both sides with $\mathcal L_x$.
\end{proof}

\begin{theorem}\label{thm.cijxizk}
Let $n$ and $r$ be non-negative integers. If $(t_k)$ and $(\tau_k)$, $k=0,1,2,\dots$, are a binomial-transform pair of the first kind, then so are the sequences $(a_k)$ and $(\alpha_k)$, $k=r,r+1,r+2,\dots$, where
\begin{equation*}
a_k  = \binom{{k}}{r}t_k \text{ and }\alpha _k  = \binom{{k}}{r}\sum_{j = 0}^r {( - 1)^j \binom{{r}}{j}\tau _{k - j} } .
\end{equation*}
Thus,
\begin{equation}\label{zuiou6h}
\sum_{k = r}^n {( - 1)^k \binom{{n}}{k}\binom{{k}}{r}t_k }  = \binom{n}{r}\sum_{k = 0}^r {( - 1)^k \binom{{r}}{k}\tau _{n - k} } .
\end{equation}
\end{theorem}

\begin{proof}
Write the binomial theorem as
\begin{equation*}
\sum_{k = 0}^n {( - 1)^k \binom{{n - r}}{k}x^k }  = \left( {1 - x} \right)^{n - r},\quad n\ge r
\end{equation*}
and use the fact that
\begin{equation*}
\binom{{n - r}}{k}\binom{{n}}{r} = \binom{{n - k}}{r}\binom{{n}}{k}
\end{equation*}
to obtain
\begin{equation*}
\sum_{k = 0}^n {( - 1)^k \binom{{n}}{k}\binom{{k}}{r}x^{ - k} }  = ( - 1)^n \binom{{n}}{r}x^{ - n} \left( {1 - x} \right)^{n - r} .
\end{equation*}
Now write $1/x$ for $x$, simplify and replace $x$ with $1-x$ to get
\begin{equation}\label{my5n6de}
\sum_{k = 0}^n {( - 1)^k \binom{{n}}{k}\binom{{k}}{r}\left( {1 - x} \right)^k }  = \binom nr\sum_{k = 0}^r {( - 1)^k \binom{{r}}{k}x^{n - k} },
\end{equation}
from which~\eqref{zuiou6h} follows.
\end{proof}

\begin{theorem}\label{thm.u5nhiva}
If $(t_k)$ and $(\tau_k)$, $k=0,1,2,\dots$, are a binomial-transform pair of the first kind, then so are the sequences $(a_k)$ and $(\alpha_k)$, $k=0,1,2,\dots$, where
\begin{equation}\label{i2f6ya8}
a_k  = \frac{{t_{k + 1} }}{{k + 1}}\text{ and }\alpha _k  = \frac{{\tau _0  - \tau _{k + 1} }}{{k + 1}} .
\end{equation}
Thus,
\begin{equation}\label{xrpbz0j}
\sum_{k = 0}^n {( - 1)^k \binom{{n}}{k}\frac{{t_{k + 1} }}{{k + 1}}}  = \frac{{\tau _0  - \tau _{n + 1} }}{{n + 1}}.
\end{equation}
\end{theorem}

\begin{proof}
Integrate the binomial theorem~\eqref{binomial} from $0$ to $y$ and replace $y$ with $x$. This gives
\begin{equation*}
\sum_{k = 0}^n {( - 1)^k \binom{{n}}{k}\frac{{\left( {1 - x} \right)^{k + 1} }}{{k + 1}}}  = \frac{{1 - x^{n + 1} }}{{n + 1}}.
\end{equation*}
Now act on both sides with $\mathcal L_x$ to obtain~\eqref{xrpbz0j}.
\end{proof}

\begin{remark}
Each repeated application of Theorem~\ref{thm.u5nhiva} produces a new binomial-transform pair of the first kind. Given the binomial-transform pair $\{(t_k),(\tau_k)\}$, $k=0,1,2,\ldots$ and starting with~\eqref{i2f6ya8} as
\begin{equation*}
t_k^{(1)}  = \frac{{t_{k + 1} }}{{k + 1}}\text{ and }\tau _k^{(1)}  = \frac{{\tau _0  - \tau _{k + 1} }}{{k + 1}},
\end{equation*}
we can derive a new binomial transform pair, namely $\{t_k^{(2)},\tau_k^{(2)}\}$, for $k=0,1,2,\ldots$, where
\begin{equation}\label{z7tq26h}
t_k^{(2)}  = \frac{{t_{k + 1}^{(1)} }}{{k + 1}}=\frac{t_{k+2}}{(k+1)(k+2)},
\end{equation}
and
\begin{equation}\label{lvddhnp}
\tau _k^{(2)}  = \frac{{\tau _0^{(1)}  - \tau _{k + 1}^{(1)} }}{{k + 1}} = \frac{{\tau _0  - \tau _1 }}{{k + 1}} - \frac{{\left( {\tau _0  - \tau _{k + 2} } \right)}}{{(k + 1)(k + 2)}}.
\end{equation}
From~\eqref{z7tq26h} and~\eqref{lvddhnp} we can develop $\{t_k^{(3)},\tau_k^{(3)}\}$, for $k=0,1,2,\ldots$, where
\begin{equation}
t_k^{(3)}  = \frac{{t_{k + 1}^{(2)} }}{{k + 1}} = \frac{{t_{k + 3} }}{{(k + 1)(k + 2)(k + 3)}}
\end{equation}
and
\begin{align}
\tau _k^{(3)}  &= \frac{{\tau _0^{(2)}  - \tau _{k + 1}^{(2)} }}{{k + 1}}\nonumber\\
&= \frac{{\tau _0  - \tau _1 }}{{k + 1}} - \frac{{\tau _0  - \tau _2 }}{{2(k + 1)}} - \frac{{\tau _0  - \tau _1 }}{{(k + 1)(k + 2)}} + \frac{{\tau _0  - \tau _{k + 3} }}{{(k + 1)(k + 2)(k + 3)}}.
\end{align}
\end{remark}
The consequence of~\ref{thm.u5nhiva} stated in Corollary~\ref{cor.f3kb0ir} can be proved by induction.

\begin{corollary}\label{cor.f3kb0ir}
Let $\{(t_k),(\tau_k)\}$, $k=0,1,2,\ldots$, be a binomial-transform pair of the first kind. If $\tau_0=0=\tau_1=\cdots=\tau_{r-1}$, then $(t_k^{(r)})$ and $(\tau_k^{(r)})$ are a binomial transform pair of the first kind, where
\begin{equation}
t_k^{(r)}  = \frac{{t_{k + r} }}{{\left( {k + 1} \right)\left( {k + 2} \right) \cdots \left( {k + r} \right)}}\text{ and }\tau_k^{(r)}  = \frac{{( - 1)^r \tau _{k + r} }}{{\left( {k + 1} \right)\left( {k + 2} \right) \cdots \left( {k + r} \right)}},\quad r\in\mathbb Z^+,
\end{equation}
or more compactly,
\begin{equation}
t_k^{(r)}  = \binom{{k + r}}{r}^{ - 1} t_{k + r} \text{ and }\tau _k^{(r)}  = ( - 1)^r \binom{{k + r}}{r}^{ - 1} \tau _{k + r} .
\end{equation}
\end{corollary}

\begin{example}
From the identity~\cite[Identity (6.17)]{graham}:
\begin{equation}
\sum_{k = 0}^n {( - 1)^k \binom{{n}}{k}\braces{{ k + 1}}{{m + 1}}}  = ( - 1)^n \braces{{ n}}{m},
\end{equation}
we identify the binomial-transform pair $\{(t_k),(\tau_k)\}$, $k=0,1,2,\ldots$, where
\begin{equation}\label{v0j8dgg}
t_k  = \braces{{ k + 1}}{{m + 1}}\text{ and }\tau _k  = ( - 1)^k \braces{{ k}}{m},\quad m=0,1,2,\ldots.
\end{equation}
Let $r$ be a positive integer. By~\eqref{urlbi68}, we have that if $m>r-1$, then $\tau_0=0=\tau_1=\tau_2=\cdots=\tau_{r-1}$. Using this fact in Corollary~\ref{cor.f3kb0ir} gives the result stated in Proposition~\ref{prop.ef6zwl0}.
\begin{proposition}\label{prop.ef6zwl0}
Let $m$ and $r$ be non-negative integers such that $m\ge r$. Then $(t_k^{(r)})$ and $(\tau_k^{(r)})$ are a binomial-transform pair of the first kind, where
\begin{equation}
t_k^{(r)}  = \braces{{ k + r + 1}}{{m + 1}}\binom{{k + r}}{r}^{ - 1} \text{ and }\tau _k^{(r)}  = ( - 1)^k \braces{{ k + r}}{m}\binom{{k + r}}{r}^{ - 1} .
\end{equation}
Thus,
\begin{equation}
\sum_{k = 0}^n {( - 1)^k \binom{{n}}{k}\braces{{ k + r + 1}}{{m + 1}}\binom{{k + r}}{r}^{ - 1} }  = ( - 1)^n \braces{{ n + r}}{m}\binom{{n + r}}{r}^{ - 1} .
\end{equation}
In particular,
\begin{equation}
\sum_{k = 0}^n {( - 1)^k \binom{{n}}{k}\braces{{ k + r + 1}}{{r + 1}}\binom{{k + r}}{r}^{ - 1} }  = ( - 1)^n \braces{{ n + r}}{r}\binom{{n + r}}{r}^{ - 1} .
\end{equation}
\end{proposition}
\end{example}

\begin{example}
Consider the following identity~\cite{batir23}:
\begin{equation}
\sum_{k = 0}^n {( - 1)^k \binom{{n}}{k}\binom{{k}}{m}}\binom{{s + k}}{k}  = ( - 1)^n \binom{{n}}{m}\binom{{s + m}}{n},
\end{equation}
from which we identify the binomial-transform pair of the first kind $\{(t_k),(\tau_k)\}$, $k=0,1,2,\ldots$, where
\begin{equation}\label{omb8ayt}
t_k  = ( - 1)^k \binom{{k}}{m}\binom{{s + m}}{k}\text{ and }\tau _k  = \binom{{k}}{m}\binom{{s + k}}{k}.
\end{equation}
Use of~\eqref{omb8ayt} in Corollary~\ref{cor.f3kb0ir} produces the next result.
\begin{proposition}
Let $s$ be a complex number that is not a negative integer. Let $m$ and $r$ be non-negative integers such that $m\ge r$. Then $(t_k^{(r)})$ and $(\tau_k^{(r)})$ are a binomial-transform pair of the first kind, where
\begin{equation}
t_k^{(r)}  = ( - 1)^k \binom{{k + r}}{m}\binom{{k + r}}{k}^{ - 1} \binom{{s + m}}{{k + r}}\text{ and }\tau _k^{(r)}  = \binom{{k + r}}{m}\binom{{k + r}}{k}^{ - 1} \binom{{k + s + r}}{s}.
\end{equation}
Thus,
\begin{equation}
\sum_{k = 0}^n {\binom nk \binom{{k + r}}{m}\binom{{k + r}}{k}^{ - 1} \binom{{s + m}}{{k + r}}}  = \binom{{n + r}}{m}\binom{{n + r}}{n}^{ - 1} \binom{{n + s + r}}{s}.
\end{equation}
In particular,
\begin{equation}
\sum_{k = 0}^n {\binom nk\binom{{s + r}}{{k + r}}}  = \binom{{n + s + r}}{s}.
\end{equation}
\end{proposition}

\end{example}

\begin{theorem}\label{thm.ali2l2f}
Let $\{(a_k),(\alpha_k)\}$, $k=0,1,2,\ldots$, be a binomial-transform pair of the first kind. Let $r$ be a non-negative integer. Then $\{(t_k),(\tau_k)\}$, $k=0,1,2,\ldots$, is also a binomial-transform pair of the first kind, where
\begin{equation}
t_k  = \binom{{k + r}}{r}^{ - 1} a_{k + r} 
\end{equation}
and
\begin{equation}
\tau _k  = ( - 1)^r \binom{{k + r}}{r}^{-1}\alpha _{k+r}  - ( - 1)^r \binom{{k + r}}{r}^{-1}\sum_{j = 0}^{r - 1} {( - 1)^j \binom{{k + r}}{j}a_j } .
\end{equation}

\end{theorem}

\begin{proof}
Let $r$ be a non-negative integer. Let $(a_k)$ be a sequence of complex numbers, where
\begin{equation}
a_k  =  
\begin{cases}
 0,&\text{if $k < r$;} \\ 
 x^k ,&\text{if $k \ge r$;} \\ 
\end{cases}
\end{equation}
where $x$ is a complex number.
The inverse transform of $(a_k)$ is computed as
\begin{align*}
\alpha _k  &= \sum_{j = 0}^k {( - 1)^j \binom{{k}}{j}a_j }\\
&= \sum_{j = 0}^k {( - 1)^j \binom{{k}}{j}x^j }  - \sum_{j = 0}^{r - 1} {( - 1)^j \binom{{k}}{j}x^j }\\ 
&= \left( {1 - x} \right)^k  - \sum_{j = 0}^{r - 1} {( - 1)^j \binom{{k}}{j}x^j } .
\end{align*}
By Corollary~\ref{cor.f3kb0ir}, the sequences $(a_k^{'})$ and $(\alpha_k^{'})$ where
\begin{equation}
a_k^{'}= \binom{k + r}{r}^{ - 1} a_{k+r} \text{ and } \alpha_k^{'}=( - 1)^r \binom{k + r}{r}^{ - 1} \alpha _{k+r}
\end{equation}
are a binomial-transform pair of the first kind. Thus,
\begin{align}\label{oftxyh6}
&\sum_{j = 0}^k {( - 1)^j \binom{k}{j}\binom{{j + r}}{r}^{ - 1} x^{j + r} } \nonumber\\
& = ( - 1)^r \binom{{j + r}}{r}^{-1}\left( {1 - x} \right)^{k+r}  - ( - 1)^r \binom{{k + r}}{r}^{ - 1} \sum_{j = 0}^{r - 1} {( - 1)^j \binom{{k + r}}{j}x^j } .
\end{align}
The statement of the theorem now follows upon application of Lemma~\ref{lem.qqdca04} to~\eqref{oftxyh6}.
\end{proof}

\begin{remark}
Identity~\eqref{oftxyh6} is entry (4.13) in Gould's book~\cite[p.47]{gould}.
\end{remark}

\begin{example}
If we consider the anti-self-inverse sequence $(a_k)$ and $(\alpha_k)$, where $a_k=F_k=-\alpha_k$, we see on account of Theorem~\ref{thm.ali2l2f} that $t_k$ and $\tau_k$ are a binomial-transform pair of the first kind, where
\begin{equation}
t_k  = \binom{{k + r}}{r}^{ - 1} 
\end{equation}
and
\begin{equation}
\tau _k  = ( - 1)^{r - 1} \binom{{k + r}}{r}^{ - 1} F_{k + r}  - ( - 1)^r \binom{{k + r}}{r}^{ - 1} \sum_{j = 0}^{r - 1} {( - 1)^j \binom{{k + r}}{r}F_j } .
\end{equation}
Thus,
\begin{align}
&\sum_{k = 0}^n {( - 1)^k \binom{{n}}{k}\binom{{k + r}}{r}^{ - 1} F_{k + r} }\nonumber\\
&\qquad = ( - 1)^{r - 1} \binom{{n + r}}{r}^{ - 1} F_{n + r}  - ( - 1)^r \binom{{n + r}}{r}^{ - 1} \sum_{j = 0}^{r - 1} {( - 1)^j \binom{{n + r}}{j}F_j } ;
\end{align}
with the corresponding result for the Lucas sequence being
\begin{align}
&\sum_{k = 0}^n {( - 1)^k \binom{{n}}{k}\binom{{k + r}}{r}^{ - 1} L_{k + r} }\nonumber\\
&\qquad = ( - 1)^r \binom{{n + r}}{r}^{ - 1} L_{n + r}  - ( - 1)^r \binom{{n + r}}{r}^{ - 1} \sum_{j = 0}^{r - 1} {( - 1)^j \binom{{n + r}}{j}L_j } ;
\end{align}

\end{example}

\begin{proposition}
Let $n$ be a non-negative integer. If $\{(\bar t_k), (\bar\tau_k)\}$ is a binomial-transform pair of the second kind, then
\begin{equation}\label{ps67scn}
\sum_{k = 0}^n {\bar \tau _k }  = \sum_{k = 0}^n {\binom{{n + 1}}{{k + 1}}\bar t_k } .
\end{equation}

\end{proposition}
In particular, if $x$ is a complex variable, then
\begin{equation}\label{vz664g0}
\sum_{k = 0}^n {(1+x)^k }  = \sum_{k = 0}^n {\binom{{n + 1}}{{k + 1}}x^k } .
\end{equation}
\begin{proof}
Set $m=n$ and $p=0$ in~\eqref{a2qzajx}.
\end{proof}

\begin{remark}
Shifting the summation index, identity~\eqref{ps67scn} can also be written as
\begin{equation}\label{oy8uhm0}
\sum_{k = 1}^n {\bar \tau _{k-1} }  = \sum_{k = 1}^n {\binom nk\bar t_{k-1} } ;
\end{equation}
which for a binomial transform pair of the first kind $\{(t_k),(\tau_k)\}$, $k=0,1,2,\ldots$, also means
\begin{equation}\label{xc556tq}
\sum_{k = 1}^n { \tau _{k-1} }  = \sum_{k = 1}^n {(-1)^{k-1}\binom nk t_{k-1} } ,
\end{equation}
in view of Remark~\ref{rem.r7dv221}.
\end{remark}

\begin{theorem}\label{thm.hq03eji}
If $(t_k)$ and $(\tau_k)$, $k=0,1,2,\dots$, are a binomial-transform pair of the first kind, then so are the sequences $(a_k)$ and $(\alpha_k)$, $k=0,1,2,\dots$, where
\begin{equation*}
a_k=- \sum_{j = 1}^k {\tau _{j - 1} } \text{ and } \alpha_k=(1 - \delta _{k0} )t_{k - 1 + \delta _{k0} }.
\end{equation*}

\end{theorem}
\begin{proof}
An immediate consequence of~\eqref{xc556tq}.
\end{proof}

\begin{example}
Sequences $(L_{k+1}-1)$ and $(-L_{k-1})$, $k=1,2,\ldots$, are a binomial-transform pair of the first kind since, if we take $t_k=L_k=\tau_k$, we have
\begin{equation*}
\sum_{j=1}^k L_{j-1}=L_{k+1}-1.
\end{equation*}

Similarly, $(F_{k+1}-1)$ and $(F_{k-1})$, $k=1,2,\ldots$, are a binomial-transform pair of the first kind.
\end{example}

\begin{theorem}\label{thm.qpev64c}
If $(t_k)$ and $(\tau_k)$, $k=0,1,2,\dots$, are a binomial-transform pair of the first kind, then so are the sequences $(b_k)$ and $(\beta_k)$, $k=0,1,2,\dots$, where
\begin{equation*}
b_k=\sum_{j = 1}^{k - 1 + \delta _{k0} } {t_{j - 1} } \text{ and } \beta_k=\sum_{j = 1}^{k - 1 + \delta _{k0} } {\tau _{j - 1} }; 
\end{equation*}
so that
\begin{equation}\label{vcrpiix}
\sum_{k = 1}^n {( - 1)^k \binom{{n}}{k}\sum_{j = 1}^{k - 1} {t_{j - 1} } }  = \sum_{k = 1}^{n - 1} {\tau _{k - 1} },\quad n\in\mathbb Z^+ .
\end{equation}

\end{theorem}
In particular, if $(t_k)$, $k=0,1,2,\dots$, is a self-inverse sequence, so is the sequence $\left(\sum_{j = 1}^{k - 1 + \delta _{k0} } {t_{j - 1} }\right)$.

\begin{proof}
With Theorem~\ref{thm.hq03eji} in mind, let $\{(a_k),(\alpha_k)\}$, $k=0,1,2,\dots$, be a binomial-transform pair of the first kind, where
\begin{equation}\label{ifevf20}
a_k  =  - \sum_{j = 1}^k {\tau _{j - 1} },\quad\alpha _k  = \left( {1 - \delta _{k0} } \right)t_{k - 1 + \delta _{k0} } ,
\end{equation}
and $\{(t_k),(\tau_k)\}$, $k=0,1,2,\dots$, is a binomial-transform pair of the first kind.

By the same Theorem~\ref{thm.hq03eji}, $ - \sum_{j = 1}^k {\alpha _{j - 1} }$ and $\left( {1 - \delta _{k0} } \right)a_{k - 1 + \delta _{k0} }$ are a binomial-transform pair of the first kind.

Using~\eqref{ifevf20}, we have
\begin{align*}
 - \sum_{j = 1}^k {\alpha _{j - 1} }  &=  - \sum_{j = 1}^k {\left( {1 - \delta _{j1} } \right)t_{j - 2 + \delta _{j1} } } \\
& =  - \sum_{j = 1}^k {t_{j - 2 + \delta _{j1} } }  + \sum_{j = 1}^k {\delta _{j1} t_{j - 2 + \delta _{j1} } } \\
& =  - t_0  - \sum_{j = 2}^k {t_{j - 2} }  + t_0 \\
 &=  - \sum_{j = 2}^k {t_{j - 2} }  =  - \sum_{j = 1}^{k - 1} {t_{j - 1} } .
\end{align*}
Since the left side is an empty sum for $k=0$, in order to preserve equality, we therefore write
\begin{equation*}
- \sum_{j = 1}^k {\alpha _{j - 1} }= - \sum_{j = 1}^{k - 1 + \delta_{k0}} {t_{j - 1} } .
\end{equation*}
From~\eqref{ifevf20}, we also have
\begin{equation*}
\left( {1 - \delta _{k0} } \right)a_{k - 1 + \delta _{k0} }  =  - \left( {1 - \delta _{k0} } \right)\sum_{j = 1}^{k - 1 + \delta _{k0} } {\tau _{j - 1} }  =  - \sum_{j = 1}^{k - 1 + \delta _{k0} } {\tau _{j - 1} } .
\end{equation*}
\end{proof}

\begin{theorem}\label{thm.hn3cm79}
If $(t_k)$ and $(\tau_k)$, $k=0,1,2,\dots$, are a binomial-transform pair of the first kind, then so are the sequences $(a_k)$ and $(\alpha_k)$, $k=0,1,2,\dots$, where
\begin{equation}\label{abwtmhn}
a_k=\frac{1}{{\left( {k + 1} \right)\left( {k + 2} \right)}}\sum\limits_{j = 0}^k {t_j }\text{ and } \alpha_k=\frac{1}{{\left( {k + 1} \right)\left( {k + 2} \right)}}\sum\limits_{j = 0}^k {\tau _j }.
\end{equation}

\end{theorem}
In particular, if $(t_k)$, $k=0,1,2,\dots$, is a self-inverse sequence, so is the sequence
\begin{equation*}
\left(\frac{1}{{\left( {k + 1} \right)\left( {k + 2} \right)}}\sum\limits_{j = 0}^k {t_j }\right).
\end{equation*}

\begin{proof}
By shifting the index, identity~\eqref{vcrpiix} can be written as
\begin{equation*}
\sum_{k = 0}^n {( - 1)^k \binom{{n + 2}}{{k + 2}}\sum_{j = 0}^k {t_j } }  = \sum_{k = 0}^n {\tau _k } ,
\end{equation*}
from which~\eqref{abwtmhn} follows since
\begin{equation*}
\binom{{n + 2}}{{k + 2}} = \frac{{\left( {n + 2} \right)\left( {n + 1} \right)}}{{\left( {k + 2} \right)\left( {k + 1} \right)}}\binom{{n}}{k}.
\end{equation*}
\end{proof}
\begin{example}
The sequence
\begin{equation*}
\left(\frac {{L_{k+2}-1}}{(k+1)(k+2)}\right),\quad k=0,1,2,\ldots
\end{equation*}
is a self-inverse sequence, since $(L_k)$ is a self-inverse sequence and
\begin{equation*}
\sum_{k=0}^n L_k=L_{n+2}-1.
\end{equation*}
\end{example}

\begin{theorem}\label{thm.ph6iklv}
If $(t_k)$ and $(\tau_k)$, $k=0,1,2,\dots$, are a binomial-transform pair of the first kind, then so are the sequences $(a_k)$ and $(\alpha_k)$, $k=0,1,2,\dots$, where
\begin{equation}
a_k=\frac{1}{{k + 1}}\sum\limits_{j = 0}^k {\tau _j } \text{ and } \alpha_k=\frac{{t_k }}{{k + 1}}.
\end{equation}

\end{theorem}

\begin{proof}
Identity~\eqref{ps67scn} written for a binomial-transform pair of the first kind is
\begin{equation*}
\sum_{k = 0}^n {\tau _k }  = \sum_{k = 0}^n {( - 1)^k \binom{{n + 1}}{{k + 1}}t_k } ,
\end{equation*}
which can also be written as
\begin{equation*}
\frac{1}{{n + 1}}\sum_{k = 0}^n {\tau _k }  = \sum_{k = 0}^n {( - 1)^k \binom{{n}}{k}\frac{{t_k }}{{k + 1}}} ,
\end{equation*}
and hence the theorem.
\end{proof}

\begin{theorem}\label{thm.spugtzj}
If $(t_k)$ and $(\tau_k)$, $k=0,1,2,\dots$, are a binomial-transform pair of the first kind, then so are the sequences $(b_k)$ and $(\beta_k)$, $k=0,1,2,\dots$, where
\begin{equation}
b_k= - \frac{1}{{k + 1}}\sum\limits_{j = 1}^k {\tau _{j - 1} } \text{ and }\beta_k=\frac{1}{{k + 1}}\sum\limits_{j = 1}^k {t_{j - 1} } .
\end{equation}

\end{theorem}

In particular, if $(t_k)$, $k=0,1,2,\dots$, is a self-inverse sequence, then the sequence
\begin{equation*}
\left(\frac{1}{{k + 1}}\sum\limits_{j = 1}^k {t_{j - 1} }\right),\quad  k=0,1,2,\dots,
\end{equation*}
is an anti-self-inverse sequence, and vice versa.
\begin{proof}
Let $\{(a_k),(\alpha_k)\}$, $k=0,1,2,\dots$, be a binomial-transform pair of the first kind, where
\begin{equation*}
a_k  =  - \sum_{j = 1}^k {\tau _{j - 1} },\quad\alpha _k  = \left( {1 - \delta _{k0} } \right)t_{k - 1 + \delta _{k0} } ,
\end{equation*}
and $\{(t_k),(\tau_k)\}$, $k=0,1,2,\dots$, are a binomial-transform pair of the first kind.

Then
\begin{equation*}
\frac{1}{{k + 1}}\sum\limits_{j = 0}^k {\alpha _j }  = \frac{1}{{k + 1}}\sum\limits_{j = 0}^k {\left( {1 - \delta _{j0} } \right)t_{j - 1 + \delta _{j0} } }  = \frac{1}{{k + 1}}\sum\limits_{j = 1}^k {t_{j - 1} } 
\end{equation*}
and
\begin{equation*}
\frac{{a_k }}{{k + 1}} =  - \frac{1}{{k + 1}}\sum\limits_{j = 1}^k {\tau _{j - 1} }
\end{equation*}
are a binomial-transform pair of the first kind, by Theorem~\ref{thm.ph6iklv}.
\end{proof}

\begin{theorem}
If $(t_k)$ and $(\tau_k)$, $k=0,1,2,\dots$, are a binomial-transform pair of the first kind, then so are the sequences $(b_k)$ and $(\beta_k)$, $k=0,1,2,\dots$, where
\begin{equation}
b_k=\frac{1}{{k + 1}}\sum\limits_{j = 0}^k {\frac{{t_j }}{{j + 1}}} \text{ and }\beta_k=\frac{1}{{\left( {k + 1} \right)^2 }}\sum\limits_{j = 0}^k {\tau _j }.
\end{equation}

\end{theorem}

\begin{proof}
Let $\{(a_k),(\alpha_k)\}$, $k=0,1,2,\dots$, be a binomial-transform pair of the first kind, where
\begin{equation*}
a_k  = \frac{1}{{k + 1}}\sum\limits_{j = 0}^k {\tau _k },\quad \alpha _k  = \frac{{t_k }}{{k + 1}},
\end{equation*}
and $\{(t_k),(\tau_k)\}$, $k=0,1,2,\dots$, is a binomial-transform pair of the first kind.

Then the sequences $(b_k)$ and $(\beta_k)$, $k=0,1,2,\dots$, where
\begin{equation*}
b_k=\frac{1}{{k + 1}}\sum\limits_{j = 0}^k {\alpha _j }  = \frac{1}{{k + 1}}\sum\limits_{j = 0}^k {\frac{{t_j }}{{j + 1}}}
\end{equation*}
and
\begin{equation*}
\beta_k=\frac{{a_k }}{{k + 1}} = \frac{1}{{\left( {k + 1} \right)^2 }}\sum\limits_{j = 0}^k {\tau _j }
\end{equation*}
are a binomial-transform pair of the first kind, by Theorem~\ref{thm.ph6iklv}.
\end{proof}

\begin{theorem}\label{thm.pumwibe}
Let $m$ be a non-negative integer. If $(t_k)$ and $(\tau_k)$, $k=0,1,2,\dots$, are a binomial-transform pair of the first kind, then so are the sequences $(b_k)$ and $(\beta_k)$, $k=0,1,2,\dots$, where
\begin{equation}
b_k  = \frac{1}{{k + 1}}\sum_{j = 0}^k {t_{j + m} } \text{ and } \beta_k  = \frac{1}{{k + 1}}\sum_{p = 0}^m {( - 1)^p \binom{{m}}{p}\tau _{k + p} } .
\end{equation}

\end{theorem}
\begin{proof}
Let $\{(a_k),(\alpha_k)\}$, $k=0,1,2,\ldots$, be a binomial-transform pair of the first kind, where (see Equation~\eqref{w95ejl1})
\begin{equation*}
a_k  = \sum_{p = 0}^m {( - 1)^p \binom{{m}}{p}\tau _{k + p} } \text{ and }\alpha _k  = t_{k + m} .
\end{equation*}
By Theorem~\ref{thm.ph6iklv}, sequences $(b_k)$ and $(\beta_k)$, $k=0,1,2,\ldots$, given by
\begin{equation*}
b_k  = \frac{1}{{k + 1}}\sum_{j = 0}^k {\alpha _k }  = \frac{1}{{k + 1}}\sum_{j = 0}^k {t_{j + m} },\quad k=0,1,2,\ldots,
\end{equation*}
and
\begin{equation*}
\beta_k  = \frac{{a_k }}{{k + 1}} = \frac{1}{{k + 1}}\sum_{p = 0}^m {( - 1)^p \binom{{m}}{p}\tau _{k + p} } ,\quad k=0,1,2,\ldots,
\end{equation*}
are also a binomial-transform pair of the first kind.
\end{proof}

\begin{example}
By choosing $t_k=F_k$ and $\tau_k=-F_k$ in Theorem~\ref{thm.pumwibe}, we find a new binomial-transform pair of the first kind $\{(b_k),(\beta_k)\}$, $k=0,1,2,\ldots$, where
\begin{equation}
b_k=\frac{F_{m+k+2}-F_{m+1}}{k+1}\text{ and }\beta_k=\frac{(-1)^kF_{m-k}}{k+1},
\end{equation}
since~\cite[p.87]{koshy}:
\begin{equation*}
\sum_{j=0}^k{F_{j+m}}=F_{k+m+2}-F_{m+1}
\end{equation*}
and from~\eqref{q96hz7w}:
\begin{equation*}
\sum_{p=0}^m{(-1)^p\binom mp F_{k+p}}=(-1)^{k-1}F_{m-k}.
\end{equation*}

\end{example}

\begin{theorem}
Let $m$ be a non-negative integer. If $(t_k)$ and $(\tau_k)$, $k=0,1,2,\dots$, are a binomial-transform pair of the first kind, then so are the sequences $(b_k)$ and $(\beta_k)$, $k=0,1,2,\dots$, where
\begin{equation}
b_k  = \frac{1}{{k + m + 1}}\sum_{j = 0}^{k + m} {\tau _j } \text{ and }\beta _k  = \sum_{p = 0}^m {( - 1)^p \binom{{m}}{p}\frac{{t_{k + p} }}{{k + p + 1}}} .
\end{equation}
\end{theorem}
Thus,
\begin{equation}
\sum_{k = 0}^n {\frac{{( - 1)^k }}{{k + m + 1}}\binom{{n}}{k}\sum_{j = 0}^{k + m} {\tau _j } }  = \sum_{k = 0}^m {\frac{{( - 1)^k }}{{k + n + 1}}\binom{{m}}{k}t_{k + n} } .
\end{equation}
\begin{proof}
Let $\{(a_k),(\alpha_k)\}$, $k=0,1,2,\ldots$, be a binomial-transform pair of the first kind, where 
\begin{equation*}
a_k  = \frac{1}{{k + 1}}\sum_{j = 0}^k {\tau _j } \text{ and }\alpha _k  = \frac{{t_k }}{{k + 1}}.
\end{equation*}
By~\eqref{w95ejl1}, so also are $(b_k)$ and $(\beta_k)$, $k=0,1,2,\ldots$, where
\begin{equation*}
b_k=a_{k+m}= \frac{1}{{k + m + 1}}\sum_{j = 0}^{k + m} {\tau _j }
\end{equation*}
and
\begin{equation*}
\beta_k=\sum_{p=0}^m{(-1)^p\binom mp\alpha_{k+p}}=\sum_{p = 0}^m {( - 1)^p \binom{{m}}{p}\frac{{t_{k + p} }}{{k + p + 1}}}.
\end{equation*}
\end{proof}

\begin{theorem}
If $(t_k)$ and $(\tau_k)$, $k=0,1,2,\dots$, are a binomial-transform pair of the first kind, then so are the sequences $(b_k)$ and $(\beta_k)$, $k=0,1,2,\dots$, where
\begin{equation}
b_k  =  - \sum_{j = 1}^k {\frac{{t_{j - 1} }}{j}} \text{ and }\beta _k  = \frac{1}{{k + \delta _{k0} }}\sum_{j = 1}^k {\tau _{j-1} } .
\end{equation}

\end{theorem}

\begin{proof}
Let
\begin{equation*}
a_k  = \frac{1}{{k + 1}}\sum_{j = 0}^k {\tau _j } \text{ and }\alpha _k  = \frac{{t_k }}{{k + 1}},\quad k=0,1,2,\ldots,
\end{equation*}
be a binomial-transform pair of the first kind. By Theorem~\ref{thm.hq03eji}, so also are
\begin{equation*}
 - \sum_{j = 1}^k {\alpha _{j - 1} }= - \sum_{j = 1}^k {\frac{{t_{j - 1} }}{j}}
\end{equation*}
and
\begin{equation*}
\left( {1 - \delta _{k0} } \right)a_{k - 1 + \delta _{k0} } = \frac{1}{{k + \delta _{k0} }}\sum_{j = 1}^k {\tau _{j-1} } .
\end{equation*}
\end{proof}

\begin{theorem}
If $(s_k)$ and $(\sigma_k)$, $k=0,1,2,\dots$, are a binomial-transform pair of the first kind, then so are the sequences $(a_k)$ and $(\alpha_k)$, $k=0,1,2,\dots$, where
\begin{equation}
a_k= - \sum_{j = 1}^k {2^{ - j + 1} s_{j - 1} }\text{ and }\alpha_k= \frac{{2^{ - k + 1} }}{k+\delta_{k0}}\sum_{j = 1}^k {j\binom{{k}}{j}\sigma _{j - 1} },\quad k=0,1,2,\ldots .
\end{equation}
\end{theorem}

\begin{proof}
Consider the binomial-transform pair of the first kind version of~\eqref{pop9ybt}, namely,
\begin{equation}\label{r2y4bpk}
\sum_{k = 0}^n {\binom{{n}}{k}\sigma _k }  = \sum_{k = 0}^n {( - 1)^k \binom{{n}}{k}2^{n - k} s_k } .
\end{equation}
We see that $(b_k)$ and $(\beta_k)$, $k=0,1,2\ldots$, are a binomial transform pair of the first kind, where
\begin{equation}\label{dmpo94c}
b_k=2^{-k}\sum_{j=0}^k{\binom kj\sigma_j}\text{ and } \beta_k=2^{-k}s_k.
\end{equation}
By Theorem~\ref{thm.hq03eji}, $(a_k)$ and $(\alpha_k)$, $k=0,1,2\ldots$, are also a binomial-transform pair of the first kind, where
\begin{equation*}
a_k  =  - \sum_{j = 1}^k {\beta _{j - 1} }\text{ and } \alpha _k  = \left( {1 - \delta _{k0} } \right)b_{k - 1 + \delta _{k0} } .
\end{equation*}
\end{proof}

\begin{theorem}
If $(s_k)$ and $(\sigma_k)$, $k=0,1,2,\dots$, are a binomial-transform pair of the first kind, then so are the sequences $(a_k)$ and $(\alpha_k)$, $k=0,1,2,\dots$, where
\begin{equation}\label{bsp377t}
a_k  = \frac{1}{{k + 1}}\sum_{j = 0}^k {2^{ - j} s_j } \text{ and }\alpha _k  = \frac{{2^{ - k} }}{{k + 1}}\sum_{j = 0}^k {\binom{{k}}{j}\sigma _j } .
\end{equation}
\end{theorem}

\begin{proof}
Let $(b_k)$ and $(\beta_k)$, $k=0,1,2,\ldots$, be the binomial-transform pair of sequences defined in~\eqref{dmpo94c}. By Theorem~\ref{thm.ph6iklv}, $(a_k)$ and $(\alpha_k)$ are also a binomial-transform pair of the first kind, where
\begin{equation*}
a_k  = \frac{1}{{k + 1}}\sum_{j = 0}^k {\beta _j }  = \frac{1}{{k + 1}}\sum_{j = 0}^k {2^{ - j} s_j } 
\end{equation*}
and
\begin{equation*}
\alpha _k  = \frac{{a_k }}{{k + 1}} = \frac{{2^{ - k} }}{{k + 1}}\sum_{j = 0}^k {\binom{{k}}{j}\sigma _j } .
\end{equation*}
\end{proof}

\begin{example}
By choosing $s_k=1$ and $\sigma_k=\delta_{k0}$ in~\eqref{bsp377t}, we find that
\begin{equation*}
\left(\frac{2-2^{-k}}{k+1}\right)\text{ and }\left(\frac{2^{-k}}{k+1}\right),\quad k=0,1,2,\ldots,
\end{equation*}
are a binomial-transform pair of the first kind.
\end{example}

\begin{theorem}
If $(s_k)$ and $(\sigma_k)$, $k=0,1,2,\dots$, are a binomial-transform pair of the first kind, then so are the sequences $(a_k)$ and $(\alpha_k)$, $k=0,1,2,\dots$, where
\begin{equation}
a_k  = \frac{{2^{ - k} }}{{k + 1}}\sum_{j = 0}^k {\sigma _j } \text{ and }\alpha _k  = 2^{ - k} \sum_{j = 0}^k {\binom{{k}}{j}\frac{{s_j }}{{j + 1}}} .
\end{equation}
\end{theorem}

\begin{proof}
Using Theorem~\ref{thm.ph6iklv}, consider the binomial-transform pair of the first kind $(b_k)$ and $(\beta_k)$, $k=0,1,2,\ldots$, where
\begin{equation*}
b_k=\frac{1}{{k + 1}}\sum\limits_{j = 0}^k {\sigma _j } \text{ and } \beta_k=\frac{{s_k }}{{k + 1}}.
\end{equation*}
By~\eqref{dmpo94c},
\begin{equation*}
a_k  = 2^{ - k} b_k \text{ and }\alpha _k  = 2^{ - k} \sum_{j = 0}^k {\binom{{k}}{j}\beta _j } 
\end{equation*}
are also a binomial-transform pair of the first kind.
\end{proof}

\begin{theorem}
If $(s_k)$ and $(\sigma_k)$, $k=0,1,2,\dots$, are a binomial-transform pair of the first kind, then so are the sequences $(a_k)$ and $(\alpha_k)$, $k=0,1,2,\dots$, where
\begin{equation}\label{i4pg5bw}
a _k  = 2^{ - k} \sum_{j = 0}^k {\binom{{k}}{j}2^{ - j} s_j }\text{ and }\alpha_k  = 2^{ - 2k} \sum_{j = 0}^k {\binom{{k}}{j}\sigma _j }  .
\end{equation}

\end{theorem}

\begin{proof}
With~\eqref{dmpo94c} in mind, consider the binomial-transform pair of the first kind $(b_k)$ and $(\beta_k)$, $k=0,1,2,\ldots$, where
\begin{equation*}
b_k= 2^{ - k} \sum_{j = 0}^k {\binom{{k}}{j}\sigma _j }\text{ and } \beta_k  = 2^{ - k} s_k.  
\end{equation*}
By the same~\eqref{dmpo94c}, $a_k$ and $\alpha_k$ are also a binomial-transform pair of the first kind, where
\begin{equation*}
a_k  = 2^{ - k} \sum_{j = 0}^k {\binom{{k}}{j}\beta _j  = 2^{ - k} \sum_{j = 0}^k {\binom{{k}}{j}2^{ - j} s_j } },
\end{equation*}
and
\begin{equation*}
\alpha _k  = 2^{ - k} b_k  = 2^{ - 2k} \sum_{j = 0}^k {\binom{{k}}{j}\sigma _j } .
\end{equation*}
\end{proof}

\begin{example}
By choosing $s_k=1$ and $\sigma_k=\delta_{k0}$ in~\eqref{i4pg5bw}, we obtain
\begin{equation*}
\left(2^{-2k}3^k\right)\text{ and }\left(2^{-2k}\right),\quad k=0,1,2,\ldots,
\end{equation*}
as a binomial-transform pair of the first kind.
\end{example}

\begin{example}
By choosing $s_k=(1+k)^{-1}=\sigma_k$ in~\eqref{i4pg5bw}, we find that the sequences $(a_k)$ and $(\alpha_k)$, $k=0,1,2,\ldots$, where
\begin{equation*}
a_k  = \frac{{2^{1 - k}  - 2^{ - 2k} }}{{k + 1}}\text{ and }\alpha _k  = \frac{{2^{ - 2k} 3^{k + 1}  - 2^{1 - k} }}{{k + 1}},
\end{equation*}
are a binomial-transform pair of the first kind.
\end{example}

\begin{theorem}
If $(s_k)$ and $(\sigma_k)$, $k=0,1,2,\dots$, are a binomial-transform pair of the first kind, then so are the sequences $(a_k)$ and $(\alpha_k)$, $k=0,1,2,\dots$, where
\begin{equation}\label{dzf57nf}
a _k  = -2^{ - k} \sum_{j = 1}^k {\sigma_{j-1} }\text{ and }\alpha_k  = 2^{ - k} \sum_{j = 1}^k {\binom{{k}}{j}s_{j-1} }  .
\end{equation}

\end{theorem}

\begin{proof}
Let $(b_k)$ and $(\beta_k)$ be the binomial-transform pair of the first kind given by
\begin{equation*}
b_k=-\sum_{j=1}^k{\sigma_{j-1}}\text{ and } \beta_k=(1-\delta_{k0})s_{k-1+\delta_{k0}},
\end{equation*}
where $(s_k)$ and $(\sigma_k)$, $k=0,1,2,\ldots$, are a binomial transform-pair of the first kind. By~\eqref{dmpo94c}, it is also true that
\begin{equation*}
a_k=2^{-k}b_k\text{ and }\alpha_k=2^{-k}\sum_{j=0}^k{\binom kj}\beta_j,
\end{equation*}
are also a binomial-transform pair of the first kind.
\end{proof}

\begin{theorem}
Let $m$ be a non-negative integer. If $(s_k)$ and $(\sigma_k)$, $k=0,1,2,\dots$, are a binomial-transform pair of the first kind, then so are the sequences $(a_k)$ and $(\alpha_k)$, $k=0,1,2,\dots$, where
\begin{equation}\label{s3p73qx}
a_k  = 2^{ - k} \sum_{p = 0}^m {( - 1)^p \binom{{m}}{p}\sigma _{p + k} }\text{ and } \alpha _k  = 2^{ - k} \sum_{p = 0}^k {\binom{{k}}{p}s_{p + m} } .
\end{equation}

\end{theorem}

\begin{proof}
Recall (see~\eqref{w95ejl1}) that $(b_k)$ and $(\beta_k)$, $k=0,1,2,\ldots$, are a binomial-transform pair of the first kind, where
\begin{equation*}
b_k=\sum_{k=0}^m{(-1)^p\binom mp\sigma_{k+p}}\text{ and }\beta_k=s_{k+m}.
\end{equation*}
By~\eqref{dmpo94c}, it is also a fact that
\begin{equation*}
a_k=2^{-k}b_k\text{ and }\alpha_k=2^{-k}\sum_{j=0}^k{\binom kj\beta_j}
\end{equation*}
are a binomial-transform pair of the first kind.
\end{proof}

\begin{example}
Choosing $s_k=L_k=\sigma_k$ in~\eqref{s3p73qx} gives the binomial-transform pair of the first kind $(a_k)$ and $(\alpha_k)$, $k=0,1,2,\ldots$, where
\begin{equation*}
a_k=(-1)^k2^{-k}L_{m-k}\text{ and }\alpha_k=2^{-k}L_{m+2k},
\end{equation*}
since (Equation~\eqref{q96hz7w}):
\begin{equation*}
\sum_{p = 0}^m {( - 1)^p \binom{{m}}{p}L_{p + k} }  = ( - 1)^k L_{m - k} 
\end{equation*}
and~\cite[Equation (49)]{vajda}
\begin{equation*}
\sum_{p = 0}^k {\binom{{k}}{p}L_{p + m} }  = L_{m + 2k} .
\end{equation*}
\end{example}

\begin{theorem}
Let $m$ be a non-negative integer. If $(s_k)$ and $(\sigma_k)$, $k=0,1,2,\dots$, are a binomial-transform pair of the first kind, then so are the sequences $(a_k)$ and $(\alpha_k)$, $k=0,1,2,\dots$, where
\begin{equation}
a_k  = 2^{ - k} \sum_{p = 0}^m {( - 1)^p \binom{{m}}{p}2^{ - p} \sigma _{k + p} }\text{ and } \alpha _k  = 2^{ - k - m} \sum_{j = 0}^{k + m} {\binom{{k + m}}{j}s_j } .
\end{equation}
\end{theorem}
\begin{proof}
Let $(b_k)$ and $(\beta_k)$, $k=0,1,2,\ldots$, be a binomial-transform pair of the first kind (see \eqref{dmpo94c}), where
\begin{equation*}
b_k  = 2^{ - k} \sum_{j = 0}^k {\binom{{k}}{j}s_j }\text{ and }\beta _k  = 2^{ - k} \sigma _k,
\end{equation*} 
and $(s_k)$ and $(\sigma_k)$, $k=0,1,2,\ldots$, are a binomial-transform pair of the first kind. By the same~\eqref{dmpo94c}, $(a_k)$ and $(\alpha_k)$, $k=0,1,2,\ldots$, constitute a binomial-transform pair of the first kind, where
\begin{equation*}
a_k  = \sum_{p = 0}^m {( - 1)^p \binom{{m}}{p}\beta _{k + p} }  = 2^{ - k} \sum_{p = 0}^m {( - 1)^p \binom{{m}}{p}2^{ - p} \sigma _{k + p} } 
\end{equation*}
and
\begin{equation*}
\alpha _k  = b_{k + m}  = 2^{ - k - m} \sum_{j = 0}^{k + m} {\binom{{k + m}}{j}s_j } .
\end{equation*}
\end{proof}

\begin{theorem}
Let $m$ be a non-negative integer. If $(s_k)$ and $(\sigma_k)$, $k=0,1,2,\dots$, are a binomial-transform pair of the first kind, then so are the sequences $(a_k)$ and $(\alpha_k)$, $k=0,1,2,\dots$, where
\begin{equation}
a_k  =  - \sum_{j = 1}^k {s_{m + j - 1} } \text{ and }\alpha _k  = \left( {1 - \delta _{k0} } \right)\sum_{p = 0}^m {( - 1)^p \binom{{m}}{p}\sigma _{k - 1 + \delta _{k0}  + p} } .
\end{equation}
\end{theorem}

\begin{proof}
Consider the binomial transform pair of the first kind (Equation~\eqref{w95ejl1}), $\{(a_k), (\alpha_k)\}$, $k=0,1,2,\ldots$,  where
\begin{equation*}
b _k  = \sum_{p = 0}^m {( - 1)^p \binom{{m}}{p}\sigma _{k + p} }\text{ and } \beta_k  = s_{k + m} . 
\end{equation*}
By Theorem~\ref{thm.hq03eji}, $(a_k)$ and $(\alpha_k)$, $k=0,1,2,\ldots$, are also a binomial-transform pair of the first kind, where
\begin{equation*}
a_k=- \sum_{j = 1}^k {\beta _{j - 1} } \text{ and } \alpha_k=(1 - \delta _{k0} )b_{k - 1 + \delta _{k0} }.
\end{equation*}
\end{proof}

\begin{theorem}\label{thm.xl9sjqb}
Let $\{(t_k),(\tau_k)\}$, $k=0,1,2,\ldots$, be a binomial-transform pair of the first kind. Then $\left(kt_k\right)$ and $\left(k(\tau_k-\tau_{k-1})\right)$ are a binomial-transform pair of the first kind for $k=1,2,\ldots$.
\end{theorem}
\begin{proof}
Set $j=1$ in~\eqref{s4jiizc} on page~\pageref{s4jiizc} and use the identity
\begin{equation*}
\binom nk=\frac nk\binom{n-1}{k-1}
\end{equation*}
to obtain
\begin{equation*}
n\sum_{k = 0}^n {( - 1)^k \binom{{n}}{k}t_k }  - \sum_{k = 0}^n {( - 1)^k k\binom{{n}}{k}t_k }  = n\tau _{n - 1} ;
\end{equation*}
and hence
\begin{equation}\label{gnvsdip}
\sum_{k = 0}^n {( - 1)^k \binom{{n}}{k}kt_k }  = n\left( {\tau _n  - \tau _{n - 1} } \right).
\end{equation}

\end{proof}

\begin{example}
If on the basis of the recurrence relation of Bernoulli numbers, namely,
\begin{equation*}
\sum_{k=0}^n{(-1)^k\binom nk (-1)^kB_k}=(-1)^nB_n,
\end{equation*}
and use $t_k=(-1)^kB_k=\tau_k$ in~\eqref{gnvsdip}, we obtain, for $n$ a non-negative integer,
\begin{equation}\label{hmg5qf4}
\sum_{k = 0}^n {\binom nkkB_k }  = ( - 1)^n n\left( {B_n  + B_{n - 1} } \right),
\end{equation}
that is
\begin{equation}
\sum_{k = 0}^n {\binom nkkB_k }  
= \begin{cases}
 nB_n,&\text{if $n$ is even} ; \\ 
  - nB_{n - 1},&\text{if $n>1$ is odd};  \\ 
-1/2,&\text{if $n=1$}. 
 \end{cases} 
\end{equation}

\end{example}

\begin{corollary}\label{cor.l9mldgr}
Let $\{(s_k),(\sigma_k)\}$, $k$ an integer, be a binomial-transform pair of the first kind. Let $S_0(k)=\sigma_k$ and $S_m(k)=k(S_{m-1}(k)-S_{m-1}(k-1))$ for every positive integer $m$. Then $S_m(k)$ and $k^ms_k$ are a binomial-transform pair of the first kind; that is
\begin{equation}
\sum_{k = 0}^n {( - 1)^k \binom{{n}}{k}k^m s_k }  = S_m (n),\quad m=0,1,2,\ldots.
\end{equation}

\end{corollary}
In particular,
\begin{equation}
\sum_{k = 0}^n {( - 1)^k \binom{{n}}{k}k^2 s_k }  = n^2 \left( {\sigma _n  - \sigma _{n - 1} } \right) - n(n - 1)\left( {\sigma _{n - 1}  - \sigma _{n - 2} } \right),
\end{equation}
and
\begin{align}
\sum_{k = 0}^n {( - 1)^k \binom{{n}}{k}k^3 s_k }  &= n^3 \left( {\sigma _n  - \sigma _{n - 1} } \right) - n\left( {n - 1} \right)\left( {2n - 1} \right)\left( {\sigma _{n - 1}  - \sigma _{n - 2} } \right)\nonumber\\
&\qquad + n\left( {n - 1} \right)\left( {n - 2} \right)\left( {\sigma _{n - 2}  - \sigma _{n - 3} } \right).
\end{align}

\begin{theorem}\label{thm.ov7pa39}
If $(t_k)$ and $(\tau_k)$, $k=0,1,2,\dots$, are a binomial-transform pair of the first kind, then so are the sequences $(a_k)$ and $(\alpha_k)$, $k=1,2,3,\dots$, where
\begin{equation*}
a_k=\sum_{j = 1}^k {t_j } \text{ and } \alpha_k=\tau_k-\tau_{k-1}.
\end{equation*}
We also have that $(b_k)$ and $(\beta_k)$, $k=1,2,3,\dots$, are a binomial-transform pair of the first kind, where
\begin{equation*}
b_k  = ( - 1)^k \sum_{j = 1}^k {(-1)^jt_j } \text{ and }\beta _k  = 2^k \sum_{j = 1}^k {2^{ - j} \left( {\tau _j  - \tau _{j - 1} } \right)} .
\end{equation*}

Thus,
\begin{equation}\label{z2iq1wi}
\sum_{k = 1}^n {( - 1)^k \binom{{n}}{k}\sum_{j = 1}^k {t_j } }  = \tau _n  - \tau _{n - 1} 
\end{equation}
and
\begin{equation}\label{d3cygnq}
\sum_{k = 1}^n {\binom{{n}}{k}\sum_{j = 1}^k {( - 1)^j t_j } }  = 2^n \sum_{k = 1}^n {2^{ - k} \left( {\tau _k  - \tau _{k - 1} } \right)} .
\end{equation}

\end{theorem}

\begin{proof}
The following identities are known ($s=0$ in~\cite[Cor. 2, 4]{batir21b}):
\begin{equation}\label{vpwct24}
\sum_{k = 1}^n {( - 1)^k \binom{{n}}{k}\sum_{k = 1}^n {t_k } }  =  - \sum_{k = 0}^{n - 1} {( - 1)^k \binom{{n - 1}}{k}t_{k + 1} } ,
\end{equation}
and
\begin{equation}\label{ik67ykl}
\sum_{k = 1}^n {\binom{{n}}{k}\sum_{k = 1}^n {t_k } }  = 2^n \sum_{k = 1}^n {\frac{1}{{2^k }}\sum_{j = 0}^{k - 1} {\binom{{k - 1}}{j}t_{j + 1} } } .
\end{equation}
Write~\eqref{vpwct24} as
\begin{equation*}
\sum_{k = 1}^n {( - 1)^k \binom{{n}}{k}\sum_{k = 1}^n {t_k } }  = \sum_{k = 1}^n {( - 1)^k \binom{{n - 1}}{{k - 1}}t_k }  = \sum_{k = 0}^n {( - 1)^k \frac{k}{n}\binom{{n}}{k}t_k } 
\end{equation*}
and use~\eqref{gnvsdip} to obtain~\eqref{z2iq1wi}. The proof of~\eqref{d3cygnq} is similar.
\end{proof}

\section{Self-inverse sequences}\label{sec.invariant}
In this section, we discuss how to construct invariant or self-inverse sequences and anti-self-inverse sequences from given binomial transform pairs. The results derived here complement those found by previous researchers~\cite{donaghey76,sun01,wang05}.
\begin{theorem}\label{thm.hdq3q7r}
Let $(s_k)$ and $(\sigma_k)$, $k=0,1,2,\ldots$, be a binomial transform pair of the first kind. Then, for $k=0,1,2,\ldots$, the sequence $(a_k)$ given by $a_k=s_k+\sigma_k$ is a self-inverse sequence while the sequence $(b_k)$ where $b_k=s_k-\sigma_k$ is an anti-self-inverse sequence.
\end{theorem}
\begin{proof}
We have
\begin{align*}
\sum_{k = 0}^n {( - 1)^k \binom{{n}}{k}\left( {s_k  + \sigma _k } \right)}  &= \sum_{k = 0}^n {( - 1)^k \binom{{n}}{k}s_k }  + \sum_{k = 0}^n {( - 1)^k \binom{{n}}{k}\sigma _k }\\ 
&=\sigma_n+s_n,
\end{align*}
as claimed. Similarly,
\begin{equation*}
\sum_{k = 0}^n {( - 1)^k \binom{{n}}{k}\left( {s_k  - \sigma _k } \right)}=\sigma_n-s_n=-(s_n-\sigma_n).
\end{equation*}
\end{proof}
In particular, for $k=0,1,2,\ldots$, the sequence $(a_k)$ given by $a_k=x^k+(1-x)^k$ is a self-inverse sequence while the sequence $(b_k)$ where $b_k=x^k-(1-x)^k$ is an anti-self-inverse sequence for every complex number $x$.

\begin{theorem}\label{thm.dcnj368}
If $(t_k)$, $k=0,1,2,\ldots,$ is an anti-self-inverse sequence, then the sequence $(a_k)$, where
\begin{equation}\label{wrspbj4}
a_k=\frac{t_{k+1}}{k+1}
\end{equation}
is a self inverse sequence.

\end{theorem}

\begin{proof}
Clearly, $t_0=0=\tau_0$ for an anti-self-inverse sequence. Since here $(t_k)$ is anti-self-inverse by hypothesis, we have $\tau_k=-t_k$ and hence~\eqref{wrspbj4} when we invoke Theorems~\ref{thm.u5nhiva} and~\ref{thm.hdq3q7r}.
\end{proof}

\begin{remark}
Theorem~\ref{thm.dcnj368} is known; see, for example, Wang~\cite{wang05}.
\end{remark}

\begin{theorem}
Let $(s_k)$ and $(\sigma_k)$, $k=0,1,2,\ldots$, be a binomial transform pair of the first kind. Then, for $k=0,1,2,\ldots$, the sequence $(a_k)$ given by 
\begin{equation}
a_k=\frac1{k+1}\left(s_{k+1}-\sigma_{k+1}\right)
\end{equation}
is a self-inverse sequence.
\end{theorem}
\begin{proof}
This is a consequence of Theorems~\ref{thm.hdq3q7r} and~\ref{thm.dcnj368}.
\end{proof}
\begin{theorem}
Let $(s_k)$ and $(\sigma_k)$, $k=0,1,2,\ldots$, be a binomial transform pair of the first kind. Then, for $k=0,1,2,\ldots$, the sequence $(a_k)$ given by 
\begin{equation}
a_k=H_{k-1+\delta_{k0}}
\end{equation}
is a self-inverse sequence while the sequence $(b_k)$, where
\begin{equation}
b_k  = \frac{ {kH_k  + 1 - \delta_{k0}} }{{k + \delta _{k0} }}
\end{equation}
is an anti-self-inverse sequence.
\end{theorem}
\begin{proof}
Since
\begin{equation*}
\sum_{k = 0}^n {( - 1)^k \binom{{n}}{k}H_k } 
=\begin{cases}
  - 1/n,&\text{$n\ne 0$}; \\ 
 0,&\text{$n=0$}; 
 \end{cases} 
\end{equation*}
we have that, for $k=0,1,2,\ldots$, sequences $(s_k)$ and $(\sigma_k)$ are a binomial-transform pair of the first kind, where
\begin{equation}\label{idh7u98}
s_k=H_k\text{ and }\sigma_k=-\frac{1-\delta_{k0}}{k+\delta_{k0}}.
\end{equation}
By Theorem~\ref{thm.hdq3q7r}, for $k=0,1,2,\ldots$, we have that $(a_k)$ is a self-inverse sequence while $(b_k)$ is an anti-self-inverse sequence, where
\begin{align*}
a_k  = s_k  + \sigma _k  = H_k  - \frac{{1 - \delta _{k0} }}{{k + \delta _{k0} }}
& =  \begin{cases}
 H_k  - 1/k,&\text{$k\ne 0$} ;\\ 
 0,&\text{$k=0$} \\ 
 \end{cases} \\
& =  \begin{cases}
 H_{k-1},&\text{$k\ne 0$}; \\ 
 0,&\text{$k=0$}; \\ 
 \end{cases} \\
& = H_{k - 1 + \delta _{k0} } ;
\end{align*}
and
\begin{align*}
a_k  = s_k  - \sigma _k  = H_k  + \frac{{1 - \delta _{k0} }}{{k + \delta _{k0} }}
& =  \begin{cases}
 H_k  + 1/k,&\text{$k\ne 0$} ;\\ 
 0,&\text{$k=0$;} \\ 
 \end{cases} \\
& =  \begin{cases}
 \frac{kH_k+1}k,&\text{$k\ne 0$}; \\ 
 0,&\text{$k=0$}; \\ 
 \end{cases} \\
& =  \frac{{1 - \delta _{k0} }}{{k + \delta _{k0} }}\left( {kH_k  + 1} \right) .
\end{align*}
\end{proof}
\begin{remark}
The self-inverse sequence~$(a_k)$ and the anti-self-inverse sequence~$(b_k)$ are given more briefly by
\begin{equation*}
a_k=H_{k-1}\text{ and }b_k  = \frac{\left( {kH_k  + 1} \right)}k,\quad k=1,2,3,\ldots
\end{equation*}
\end{remark}

\begin{theorem}
For $k=1,2,3,\ldots$, the sequence $(s_k)$ is an anti-self-inverse sequence, where
\begin{equation}
s_k=(-1)^{k-1}kB_{k-1};
\end{equation}
so that
\begin{equation}
\sum_{k = 1}^n {\binom{{n}}{k}kB_{k - 1} }  = ( - 1)^{n - 1} nB_{n - 1} .
\end{equation}
\end{theorem}
\begin{proof}
Follows from~\eqref{hmg5qf4} and Theorem~\ref{thm.hdq3q7r}.
\end{proof}

Using Theorem~\ref{thm.hdq3q7r}, self-inverse sequences and anti-self-inverse sequences can be deduced from each of the binomial-transform pairs established in Section~\ref{sec.relations}. We give some examples.
\begin{theorem}
Let $(s_k)$ and $(\sigma_k)$, $k=0,1,2,\ldots$, be a binomial transform pair of the first kind. Then, for $k=0,1,2,\ldots$, the sequence $(a_k)$ given by 
\begin{equation}
a_k  = \sum\limits_{j = 1}^k {\sigma _{j - 1} }  - \left( {1 - \delta _{k0} } \right)s_{k - 1 + \delta _{k0} }
\end{equation}
is a self-inverse sequence while the sequence $(b_k)$, where
\begin{equation}
b_k  = \sum\limits_{j = 1}^k {\sigma _{j - 1} }  + \left( {1 - \delta _{k0} } \right)s_{k - 1 + \delta _{k0} }
\end{equation}
is an anti-self-inverse sequence.
\end{theorem}
\begin{proof}
Follows from Theorem~\ref{thm.hq03eji} and Theorem~\ref{thm.hdq3q7r}.
\end{proof}

\begin{corollary}\label{cor.c31jd2c}
If $(s_k)$, $k=0,1,2,\ldots$, is a self-inverse sequence, then the sequence $(a_k)$, where
\begin{equation}
a_k=\sum_{j=1}^{k-1}s_{j-1},\quad k=1,2,3,\ldots,
\end{equation}
is a self-inverse sequence while the sequence $(b_k)$, where
\begin{equation}
b_k=\sum_{j=1}^{k-1}s_{j-1}+2s_{k-1},\quad k=1,2,3,\ldots,
\end{equation}
is an anti-self-inverse sequence. Similarly, if $(s_k)$, $k=0,1,2,\ldots$, is an anti-self-inverse sequence, then the sequence $(b_k)$, where
\begin{equation}
b_k=\sum_{j=1}^{k-1}s_{j-1},\quad k=1,2,3,\ldots,
\end{equation}
is an anti-self-inverse sequence while the sequence $(a_k)$, where
\begin{equation}
a_k=\sum_{j=1}^{k-1}s_{j-1}+2s_{k-1},\quad k=1,2,3,\ldots,
\end{equation}
is a self-inverse sequence.
\end{corollary}

Corollary~\ref{cor.c31jd2c} is useful because it allows one to derive a new self-inverse sequence and a new anti-self-inverse sequence from a known self-inverse sequence or anti-self-inverse sequence. See also Corollary~\ref{cor.ego8coh}.

\begin{example}
If we choose $s_k=H_{k-1+\delta_{k0}}$, $k=0,1,2,\ldots$ in Corollary~\ref{cor.c31jd2c} and use also~\eqref{qqzbh28}, we find that $(a_k)$ is a self-inverse sequence, where
\begin{equation}
a_k=(k-2)(H_{k-2}-1),\quad k=2,3,4,\ldots,
\end{equation}
and $(b_k)$ is an anti-self-inverse sequence, where
\begin{equation}
b_k=k(H_{k-2}-1)+2,\quad k=2,3,4,\ldots
\end{equation}
\end{example}

\begin{theorem}\label{thm.uix4p1u}
Let $(s_k)$ and $(\sigma_k)$, $k=0,1,2,\ldots$, be a binomial-transform pair of the first kind. Then, for $k=1,2,3,\ldots$, the sequence $(a_k)$ given by 
\begin{equation}
a_k  =\sum_{j=1}^{k-1}{\left(s_{j-1}+\sigma_{j-1}\right)},
\end{equation}
is a self-inverse sequence while the sequence $(b_k)$, where
\begin{equation}
b_k  =\sum_{j=1}^{k-1}{\left(s_{j-1}-\sigma_{j-1}\right)},
\end{equation}
is an anti-self-inverse sequence.
\end{theorem}
\begin{proof}
A consequence of Theorem~\ref{thm.qpev64c}, using Theorem~\ref{thm.hdq3q7r}.
\end{proof}
\begin{example}
The choice $s_k=x^k$ and $\sigma_k=(1-x)^k$ in Theorem~\ref{thm.uix4p1u} gives the self-inverse sequence $(a_k)$, where
\begin{equation}\label{q6q4k98}
a_k=\frac{1-x^{k-1}}{1-x}+\frac{1-(1-x)^{k-1}}x,\quad k=1,2,3,\ldots,
\end{equation}
and the anti-self-inverse sequence $(b_k)$, where
\begin{equation}\label{r9po5xl}
b_k=\frac{1-x^{k-1}}{1-x}-\frac{1-(1-x)^{k-1}}x,\quad k=1,2,3,\ldots,
\end{equation}
for every complex number $x$.

\end{example}
The next result comes from Theorem~\ref{thm.hn3cm79}.
\begin{theorem}\label{thm.n6j6oyj}
Let $(s_k)$ and $(\sigma_k)$, $k=0,1,2,\ldots$, be a binomial transform pair of the first kind. Then, for $k=0,1,2,3,\ldots$, we have a self-inverse sequence $(a_k)$ given by
\begin{equation}
a_k  = \frac{1}{{\left( {k + 1} \right)\left( {k + 2} \right)}}\sum_{j = 0}^k {\left( {s_j  + \sigma _j } \right)} ,
\end{equation}
and an anti-self-inverse sequence $(b_k)$ given by
\begin{equation}
b_k  = \frac{1}{{\left( {k + 1} \right)\left( {k + 2} \right)}}\sum_{j = 0}^k {\left( {s_j  - \sigma _j } \right)} .
\end{equation}

\end{theorem}

\begin{example}
Using~\eqref{idh7u98} in Theorem~\ref{thm.n6j6oyj}, we establish, for $k=0,1,2,\ldots$, the self-inverse sequence $(a_k)$, where
\begin{equation}
a_k=\frac{k(H_k-1)}{(k+1)(k+2)},
\end{equation}
and the anti-self-inverse sequence $(b_k)$, where
\begin{equation}
b_k=\frac{(k+2)H_k-k}{(k+1)(k+2)}.
\end{equation}
\end{example}

The next result is from Theorem~\ref{thm.spugtzj}.
\begin{theorem}\label{thm.czsuc0u}
Let $(s_k)$ and $(\sigma_k)$, $k=0,1,2,\ldots$, be a binomial transform pair of the first kind. Then, for $k=0,1,2,3,\ldots$, we have a self-inverse sequence $(a_k)$ given by
\begin{equation}
a_k  = \frac1{k+1}\sum_{j = 1}^k {\left( {s_{j-1}  - \sigma _{j-1} } \right)} ,
\end{equation}
and an anti-self-inverse sequence $(b_k)$ given by
\begin{equation}
b_k  = \frac1{k+1}\sum_{j = 1}^k {\left( {s_{j-1}  + \sigma _{j-1} } \right)} .
\end{equation}

\end{theorem}

\begin{corollary}\label{cor.ego8coh}
Let $(s_k)$, $k=0,1,2,\ldots$, be a sequence of complex numbers. Let $(w_k)$ be the sequence given by
\begin{equation}\label{puv0y71}
w_k  = \frac1{k+1}\sum_{j = 0}^{k-1} s_j,\quad k=0,1,2,\ldots 
\end{equation}
If $(s_k)$ is a self-inverse sequence, then $(w_k)$ is an anti-self-inverse sequence while if $(s_k)$ is an anti-self-inverse sequence, then $(w_k)$ is a self-inverse sequence.
\end{corollary}
The implication of Corollary~\ref{cor.ego8coh} is that it is always possible to construct an anti-self-inverse sequence from a given self-inverse sequence and vice versa.
\begin{example}
By choosing $s_k=1/(k+1)$ in~\eqref{puv0y71}, we deduce that the sequence $(w_k)$, where
\begin{equation}
w_k=\frac {H_k}{k+1},\quad k=0,1,2,\ldots,
\end{equation}
is an anti-self-inverse sequence. Similarly, $s_k=H_{k-1+\delta_{k0}}$ yields $(w_k)$, where
\begin{equation}
w_k=\frac{k-1}{k+1}(H_{k-1}-1),\quad k=1,2,3,\ldots,
\end{equation}
as an anti-self-inverse sequence. The choice $s_k=L_k$ in~\eqref{puv0y71} establishes the sequence $(w_k)$ given by
\begin{equation}
w_k=\frac{L_{k+1}-1}{k+1},\quad k=0,1,2,\ldots,
\end{equation}
as an anti-self-inverse sequence. Taking $s_k=F_k$ in~\eqref{puv0y71} produces the sequence $(w_k)$ given by
\begin{equation}\label{yjevokt}
w_k=\frac{F_{k+1}-1}{k+1},\quad k=0,1,2,\ldots,
\end{equation}
as a self-inverse sequence.
\end{example}
\begin{remark}
Since the sequence $(a_k)$, $k=0,1,2,\ldots$, where $a_k=1/(k+1)$ is a self-inverse sequence, $w_k$ given in~\eqref{yjevokt} can simply be replaced with
\begin{equation}
w_k=\frac{F_{k+1}}{k+1},\quad k=0,1,2,\ldots.
\end{equation}
\end{remark}
The next result is an immediate consequence of identity~\eqref{dmpo94c}.
\begin{theorem}
If $(s_k)$, $k=0,1,2,\ldots$, is a self-inverse sequence, then the sequence $(w_k)$, where
\begin{equation}
w_k=2^{-k}\sum_{j=0}^{k-1}\binom kjs_j,
\end{equation}
is an anti-self-inverse sequence, and vice versa.
\end{theorem}

\begin{theorem}\label{thm.j9v7pea}
Let $j$ be an arbitrary non-negative integer. For $k=0,1,2,\ldots$, we have a self-inverse sequence $(a_k)$, where
\begin{equation}\label{acqx937}
a_k=2^{-k}\binom k{j}\left(1+(-1)^j\right),
\end{equation}
and an anti-self-inverse sequence $(b_k)$ given by
\begin{equation}\label{ff4eehj}
b_k==2^{-k}\binom k{j}\left(1-(-1)^j\right).
\end{equation}
\end{theorem}

\begin{proof}
On account of~\eqref{siusv41} on page~\pageref{siusv41}, use
\begin{equation*}
s_k=\binom kj \text{ and } \sigma_k=(-1)^j\delta_{kj},
\end{equation*}
in~\eqref{r2y4bpk} to obtain
\begin{equation}\label{xi5wm2k}
\sum_{k = 0}^n {( - 1)^k \binom{{n}}{k}2^{ - k} \binom{{k}}{j}}  = ( - 1)^j 2^{ - n} \binom{{n}}{j},
\end{equation}
from which~\eqref{acqx937} and~\eqref{ff4eehj} follow, in view of Theorem~\ref{thm.hdq3q7r}.
\end{proof}

\begin{remark}
Since $j$ is an arbitrary integer, the self-inverse sequence $(a_k)$ and the anti-self-inverse sequence $(b_k)$ in Theorem~\ref{thm.j9v7pea} are given more briefly by
\begin{equation}\label{v2huu8h}
a_k=2^{-k}\binom k{2j},\quad b_k=2^{-k}\binom k{2j-1}.
\end{equation}
\end{remark}

Theorem~\ref{thm.xl9sjqb} yields the result stated in the next theorem.
\begin{theorem}
Let $(s_k)$ and $(\sigma_k)$, $k=0,1,2,\ldots$, be a binomial transform pair of the first kind. Then, for $k=1,2,3,\ldots$, we have a self-inverse sequence $(a_k)$ given by
\begin{equation}\label{a48y0a0}
a_k  = k\left(s_k+\sigma_k-\sigma_{k-1}\right) ,
\end{equation}
and an anti-self-inverse sequence $(b_k)$ given by
\begin{equation}\label{r2cb2lo}
b_k  = k\left(s_k-\sigma_k+\sigma_{k-1}\right) .
\end{equation}

\end{theorem}

\begin{corollary}\label{cor.noh1qvl}
If $(s_k)$, $k=1,2,3,\ldots,$ is a self-inverse sequence, then the sequence $(b_k)$ given by
\begin{equation}\label{h4mq163}
b_k=ks_{k-1}
\end{equation}
is an anti-self-inverse sequence.  Similarly, if $(s_k)$, $k=1,2,3,\ldots$, is an anti-self-inverse sequence, then the sequence $(a_k)$ given by
\begin{equation}\label{ydcjlo0}
a_k=ks_{k-1}.
\end{equation}
is a self-inverse sequence.

\end{corollary}
\begin{example}
Choosing $s_k=L_k=\sigma_k$, for $k=0,1,2,\ldots$, in~\eqref{a48y0a0} and~\eqref{r2cb2lo} yields the self-inverse sequence $(a_k)$, where
\begin{equation}
a_k=kF_{k-1},
\end{equation}
and an anti-self-inverse sequence, $(b_k)$, given by
\begin{equation}
b_k=kL_{k-1}.
\end{equation}
The choice $s_k=1/(k+1)=\sigma_k$ in~\eqref{a48y0a0} produces the self-inverse sequence $(a_k)$, where
\begin{equation}
a_k=\frac{k-1}{k+1},\quad k=1,2,3,\ldots
\end{equation}
With $s_k=H_{k-1}$, $k=1,2,3,\ldots$, in~\eqref{a48y0a0} we obtain the anti-self-inverse sequence $(b_k)$, where
\begin{equation}
b_k=kH_{k-2},\quad k=2,3,4,\ldots
\end{equation}
\end{example}

\begin{theorem}
Let $m$, $n$ and $r$ be non-negative integers. Let $(w_k)$, $k=m,m+1,m+2,\ldots$, be a sequence of complex numbers. If $(w_k)$ is a self-inverse sequence, then
\begin{equation}
\sum_{k = r+m}^n {( - 1)^k \binom{{n}}{k}\binom krw_{k - r} }  = ( - 1)^r \binom nrw_{n - r}, 
\end{equation}
while if $(w_k)$ is an anti-self-inverse sequence, then
\begin{equation}
\sum_{k = r+m}^n {( - 1)^k \binom{{n}}{k}\binom krw_{k - r} }  = ( - 1)^{r-1} \binom nrw_{n - r}. 
\end{equation}
Thus, if $(w_k)$ is a self-inverse sequence, then the sequences $(t_k)$ and $(\tau_k)$, $k=r+m,r+m+1,r+m+2,\ldots$, where
\begin{equation}\label{johjnpc}
t_k=\binom krw_{k - r}\text{ and }\tau_k=(-1)^r\binom krw_{k - r},
\end{equation}
are a binomial-transform pair of the first kind. Similarly, if $(w_k)$ is an anti-self-inverse sequence, then the sequences $(t_k)$ and $(\tau_k)$, where
\begin{equation}\label{x9agnym}
t_k=\binom krw_{k - r}\text{ and }\tau_k=(-1)^{r-1}\binom krw_{k - r},
\end{equation}
are a binomial-transform pair of the first kind.

\end{theorem}
\begin{proof}
Repeated application of~\eqref{h4mq163} and~\eqref{ydcjlo0} to the sequence $(w_k)$.
\end{proof}

\begin{example}
From the Fibonacci sequence $(F_k)$, $k=0,1,2,\ldots$, and~\eqref{x9agnym}, we find that the sequences $(t_k)$ and $(\tau_k)$, $k=r,r+1,r+2,\ldots$, where
\begin{equation}
t_k=\binom krF_{k-r}\text{ and }(-1)^{r-1}\binom krF_{k-r},
\end{equation}
are a binomial-transform pair of the first kind.

The harmonic number self-inverse sequence $H_{k-1}$, $k=1,2,3,\ldots$, in conjunction with~\eqref{johjnpc} produces the sequences $(t_k)$ and $(\tau_k)$, $k=r+1,r+2,r+3,\ldots$, where
\begin{equation}
t_k=\binom krH_{k-r-1}\text{ and }\tau_k=(-1)^r\binom krH_{k-r-1},
\end{equation}
as a binomial-transform pair of the first kind.

From the self-inverse sequence $(a_k)$ given in~\eqref{v2huu8h} and the result stated in~\eqref{johjnpc}, we find that the sequences $(t_k)$ and $(\tau_k)$, $k=r,r+1,r+2,\ldots$, given by
\begin{equation}
t_k=\binom{{k}}{r}2^{ - k} \binom{{k - r}}{{2j}}\text{ and }\tau_k=( - 1)^r \binom{{k}}{r}2^{ - k} \binom{{k - r}}{{2j}},
\end{equation}
constitute a binomial-transform pair of the first kind.
\end{example}

\section{Various extensions}
The results in this section extend some of the binomial convolution identities and other identities derived in Sections~\ref{sec.first}--\ref{sec.mixed}.

The result stated in the next theorem extends~\eqref{pop9ybt}.
\begin{theorem}
Let $j$ and $n$ be non-negative integers such that $j\le n$. Let $\{(s_k),(\sigma_k)\}$, $k=0,1,2,\ldots$, be a binomial-transform pair of the first kind. Then
\begin{equation}\label{b0c9iwa}
\sum_{k = 0}^n {( - 1)^k \binom{n - j}k2^{n - k} s_k }  = 2^j \sum_{k = 0}^n {\binom{n - j}k\sigma _k } .
\end{equation}
\end{theorem}

\begin{proof}
By writing~\eqref{s4jiizc} in the equivalent form
\begin{equation*}
\sum_{k = 0}^n {( - 1)^k \binom{{n}}{k}\binom{{k}}{j}2^k }  = ( - 1)^n \binom{{n}}{j}2^j ,
\end{equation*}
we see that $\binom kj2^k$ and $(-1)^k\binom kj2^j$ are a binomial pair of the first kind. Using
\begin{equation}
t_k=\binom kj2^k \text{ and } \tau_k=(-1)^k\binom kj2^j,
\end{equation}
in~\eqref{main1} gives~\eqref{b0c9iwa}.
\end{proof}

\begin{theorem}
Let $j$ and $n$ be non-negative integers. Let $\{(s_k), (\sigma_k)\}$, $k=0,1,2,\ldots$, be a binomial-transform pair of the first kind. Then
\begin{equation}\label{ajkcgco}
\sum_{k = 0}^n {( - 1)^{n - k} \binom{{n}}{k}\binom{{n - k}}{j}2^k s_k }  = \sum_{k = 0}^n {( - 1)^{j - k} \binom{{n}}{k}\binom{{n - k}}{j}2^k \sigma _k } .
\end{equation}

\end{theorem}

\begin{proof}
From~\eqref{xi5wm2k}, identify the binomial-transform pair of the first kind:
\begin{equation*}
t_k  = 2^{ - k} \binom{{k}}{j}\text{ and }\tau _k  = ( - 1)^j 2^{ - k} \binom{{k}}{j},
\end{equation*}
and plug these into~\eqref{main1} to obtain~\eqref{ajkcgco}.
\end{proof}

\begin{corollary}
Let $j$ and $n$ be non-negative integers having different parity. If $(s_k)$ is a self-inverse sequence, that is $\sigma_k=s_k$, $k=0,1,2,\ldots$, then
\begin{equation}
\sum_{k = 0}^n {( - 1)^k \binom{{n}}{k}\binom{{n - k}}{j}2^k s_k }  = 0.
\end{equation}
In particular, if $n$ is an odd integer and $(s_k)$ is a self-inverse sequence, then
\begin{equation}
\sum_{k = 0}^n {( - 1)^k \binom nk2^k s_k }  = 0.
\end{equation}

\end{corollary}

\begin{corollary}
Let $j$ and $n$ be non-negative integers having the same parity. If $(s_k)$ is an anti-self-inverse sequence, that is $\sigma_k=-s_k$, $k=0,1,2,\ldots$, then
\begin{equation}
\sum_{k = 0}^n {( - 1)^k \binom{{n}}{k}\binom{{n - k}}{j}2^k s_k }  = 0.
\end{equation}
In particular, if $n$ is an even integer and $(s_k)$ is an anti-self-inverse sequence, then
\begin{equation}
\sum_{k = 0}^n {( - 1)^k \binom nk2^k s_k }  = 0.
\end{equation}

\end{corollary}

\begin{theorem}
Let $n$ be a non-negative integer. If $\{(\bar s_k), (\bar\sigma_k)\}$ and $\{(\bar t_k), (\bar\tau_k)\}$, $k=0,1,2,\ldots$, are binomial-transform pairs of the second kind, then
\begin{equation}
\sum_{k = 0}^n {\binom{{n}}{k}2^{ - k} \bar s_k \bar \tau _{n - k} }  = \sum_{k = 0}^n {\binom{{n}}{k}2^{ - k} \bar t_{n - k} \sum_{j = 0}^k {\binom{{k}}{j}\bar \sigma _j } } .
\end{equation}

\end{theorem}

\begin{proof}
From~\eqref{pop9ybt}, we can identity the following binomial-transform pair of the second kind:
\begin{equation*}
\bar a_k  = 2^{ - k} \bar s_k \text{ and } \bar\alpha _k  = 2^{ - k}\sum_{j = 0}^k {\binom{{k}}{j}\bar \sigma _j }. 
\end{equation*}
Replace $\bar s_k$ with $\bar a_k$ and $\bar\sigma_k$ with $\bar\alpha_k$ in~\eqref{o5suprg}. This completes the proof.
\end{proof}

\begin{theorem}
Let $n$ be a non-negative integer. If $\{(s_k), (\sigma_k)\}$ and $\{(t_k), (\tau_k)\}$, $k=0,1,2,\ldots$, are binomial-transform pairs of the first kind, then
\begin{equation}
\sum_{k = 1}^n {( - 1)^{n - k - 1} \binom{{n}}{k}s_{k - 1} t_{n - k} }  = \sum_{k = 1}^n {( - 1)^k \binom{{n}}{k}\tau _{n - k} \sum_{j = 1}^k {\sigma _{j - 1} } } .
\end{equation}
\end{theorem}

\begin{proof}
In view of the result stated in~Theorem~\ref{thm.hq03eji}, let
\begin{equation*}
a_k=(1 - \delta _{k0} )s_{k - 1 + \delta _{k0} } \text{ and }\alpha_k= - \sum_{j = 1}^k {\sigma _{j - 1} },\quad k=0,1,2,\ldots .
\end{equation*}
Now, replace $s_k$ with $a_k$ and $\sigma_k$ with $\alpha_k$ in~\eqref{main1}.
\end{proof}
We bring this section to a close by giving another generalization of~\eqref{main1} based on the result from Corollary~\ref{cor.l9mldgr}.
\begin{theorem}
Let $n$ be a non-negative integer. Let $\{(s_k), (\sigma_k)\}$ and $\{(t_k), (\tau_k)\}$, $k=0,1,2,\ldots$, are binomial-transform pairs of the first kind. Let $S_0(k)=\sigma_k$ and $S_m(k)=k(S_{m-1}(k)-S_{m-1}(k-1))$ for every positive integer $m$. Then
\begin{equation}
\sum_{k = 0}^n {( - 1)^{n - k} \binom{{n}}{k}k^m s_k t_{n - k} }  = \sum_{k = 0}^n {( - 1)^k \binom{{n}}{k}S_m (k)\tau _{n - k} } .
\end{equation}
\end{theorem}
In particular,
\begin{equation}
\sum_{k = 1}^n {( - 1)^{n - k} \binom{{n}}{k}k s_k t_{n - k} }  = \sum_{k = 1}^n {( - 1)^k \binom{{n}}{k}k(\sigma_k-\sigma_{k-1})\tau _{n - k} } ,
\end{equation}
with the special value
\begin{equation}
\sum_{k = 1}^n {( - 1)^{n - k - 1} \binom nkt_{n - k} }  = \sum_{k = 1}^n {( - 1)^k \binom nk\tau _{n - k} } ,
\end{equation}
obtained with $s_k=-1/k$ and $\sigma_k=H_k$.
\begin{proof}
Replace $s_k$ with $k^ms_k$ and $\sigma_k$ with $S_m(k)$ in~\eqref{main1}.
\end{proof}
\section{Polynomial identities involving a binomial transform pair}

We have already encountered the identities stated in the next theorem; they are variations on~\eqref{eq.renl1hy} and~\eqref{eq.cr8jcr4}.

\begin{theorem}\label{thm.xzlyuwd}
If $n$ is a non-negative integer and $y$ is a complex variable, then
\begin{align}
\sum_{k = 0}^n { \binom{{n}}{k}s_{n - k}y^k  }  = \sum_{k = 0}^n {( - 1)^{n-k} \binom{{n}}{k}\sigma _{n - k}\left( {1 + y} \right)^k }\label{eq.t182qp7},\\
\sum_{k = 0}^n {( - 1)^k \binom{{n}}{k}\bar s_{n - k}y^k  }  = \sum_{k = 0}^n {( - 1)^k \binom{{n}}{k}\bar\sigma _{n - k}\left( {1 + y} \right)^k } \label{eq.l6e3qzc}.
\end{align}
\end{theorem}

The implication of Theorem~\ref{thm.xzlyuwd} is that every binomial-transform pair has associated with it a polynomial identity.

\begin{remark}
Identity~\eqref{eq.l6e3qzc} is equivalent to Boyadzhiev's~\cite[Corollary 1]{boyadzhiev16}.
\end{remark}

\section{Binomial transform identities associated with polynomial identities of a certain type}
Polynomial identities of the following form abound in the literature:
\begin{equation}\label{poly}
\sum_{k = 0}^n {f(k)t^{p(k)} }  = \sum_{k = 0}^m {g(k)\left( {1 - t} \right)^{q(k)} } ;
\end{equation}
where $m$ and $n$ are non-negative integers, $t$ is a complex variable, $p(k)$ and $q(k)$ are sequences of integers, and $f(k)$ and $g(k)$ are sequences of complex numbers.

The binomial identities stated in Theorems~\ref{thm.qs3j4nk} and~\ref{thm.ru7y30f} are readily derived using the operators $\mathcal L_y$ and $\mathcal M_y$.
\begin{theorem}\label{thm.qs3j4nk}
Let a polynomial identity have the form stated in~\eqref{poly}. If $\{(s_k), (\sigma_k)\}$ is a binomial-transform pair of the first kind, then
\begin{equation}
\sum_{k = 0}^n {f(k)s_{p(k)} }  = \sum_{k = 0}^m {g(k)\sigma _{q(k)} } .
\end{equation}

\end{theorem}

\begin{theorem}\label{thm.ru7y30f}
Let a polynomial identity have the form stated in~\eqref{poly}. If $\{(\bar s_k), (\bar\sigma_k)\}$ is a binomial-transform pair of the second kind, then
\begin{equation}
\sum_{k = 0}^n {( - 1)^{p(k)} f(k)\bar s_{p(k)} }  = \sum_{k = 0}^m {g(k)\bar \sigma _{q(k)} } .
\end{equation}

\end{theorem}

\begin{lemma}{Sun~\cite[Lemma 3.1]{sun03}}
If $m$, $n$ and $r$ are non-negative integers and $t$ is a complex variable, then
\begin{equation}\label{lpj4ylb}
\sum_{k = 0}^n {( - 1)^{k - r} \binom{{n}}{k}\binom{{k + m}}{r}t^{k + m - r} }  = \sum_{k = 0}^m {(-1)^k\binom{{m}}{k}\binom{{k + n}}{r}\left( {1 - t} \right)^{n + k - r} } .
\end{equation}
\end{lemma}

\begin{example}
In~\eqref{lpj4ylb} we can identify
\begin{equation}
f(k) = ( - 1)^{k - r} \binom{{n}}{k}\binom{{k + m}}{r},\quad g(k) = \binom{{m}}{k}\binom{{k + n}}{r},
\end{equation}
and 
\begin{equation}
p(k) = k + m - r \mbox{ and } q(k) = n + k - r.
\end{equation}
Using these in Theorem~\ref{thm.qs3j4nk} gives
\begin{equation}\label{k0gxz3k}
\sum_{k = 0}^n {( - 1)^{k - r} \binom{{n}}{k}\binom{{k + m}}{r}s_{k + m - r} }  = \sum_{k = 0}^m {(-1)^k\binom{{m}}{k}\binom{{k + n}}{r}\sigma_{n + k - r} },
\end{equation}
for every binomial-transform pair $\{(s_k),(\sigma_k)\}$ of the first kind.
\end{example}

\begin{remark}
Identity~\eqref{k0gxz3k} is Chen's Theorem 3.2~\cite{chen07}.
\end{remark}

\section{Conversion between binomial transform identities}
If a binomial transform identity involves a binomial-transform pair of the first kind, it is straightforward to convert it to an identity involving a binomial-transform pair of the second kind; and vice versa. An identity involving $\{ s(k),\sigma(k)\}$, $k=0,1,2,\dots$, can be converted to that involving $\{ \bar s(k),\bar\sigma(k)\}$ by choosing $s(k)=x^k$ and $\sigma(k)=(1-x)^k$, replacing $x$ with $-x$ and then operating on the resulting identity with $\mathcal M_x$. Similarly, an identity involving $\{ \bar s(k),\bar\sigma(k)\}$ can be converted to that involving $\{ s(k),\sigma(k)\}$ by choosing $\bar s(k)=x^k$ and $\bar\sigma(k)=(1+x)^k$, replacing $x$ with $-x$ and then operating on the resulting identity with $\mathcal L_x$.

\begin{example}
Let $\{(\bar t_k),(\bar\tau_k)\}$, $k=0,1,2,\ldots$,  be a binomial-transform pair of the second kind. Chen's main result~\cite[Theorem 3.1]{chen07} (after  Gould and Quaintance's simplification~\cite[Equation (10)]{gould14}) for non-negative integers $m$, $n$ and $s$ is
\begin{align}\label{eq.ollq3g3}
&\sum_{k = 0}^m {\binom{{m}}{k}\binom{{n + k + s}}{s}^{ - 1} \bar t_{n + k + s} }\nonumber\\
&\qquad= \sum_{k = 0}^n {( - 1)^{n - k} \binom{{n}}{k}\binom{{m + k + s}}{s}^{ - 1} \bar \tau _{m + k + s} }\nonumber\\
&\qquad\qquad  + \sum_{k = 0}^{s - 1} {\frac{{( - 1)^{n + s - k} s}}{{m + n + s - k}}\binom{{s - 1}}{k}\binom{{m + n + s - k - 1}}{n}^{ - 1} \bar \tau _k } .
\end{align}
In order to convert~\eqref{eq.ollq3g3} to an identity for the binomial-transform pair of the first kind, $\{(t_k),(\tau_k)\}$, $k=0,1,2,\ldots$, we choose $\bar t_k=x^k$ and $\bar\tau_k=(1+x)^k$ and replace $x$ with $-x$ to obtain
\begin{align*}
&\sum_{k = 0}^m {(-1)^k\binom{{m}}{k}\binom{{n + k + s}}{s}^{ - 1} x^{n + k + s} }\nonumber\\
&\qquad= (-1)^s\sum_{k = 0}^n {( - 1)^k \binom{{n}}{k}\binom{{m + k + s}}{s}^{ - 1} (1 - x)^{m + k + s} }\nonumber\\
&\qquad\qquad  + \sum_{k = 0}^{s - 1} {\frac{{( - 1)^k s}}{{m + n + s - k}}\binom{{s - 1}}{k}\binom{{m + n + s - k - 1}}{n}^{ - 1} (1 - x)^k } .
\end{align*}
Operating on the above equation with the linear operator $\mathcal L_x$ now gives
\begin{align}\label{eq.coax3mm}
&\sum_{k = 0}^m {(-1)^k\binom{{m}}{k}\binom{{n + k + s}}{s}^{ - 1} t_{n + k + s} }\nonumber\\
&\qquad= (-1)^s\sum_{k = 0}^n {( - 1)^k \binom{{n}}{k}\binom{{m + k + s}}{s}^{ - 1} \tau_{m + k + s} }\nonumber\\
&\qquad\qquad  + \sum_{k = 0}^{s - 1} {\frac{{( - 1)^k s}}{{m + n + s - k}}\binom{{s - 1}}{k}\binom{{m + n + s - k - 1}}{n}^{ - 1} \tau_k } .
\end{align}

\end{example}

\begin{example}
Another example, Chen~\cite[Theorem 3.2]{chen07} is
\begin{equation}
\sum_{k = 0}^m {\binom{{m}}{k}\binom{{n + k}}{s}\bar t_{n + k - s} }  = \sum_{k = 0}^n {( - 1)^{n - k}\binom{{n}}{k}\binom{{m + k}}{s} \bar \tau _{m + k - s} } ;
\end{equation}
whose corresponding version for a binomial-transform pair of the first kind is
\begin{equation}
\sum_{k = 0}^m {(-1)^{s-k}\binom{{m}}{k}\binom{{n + k}}{s} t_{n + k - s} }  = \sum_{k = 0}^n {( - 1)^k\binom{{n}}{k}\binom{{m + k}}{s} \tau _{m + k - s} } .
\end{equation}

\end{example}

\begin{example}
Gould and Quaintance~\cite[Theorem 3]{gould14} gave the following identity involving a binomial transform pair of the second kind:
\begin{align}
&\sum_{k = 0}^s {\binom{{s}}{k}\binom{{m + n + s - k}}{m}^{ - 1} \frac{{\bar t_k }}{{m + n + s + 1 - k}}}\nonumber \\
&\qquad = \sum_{k = 0}^s {\binom{{s}}{k}\binom{{m + n + s - k}}{n}^{ - 1} \frac{{( - 1)^{s - k}\bar \tau_k }}{{m + n + s + 1 - k}}} ,
\end{align}
which is valid for every non-negative integer $s$ and all complex numbers $m$ and $n$ excluding the set of negative integers. 

The corresponding result for a binomial-transform pair of the first kind is
\begin{align}
&\sum_{k = 0}^s {(-1)^k\binom{{s}}{k}\binom{{m + n + s - k}}{m}^{ - 1} \frac{{t_k }}{{m + n + s + 1 - k}}}\nonumber \\
&\qquad = \sum_{k = 0}^s {\binom{{s}}{k}\binom{{m + n + s - k}}{n}^{ - 1} \frac{{( - 1)^{s - k}\tau_k }}{{m + n + s + 1 - k}}} .
\end{align}
\end{example}

\begin{remark}\label{rem.r7dv221}
In view of the foregoing, an identity involving a binomial-transform pair of the first kind $\{s_k,\sigma_k\}$, \mbox{$k=0,1,2,\ldots$}, can always be converted to a binomial-transform pair identity of the second kind $\{\bar s_k,\bar\sigma_k\}$ by doing 
\begin{equation*}
s_k\to(-1)^k\bar s_k,\quad\sigma_k\to\bar\sigma_k.
\end{equation*}
Similarly, conversion of a second kind transform pair identity to a first kind transform pair is achieved through
\begin{equation*}
\bar s_k\to(-1)^ks_k,\quad\bar\sigma_k\to\sigma_k.
\end{equation*}
\end{remark}

\section{A table of binomial-transform pairs}\label{sec.table}
Many identities involving binomial transforms are known in the literature; ready examples can be found in the books by Gould~\cite{gould} and Boyadzhiev~\cite{boyadzhiev16} and in the numerous references therein. In this section we add to the growing list of binomial-transform pairs.

\begin{lemma}\label{lem.j5d7kdb}
If $r$ is a non-negative integer, then
\begin{equation}\label{hiigp9e}
\int_0^1 {\left( {1 + x} \right)^r \ln \left( {1 + x} \right)dx}  = \frac{{1 - 2^{r + 1} }}{{\left( {r + 1} \right)^2 }} + \frac{{2^{r + 1} }}{{r + 1}}\ln 2
\end{equation}
and
\begin{equation}\label{mhlgo8i}
\int_0^1 {x^r \ln \left( {1 + x} \right)dx}
=\frac1{r+1} 
\begin{cases}
 H_{r + 1}  - H_{(r + 1)/2} ,&\text{if $r$ is odd;} \\ 
 H_{r + 1}  - 2\,O_{(r + 2)/2}  + 2\ln 2,&\text{if $r$ is even.} \\ 
 \end{cases} 
\end{equation}
\end{lemma}
\begin{proof}
Identity~\eqref{hiigp9e} is obvious. A simple change of variable in the known integral
\begin{equation}
\int_0^1 {x^r \ln \left( {1 - x} \right)dx} =  - \frac{{H_{r + 1} }}{{r + 1}}
\end{equation}
gives
\begin{equation}
\int_0^1 {x^r \ln \left( {1 + x} \right)dx}  = \frac{{H_{r + 1}  - H_{(r + 1)/2} }}{{r + 1}},\quad r\in\mathbb C\setminus\mathbb Z^{-},
\end{equation}
of which~\eqref{mhlgo8i} is a special case, since
\begin{equation*}
H_{k+1/2}=2\,O_{k+1}-2\ln 2.
\end{equation*}
\end{proof}

\begin{lemma}\label{btransform}
Let $(a_k)$ and $(\alpha_k)$, $k=0,1,2,\ldots$, be a binomial-transform pair of the first kind. Let $\mathcal N_k$ be the linear operator defined by
\begin{equation*}
\mathcal N_k (a_k ) = \sum_{j = 0}^k {( - 1)^j \binom{{k}}{j}a_j }  = \alpha _k .
\end{equation*}
Then
\begin{align}
\mathcal N_k (\alpha _k ) = a_k,\\ 
\mathcal N_k^2 (\alpha _k ) = \alpha _k, \\
\mathcal N_k^2 (a _k ) = a _k, 
\end{align}
with the inverse relations
\begin{equation}
\mathcal N_k^{-1} (a _k ) = \alpha _k
\end{equation}
and
\begin{equation}
\mathcal N_k^{-1} (\alpha _k ) = a _k .
\end{equation}

\end{lemma}

\begin{proposition}
If $n$ is a non-negative integer, then
\begin{equation}\label{t7montf}
\sum_{k = 0}^n {( - 1)^k \binom{{n}}{k}\frac{{2^k }}{{\left( {k + 1} \right)^2 }}}  = \frac{{O_{\left\lceil {\left( {n + 1} \right)/2} \right\rceil } }}{{n + 1}}
\end{equation}
and
\begin{align}\label{xr1j6uv}
\sum_{k = 0}^n {( - 1)^k \binom{{n}}{k}\frac{{2^k }}{{\left( {k + 1} \right)^3 }}}  &= \frac{1}{4}\frac{{H_{n + 1}^2  + H_{n + 1}^{(2)} }}{{n + 1}} - \frac{1}{8}\frac{{H_{\left\lceil {n/2} \right\rceil }^2  + H_{\left\lceil {n/2} \right\rceil }^{(2)} }}{{n + 1}} + \frac{1}{2}\frac{{O_{\left\lceil {(n + 1)/2} \right\rceil }^2  + O_{\left\lceil {(n + 1)/2} \right\rceil }^{(2)} }}{{n + 1}}\nonumber\\
&\qquad - \frac{1}{2}\frac{1}{{n + 1}}\sum_{k = 0}^n {( - 1)^k \frac{{H_{k + 1} }}{{k + 1}}} .
\end{align}
\end{proposition}

\begin{proof}
Write~\eqref{vz664g0} as
\begin{equation*}
\sum_{k = 0}^n {( - 1)^k \binom{{n + 1}}{{k + 1}}\left( {1 + x} \right)^k }  = \sum_{k = 0}^{\left\lfloor {n/2} \right\rfloor } {x^{2k} }  - \sum_{k = 1}^{\left\lceil {n/2} \right\rceil } {x^{2k - 1} } ,
\end{equation*}
and termwise integrate from $0$ to $1$, using Lemma~\ref{lem.j5d7kdb}, to obtain
\begin{align*}
&\sum_{k = 0}^n {( - 1)^k \binom{{n + 1}}{{k + 1}}\frac{{1 - 2^{k + 1} }}{{\left( {k + 1} \right)^2 }}}  + \ln 2\sum_{k = 0}^n {( - 1)^k \binom{{n + 1}}{{k + 1}}\frac{{2^{k + 1} }}{{k + 1}}} \\
&\qquad = \sum_{k = 0}^{\left\lfloor {n/2} \right\rfloor } {\frac{{H_{2k + 1}  - 2O_{k + 1}  + 2\ln 2}}{{2k + 1}}}  - \sum_{k = 1}^{\left\lceil {n/2} \right\rceil } {\frac{{H_{2k}  - H_k }}{{2k}}} 
\end{align*}
Comparing rational and irrational parts leads to
\begin{equation}\label{uet8p9x}
\sum_{k = 0}^n {( - 1)^k \binom{{n + 1}}{{k + 1}}\frac{{2^k }}{{k + 1}}}  = \sum_{k = 0}^{\left\lfloor {n/2} \right\rfloor } {\frac{1}{{2k + 1}}}  = O_{\left\lfloor {(n + 2)/2} \right\rfloor } 
\end{equation}
and
\begin{equation}\label{v5ozf0f}
\sum_{k = 0}^n {( - 1)^k \binom{{n + 1}}{{k + 1}}\frac{{1 - 2^{k + 1} }}{{\left( {k + 1} \right)^2 }}}  = \sum_{k = 0}^{\left\lfloor {n/2} \right\rfloor } {\frac{{H_{2k + 1}  - 2O_{k + 1} }}{{2k + 1}}}  - \sum_{k = 1}^{\left\lceil {n/2} \right\rceil } {\frac{{H_{2k}  - H_k }}{{2k}}} .
\end{equation}
Identity~\eqref{uet8p9x} immediately gives~\eqref{t7montf} since
\begin{equation}
\sum_{k = 0}^n {( - 1)^k \binom{{n}}{k}\frac{1}{{\left( {k + 1} \right)^3 }}}  = \frac{1}{2}\frac{{H_{n + 1}^2  + H_{n + 1}^{(2)} }}{{n + 1}},
\end{equation}
\begin{equation}
\sum_{k = 1}^n {\frac{{H_k }}{k}} = \frac{1}{2}\left( {H_n^2  + H_n^{(2)} } \right),
\end{equation}
and
\begin{equation}
\sum\limits_{k = 1}^n {\frac{{O_k }}{{2k - 1}}}  = \frac{1}{2}\left( {O_n^2  + O_n^{(2)} } \right),
\end{equation}
identity~\eqref{v5ozf0f} reduces to
\begin{align*}
&\frac{1}{2}\left( {H_{n + 1}^2  + H_{n + 1}^{(2)} } \right) - 2(n + 1)\sum_{k = 0}^n {( - 1)^k \binom{{n}}{k}\frac{{2^k }}{{\left( {k + 1} \right)^3 }}} \\
&\qquad= \sum\limits_{k = 0}^n {( - 1)^k \frac{{H_{k + 1} }}{{k + 1}}}  + \frac{1}{4}\left( {H_{\left\lceil {n/2} \right\rceil }^2  + H_{\left\lceil {n/2} \right\rceil }^{(2)} } \right) - O_{\left\lceil {(n + 1)/2} \right\rceil }^2  - O_{\left\lceil {(n + 1)/2} \right\rceil }^{(2)},
\end{align*}
and hence~\eqref{xr1j6uv}.
\end{proof}

\begin{proposition}
If $n$ is a non-negative integer, then
\begin{align}
\sum_{k = 1}^n {( - 1)^{k - 1} \binom{{n}}{k}\frac{{2^k }}{{k^2 }}}  &= \frac{1}{2}\left( {H_n^2  + H_n^{(2)} } \right) - \frac{1}{4}\left( {H_{\left\lfloor {n/2} \right\rfloor }^2  + H_{\left\lfloor {n/2} \right\rfloor }^{(2)} } \right) + O_{\left\lceil {n/2} \right\rceil }^2  + O_{\left\lceil {n/2} \right\rceil }^{(2)}\nonumber\\
&\qquad  - \sum_{k = 1}^n {( - 1)^{k-1} \frac{{H_k }}{k}} .
\end{align}
\end{proposition}
\begin{proof}
Application of Theorem~\ref{thm.v41e7ah} to~\eqref{xr1j6uv}.
\end{proof}
\begin{proposition}
If $n$ is a non-negative integer, then
\begin{equation}\label{knuth}
\sum_{k = 0}^n {( - 1)^k \binom{{n}}{k}2^{ - k} \binom{{2k}}{k}}
 =  \begin{cases}
 2^{ - n} \binom{{n}}{n/2}, &\text{if $n$ is even;}\\ 
 0,&\text{if $n$ is odd;} \\ 
 \end{cases} 
\end{equation}

\begin{equation}\label{n2x6xs0}
\sum_{k = 0}^n {( - 1)^k \binom{{n}}{k}\frac{{2^k }}{{k + 1}}} 
=\begin{cases}
 \dfrac1{n + 1},&\text{if $n$ is even;}  \\ 
 0,&\text{if $n$ is odd;}  \\ 
 \end{cases} 
\end{equation}
and
\begin{equation}\label{pn5otn6}
\sum_{k = 0}^n {( - 1)^k \binom{{n}}{k}2^{ - k} C_{k + 1} } 
=\begin{cases}
 2^{ - n} C_{n/2},&\text{if $n$ is even;}  \\ 
 0,&\text{if $n$ is odd.} \\ 
 \end{cases}
\end{equation}
\end{proposition}

\begin{proof}
The following identity is known~\cite{adegoke1}:
\begin{equation}\label{hdj69wz}
\sum_{k = 0}^n {( - 1)^k \binom{{n}}{k}2^{ - k} \binom{{2k + v}}{{\left( {2k + v} \right)/2}}\binom{{k + v}}{{v/2}}^{ - 1} }  
=  \begin{cases}
 2^{ - n} \binom{{n}}{{n/2}}\binom{{\left( {n + v} \right)/2}}{{v/2}}^{ - 1},&\text{if $n$ is even;}  \\ 
 0,&\text{if $n$ is odd;} \\ 
 \end{cases} 
\end{equation}
where $v$ is a real number.

Identities~\eqref{knuth},~\eqref{n2x6xs0} and~\eqref{pn5otn6} are respective evaluations of~\eqref{hdj69wz} at $v=0$, $v=1$ and $v=2$. Note that
\begin{align}
\binom{{2k + 1}}{k} = \frac{{2k + 1}}{{k + 1}}\binom{{2k}}{k} &= \left( {2k + 1} \right)C_k ,\label{fwus8x3}\\
\binom{{2k + 1}}{{\left( {2k + 1} \right)/2}} &= \binom{{2k + 1}}{k}^{ - 1} \frac{{2^{2(2k + 1)} }}{{\pi (k + 1)}},\\
\binom{{k + 1}}{{1/2}} = \binom{{2\left( {k + 1} \right)}}{{k + 1}}^{ - 1} \frac{{2^{2(k + 1) + 1} }}{\pi } &= \frac{{2^{2(k + 1) + 1} }}{{2\pi (2k + 1)C_k }},
\end{align}
\begin{equation}\label{b9ic6lb}
\binom{{n/2 + 1/2}}{{1/2}} = \binom{{n/2 + 1/2}}{{n/2}} = \frac{{n + 1}}{{2^n }}\binom{{n}}{{n/2}},
\end{equation}
and
\begin{equation}\label{e58lg2y}
\binom{{2(k + 1)}}{{k + 1}} = 2(2k + 1)C_k .
\end{equation}
\end{proof}

\begin{remark}
Generalized binomial coefficient identities~\eqref{fwus8x3}--\eqref{b9ic6lb} and similar ones can be derived from their definition~\eqref{y89d722} on page~\pageref{y89d722}; many of these are listed toward the end of Gould's book~\cite{gould}.
\end{remark}

\begin{remark}
Identity~\eqref{pn5otn6} was first discovered by Suleiman and Sury~\cite{sulei23}.
\end{remark}

\begin{proposition}
If $n$ is a non-negative integer, then
\begin{equation}\label{ey4kpm3}
\sum_{k = 0}^n {( - 1)^k \binom{{n}}{k}2^{ - k} \binom{{2k}}{k}\left( {H_k  - O_k } \right)}
 = \begin{cases}
 2^{ - n - 1} \binom{{n}}{{n/2}}H_{n/2},&\text{if $n$ is even;}  \\ 
 0,&\text{if $n$ is odd;} \\ 
 \end{cases} 
\end{equation}
\begin{equation}\label{lfr23eo}
\sum_{k = 0}^n {( - 1)^k \binom{{n}}{k}\frac{{2^k H_k }}{{k + 1}}}
=\begin{cases}
 0,&\text{if $n$ is even;} \\ 
  - \dfrac{{2O_{(n + 1)/2} }}{{n + 1}},&\text{if $n$ is odd;} \\ 
 \end{cases} 
\end{equation}
and
\begin{equation}\label{bv4ikdf}
\sum_{k = 0}^n {( - 1)^k \binom{{n}}{k}2^{ - k} C_{k + 1} \left( {O_{k + 1}  - H_{k + 2} } \right)}
=\begin{cases}
  - 2^{ - n - 1} C_{n/2} H_{(n + 2)/2},&\text{if $n$ is even;}  \\ 
 0,&\text{if $n$ is odd.} \\ 
 \end{cases}
\end{equation}

\end{proposition}
\begin{proof}
Differentiate~\eqref{hdj69wz} with respect to $v$ to obtain
\begin{align}\label{b8yv715}
&\sum_{k = 0}^n {( - 1)^k \binom{{n}}{k}2^{ - k} \binom{{2k + v}}{{(2k + v)/2}}\binom{{k + v}}{{v/2}}^{ - 1} \left( {2\left( {H_{2k + v}  - H_{k + v} } \right) + H_{v/2}  - H_{k + v/2} } \right)}\nonumber\\ 
&\qquad =  \begin{cases}
  - 2^{ - n} \binom{{n}}{{n/2}}\binom{{(n + v)/2}}{{v/2}}^{ - 1} \left( {H_{(n + v)/2}  - H_{v/2} } \right) ,&\text{if $n$ is even;}\\ 
 0,&\text{if $n$ is odd.} 
 \end{cases}
\end{align}
Setting $v=0$ in~\eqref{b8yv715} gives~\eqref{ey4kpm3}. To prove~\eqref{lfr23eo}, set $v=1$ in~\eqref{b8yv715} and use identities~\eqref{fwus8x3}--\eqref{b9ic6lb} and the fact that
\begin{equation*}
H_{1/2}  - H_{k + 1/2}  = 2 - 2\ln 2 - (2O_{k + 1}  - 2\ln 2) = 2(1 - O_{k + 1} ),
\end{equation*}
to obtain
\begin{align*}
&\sum_{k = 0}^n {( - 1)^k \binom{{n}}{k}\frac{2^k}{k+1} \left(H_{2k + 1}  - H_{k + 1}  - O_{k + 1} \right)} + \sum_{k = 0}^n {( - 1)^k \binom{{n}}{k}\frac{{2^k }}{{k + 1}}}\\  
&\qquad=  \begin{cases}
 \dfrac{{1 - O_{(n + 2)/2} }}{{n + 1}},&\text{if $n$ is even;} \\ 
 0,&\text{if $n$ is odd;} \\ 
 \end{cases}
\end{align*}
which, on account of~\eqref{n2x6xs0} and the fact that
\begin{equation*}
H_{2k+1}-O_{k+1}-H_{k+1}=\frac12H_k-H_{k+1}=-\frac12H_k-\frac 1{k+1}
\end{equation*}
simplifies to
\begin{equation*}
\sum_{k = 0}^n {( - 1)^k \binom{{n}}{k}\frac{{2^k }}{{k + 1}}\left( { - \frac{{H_k }}{2} - \frac{1}{{k + 1}}} \right)} 
 =\begin{cases}
 \dfrac{{ - O_{(n + 2)/2} }}{{n + 1}},&\text{if $n$ is even;} \\ 
 0,&\text{if $n$ is odd;} \\ 
 \end{cases}
\end{equation*}
that is
\begin{equation*}
\sum_{k = 0}^n {( - 1)^k \binom{{n}}{k}\frac{{2^{k - 1} H_k }}{{k + 1}}}  + \sum_{k = 0}^n {( - 1)^k \binom{{n}}{k}\frac{{2^k }}{{\left( {k + 1} \right)^2 }}} 
=\begin{cases}
 \dfrac{{O_{(n + 2)/2} }}{{n + 1}},&\text{if $n$ is even;} \\ 
 0,&\text{if $n$ is odd;} 
 \end{cases}
\end{equation*}
from which, together with~\eqref{t7montf}, we finally obtain~\eqref{lfr23eo}. Identity~\eqref{bv4ikdf} follows from setting $v=2$ in~\eqref{b8yv715}.
\end{proof}
\begin{remark}
Identity~\eqref{ey4kpm3} corresponds to a simplification of~\cite[Equation (8)]{rathie22}. Identity~\eqref{lfr23eo} was first derived by Rathie and Campbell~\cite{rathie22}.
\end{remark}
\begin{proposition}
If $n$ is a non-negative integer, then
\begin{align}
&\sum_{k = 0}^n {( - 1)^k \binom{{n}}{k}\frac{{2k + 1}}{{k + 1}}\,2^{ - k} C_k \left( {H_{k + 1}  - O_{k + 1} } \right)}\\  
&\qquad=\begin{cases}
 0,&\text{if $n$ is even;} \\ 
  - 2^{ - n - 1} \binom{{n}}{{(n - 1)/2}}\dfrac{H_{(n + 1)/2}}{n+1},&\text{if $n$ is odd;}  \\ 
 \end{cases}
\end{align}
and
\begin{align}
\sum_{k = 0}^n {( - 1)^k \binom{{n}}{k}\frac{{\left( {2k + 1} \right)\left( {2k + 3} \right)}}{{\left( {k + 1} \right)\left( {k + 2} \right)^2 }}\,2^{ - k} C_k \left( {H_{k + 2}  - O_{k + 2} } \right)}\nonumber\\
=\begin{cases}
 2^{ - n - 2} \binom{{n + 1}}{{n/2}}\dfrac{{H_{(n + 2)/2} }}{{\left( {n + 1} \right)\left( {n + 2} \right)}} ,&\text{if $n$ is even;}\\ 
 0,&\text{if $n$ is odd.} \\ 
 \end{cases}
\end{align}
\end{proposition}
\begin{proof}
Application of Theorem~\ref{thm.u5nhiva} to~\eqref{ey4kpm3}. We used~\eqref{e58lg2y} and the Catalan number recurrence relation
\begin{equation*}
C_{n + 1}  = \frac{{2\left( {2n + 1} \right)}}{{n + 2}}\,C_n.
\end{equation*}
\end{proof}
The next results come from the Knuth-Boyadzhiev identity~\cite{boyadzhiev09}:
\begin{equation}\label{kb}
\sum_{k = 0}^n {( - 1)^k \binom{{n}}{k}H_k x^k }  = \left( {1 - x} \right)^n H_n  - \sum_{k = 1}^n {\frac{{\left( {1 - x} \right)^{n - k} }}{k}} 
\end{equation}
and the following identity due to Bat\i r and Sofo~\cite{batir21}:
\begin{equation}\label{batir23}
\sum_{k = 0}^n {( - 1)^k \binom{{n}}{k}\bar H_k x^k }  = \sum_{k = 1}^n {\frac{{\left( {1 - x} \right)^{n - k} }}{k}}  - \sum_{k = 1}^n {\frac{{\left( {1 - x} \right)^{n - k} \left( {1 + x} \right)^k }}{k}} ,
\end{equation}
where
\begin{equation*}
\bar H_n=\sum_{k=1}^n{\frac{(-1)^{k-1}}k} 
\end{equation*}
is the $n^{th}$ alternating harmonic number.

It is known that (see, for example, Frontczak~\cite{frontczak21}):
\begin{equation}\label{sev5sny}
\bar H_n  = 2O_{\left\lceil {n/2} \right\rceil }  - H_n,
\end{equation}
and also (Sofo~\cite{sofo15}):
\begin{equation}\label{vtufsdp}
\bar H_n=H_n-H_{\lfloor n/2\rfloor},
\end{equation}
so that
\begin{equation}\label{qxhpp6g}
2\bar H_n=2O_{\lceil n/2\rceil}-H_{\lfloor n/2\rfloor}.
\end{equation}
From~\eqref{kb} and \eqref{batir23}, using~\eqref{sev5sny},~\eqref{vtufsdp}, and~\eqref{qxhpp6g}, we find
\begin{equation}\label{comp1}
2\sum_{k = 0}^n {( - 1)^k \binom{{n}}{k}O_{\left\lceil {k/2} \right\rceil } x^k }  = \left( {1 - x} \right)^n H_n  - \sum_{k = 1}^n {\frac{{\left( {1 - x} \right)^{n - k} \left( {1 + x} \right)^k }}{k}}
\end{equation}
and
\begin{equation}\label{comp2}
\sum_{k = 0}^n {( - 1)^k \binom{{n}}{k}H_{\left\lfloor {k/2} \right\rfloor } x^k }  = \left( {1 - x} \right)^n H_n  - 2\sum_{k = 1}^n {\frac{{\left( {1 - x} \right)^{n - k} }}{k}}  + \sum_{k = 1}^n {\frac{{\left( {1 - x} \right)^{n - k} \left( {1 + x} \right)^k }}{k}} .
\end{equation}

\begin{proposition}
If $n$ is a non-negative integer, then
\begin{align}
2\sum_{k = 0}^n {( - 1)^k \binom{{n}}{k}O_{\left\lceil {k/2} \right\rceil } }  =  - \frac{{2^n }}{n},\label{kuotslw}\\
2\sum_{k = 0}^n {\binom{{n}}{k}O_{\left\lceil {k/2} \right\rceil } }  = 2^n H_n\label{n6bsu21} ,
\end{align}
and
\begin{equation}\label{efxtjr4}
\sum_{k = 1}^n {( - 1)^k \binom{{n}}{k}2^{k - 1} H_{k - 1} }  
=\begin{cases}
 2\,O_{n/2},&\text{if $n$ is even;}  \\ 
 0,&\text{if $n$ is odd.} 
 \end{cases}
\end{equation}
\end{proposition}
\begin{proof}
Using $x=1$ and $x=-1$ in~\eqref{comp1}, in turn, produces~\eqref{kuotslw} and~\eqref{n6bsu21}. Identity~\eqref{efxtjr4} is the result of adding~\eqref{kuotslw} and~\eqref{n6bsu21}.
\end{proof}

\begin{proposition}
If $n$ is a non-negative integer, then
\begin{equation}\label{watbah3}
\sum_{k = 0}^n {\binom{{n}}{k}\frac{{O_{\left\lceil {(k + 1)/2} \right\rceil } }}{{k + 1}}}  = \frac{{2^n H_{n + 1} }}{{n + 1}},
\end{equation}
and
\begin{equation}\label{zphhcbp}
\sum_{k = 0}^n {( - 1)^k \binom{{n}}{k}\frac{{2^k H_{k + 1} }}{{\left( {k + 1} \right)\left( {k + 2} \right)}}}  
=\begin{cases}
 \dfrac{{O_{(n + 2)/2} }}{{\left( {n + 1} \right)\left( {n + 2} \right)}} ,&\text{if $n$ is even;}\\ 
 0,&\text{if $n$ is odd.} \\ 
 \end{cases}
\end{equation}

\end{proposition}
\begin{proof}
Identity~\eqref{watbah3} is obtained from~\eqref{n6bsu21} by the application of Theorem~\ref{thm.u5nhiva}. Identity~\eqref{zphhcbp} follows from~\eqref{efxtjr4} by the same theorem.
\end{proof}
\begin{remark}
Identity~\eqref{lfr23eo} can also be derived from~\eqref{efxtjr4} by the application of Theorem~\ref{thm.u5nhiva}.
\end{remark}
\begin{proposition}
If $n$ is a non-negative integer, then
\begin{equation}\label{kjbgzo0}
\sum_{k = 0}^n {( - 1)^k \binom{{n}}{k}2^k H_{k + 1} }  = ( - 1)^n \left( {O_{\left\lfloor {(n + 1)/2} \right\rfloor }  + O_{\left\lceil {(n + 1)/2} \right\rceil } } \right)
\end{equation}
and
\begin{equation}\label{hv15wkp}
\sum_{k = 0}^n {( - 1)^k \binom{{n}}{k}k2^k H_k }  = ( - 1)^n 2n\left( {O_{\left\lfloor {n/2} \right\rfloor }  + O_{\left\lceil {n/2} \right\rceil } } \right).
\end{equation}
\end{proposition}

\begin{proof}
The binomial inversion of~\eqref{n6bsu21} gives
\begin{equation}\label{ruulewa}
\sum_{k = 0}^n {( - 1)^k \binom{{n}}{k}2^k H_k }  = ( - 1)^n 2O_{\left\lceil {n/2} \right\rceil }.
\end{equation}
Addition of~\eqref{n2x6xs0} and~\eqref{ruulewa} produces
\begin{align*}
\sum_{k = 0}^n {( - 1)^k \binom{{n}}{k}2^k H_{k + 1} }  
&=  \begin{cases}
 2O_{n/2}  + \dfrac{1}{{n + 1}},&\text{if $n$ is even;} \\ 
  - 2O_{(n + 1)/2},&\text{if $n$ is odd;}  \\ 
 \end{cases}\\
&=  \begin{cases}
 O_{n/2}  + O_{(n + 2)/2},&\text{if $n$ is even;}  \\ 
  - 2O_{(n + 1)/2},&\text{if $n$ is odd;}  \\ 
 \end{cases}
\end{align*}
from which~\eqref{kjbgzo0} follows. Identity~\eqref{hv15wkp} is obtained from~\eqref{kjbgzo0} by application of Theorem~\ref{thm.v41e7ah}.
\end{proof}
\begin{proposition}
If $n$ is a non-negative integer, then
\begin{align}
2\sum_{k = 0}^n {( - 1)^k \binom{{n}}{k}kO_{\left\lfloor {k/2} \right\rfloor } }  = \frac{{n2^{n - 1} }}{{n - 1}},\quad n\ne 1\label{t0ppyv0},\\
2\sum_{k = 0}^n {\binom{{n}}{k}kO_{\left\lfloor {k/2} \right\rfloor } }  = n2^{n - 1} H_{n - 1},\quad n\ne 1\label{pocicgb},
\end{align}
and
\begin{equation}\label{adnwkvc}
\sum_{k = 2}^n {( - 1)^k \binom{{n}}{k}k2^{k - 1} H_{k - 2} }
=\begin{cases}
 0,&\text{if $n$ is even;} \\ 
  - 4nO_{(n - 1)/2},&\text{if $n$ is odd.}  
 \end{cases}
\end{equation}
\end{proposition}
\begin{proof}
Identities~\eqref{t0ppyv0} and~\eqref{pocicgb} are obtained from~\eqref{kuotslw} and~\eqref{n6bsu21} by the application of Theorem~\ref{thm.v41e7ah}. Subtracting~\eqref{t0ppyv0} from~\eqref{pocicgb} gives
\begin{equation*}
\sum_{k = 0}^n {\binom{{n}}{k}kO_{\left\lfloor {k/2} \right\rfloor } \left( {1 - ( - 1)^k } \right)}  = n2^{n - 1} H_{n - 2} ,\quad n\ge 2,
\end{equation*}
from which~\eqref{adnwkvc} follows by binomial inversion.
\end{proof}

\begin{proposition}
If $m$ and $n$ are non-negative integers, then
\begin{align}
2\sum_{k = m}^n {(-1)^k\binom{{n}}{k}\binom{{k}}{m}O_{\left\lceil {k/2} \right\rceil } }  &= \frac{{( - 1)^{m+1}}}{{m + 1}}2^{n - m}\binom nm\sum_{k = 0}^m {( - 1)^k \binom{{k + n}}{n}\binom{{k + n}}{{m + 1}}^{ - 1} },\quad n>m,\label{vi6tqmr} \\
2\sum_{k = m}^n {\binom{{n}}{k}\binom{{k}}{m}O_{\left\lceil {k/2} \right\rceil } }  &= 2^{n - m} \binom{{n}}{m}H_n  + 2^{n - m} \binom{{n}}{m}\sum_{k = 1}^m {\frac{{( - 1)^{k - 1} }}{k}\binom{{m}}{k}\binom{{n}}{k}^{ - 1} }\label{k98363g}, 
\end{align}
and
\begin{equation}\label{pvb13l9}
\sum_{k = 1}^n {\frac{{( - 1)^k }}{k}\binom{{n}}{k}\binom{{m}}{k}^{ - 1} }  = H_m  - 2^{n - m + 1} \binom{{m}}{n}^{ - 1} \sum_{k = 1}^m {\binom{{m}}{k}\binom{{k}}{n}O_{\left\lceil {k/2} \right\rceil } },\quad m\ge n .
\end{equation}
\end{proposition}
In particular,
\begin{equation}
2\sum_{k = 0}^n {( - 1)^k \binom{{n}}{k}kO_{\left\lceil {k/2} \right\rceil } }  =  - 2^{n - 1} \frac{{n - 2}}{{n - 1}}
\end{equation}
and
\begin{equation}
2\sum_{k = 0}^n {\binom{{n}}{k}kO_{\left\lceil {k/2} \right\rceil } }  = 2^{n - 1} \left( {nH_n  + 1} \right).
\end{equation}
\begin{proof}
Differentiating~\eqref{comp1} $m$ times with respect to $x$ gives
\begin{align*}
&2\sum_{k = 0}^n {\binom{{n}}{k}\binom{{k}}{m}O_{\left\lceil {k/2} \right\rceil } x^{k - m} }\\
&\qquad= \left( {1 + x} \right)^{n - m} \binom{{n}}{m}H_n \\
&\qquad\qquad - \sum_{k = 1}^n {\frac{1}{k}\sum_{j = 0}^m {( - 1)^j \binom{{k}}{j}\binom{{n - k}}{{m - j}}\left( {1 - x} \right)^{k - j} \left( {1 + x} \right)^{n - k - m + j} } } 
\end{align*}
Identities~\eqref{vi6tqmr} and~\eqref{k98363g} are obtained by setting $x=-1$ and $x=1$, in turn, in the above expression. We used
\begin{equation*}
\binom{{k - m}}{{n - m}}\binom{{k}}{m} = \binom{{k}}{n}\binom{{n}}{m}.
\end{equation*}
Identity~\eqref{pvb13l9} is a rearrangement of~\eqref{k98363g} with an exchange of $m$ and $n$.
\end{proof}
\begin{proposition}
If $n$ is a non-negative integer, then
\begin{align}
\sum_{k = 0}^n {( - 1)^k \binom{{n}}{k}H_{\left\lfloor {k/2} \right\rfloor } }  &= \frac{{2^n  - 2}}{n}\label{bab1uf0},\\
\sum_{k = 0}^n {\binom{{n}}{k}H_{\left\lfloor {k/2} \right\rfloor } } & = 2^n H_n  - 2^{n + 1} \sum_{k = 1}^n {\frac{1}{{k2^k }}} 
\end{align}
and
\begin{equation}
\sum_{k = 0}^n {(-1)^k\binom{{n}}{k}2^k H_{\left\lfloor {k/2} \right\rfloor } }  = ( - 1)^n 3H_n  - ( - 1)^n 2H_{\left\lfloor {n/2} \right\rfloor }  + \sum_{k = 1}^n {( - 1)^{n - k} \frac{{3^k }}{k}} .
\end{equation}
\end{proposition}
\begin{proof}
Set $x=1$, $x=-1$ and $x=2$, in turn, in~\eqref{comp2}.
\end{proof}
\begin{proposition}
If $n$ is a non-negative integer, then
\begin{equation}
\sum_{k = 0}^n {( - 1)^k \binom{{n}}{k}\frac{{H_{\left\lceil {k/2} \right\rceil } }}{{k + 1}}}  = \frac{{2 - 2^{n + 1} }}{{\left( {n + 1} \right)^2 }}
\end{equation}
and
\begin{equation}
\sum_{k = 0}^n {( - 1)^k \binom{{n}}{k}\frac{{H_{\left\lceil {(k + 1)/2} \right\rceil } }}{{\left( {k + 1} \right)\left( {k + 2} \right)}}}  =  - \frac{{2 - 2^{n + 2} }}{{\left( {n + 1} \right)\left( {n + 2} \right)^2 }}.
\end{equation}
\end{proposition}
\begin{proof}
Application of Theorem~\ref{thm.u5nhiva} to~\eqref{bab1uf0}.
\end{proof}
\begin{proposition}
If $n$ is a non-negative integer, then
\begin{equation}
H_n  = \frac{1}{2}H_{\left\lfloor {n/2} \right\rfloor }  + O_{\left\lceil {n/2} \right\rceil }.
\end{equation}
\end{proposition}
\begin{proof}
Act on~\eqref{bab1uf0} with $\mathcal N_n$; making use of~\eqref{idh7u98} and~\eqref{kuotslw}.
\end{proof}
\begin{proposition}
If $n$ is a non-negative integer, then
\begin{align}
\sum_{k = 1}^n {( - 1)^k \binom{{n}}{k}\frac{{C_{k - 1} }}{{2^{2k} }}}  &= \frac{1}{2} - \frac{{2n + 1}}{{2^{2n + 1} }}\binom{{2n}}{n}\label{bdaobq2},\\
\sum_{k = 0}^n {( - 1)^k \binom{{n}}{k}\frac{{2^{ - 2k} C_k }}{{k + 1}}}  &= \left( {2n + 1} \right)\left( {2n + 3} \right)\frac{{2^{ - 2n} C_n }}{{n + 1}} - \frac{2}{{n + 1}}\label{ali4ba2},\\
\sum_{k = 1}^n {\binom{{n}}{k}\frac{{C_{k - 1} }}{{2^{2k} }}}  &= \frac{1}{2} - \frac{1}{{2^{2n + 1} }}\binom{{2n}}{n} + 2^{ - 2n} \binom{{2n}}{n}\sum_{k = 1}^{\left\lfloor {n/2} \right\rfloor } {2^{2k} C_{k - 1} \binom{{n}}{{2k}}\binom{{4k}}{{2k}}^{ - 1} }\label{nhdk9ch},
\end{align}
and
\begin{align}\label{m8z59d8}
\sum_{k = 0}^n {\binom{{n}}{k}\frac{{C_k }}{{2^{2k}(k+1) }}}  &=\frac{{2\left( {2n + 1} \right)}}{{2^{2n} \left( {n + 1} \right)}}C_n \sum_{k = 1}^{\left\lceil {n/2} \right\rceil } {2^{2k} C_{k - 1} \binom{{n + 1}}{{2k}}\binom{{4k}}{{2k}}^{ - 1} }\nonumber\\
&\qquad+ \frac{2}{{n + 1}} - \frac{{2n + 1}}{{2^{2n} \left( {n + 1} \right)}}C_n .
\end{align}
\end{proposition}

\begin{proof}
Write~\eqref{b0s5n16} as
\begin{equation*}
\sum_{k = 1}^n {\binom{{n}}{k}\binom{{x}}{{k + z}}}  = \binom{{n + x}}{{n + z}} - \binom xz
\end{equation*}
and set $z=0$ and $x=1/2$ to obtain
\begin{equation*}
\sum_{k = 1}^n {\binom{{n}}{k}\binom{{1/2}}{k}}  = \binom{{n + 1/2}}{n} - 1
\end{equation*}
and hence~\eqref{bdaobq2} since
\begin{equation*}
\binom{{1/2}}{s} = \frac{{( - 1)^{s - 1} }}{{2^{2s - 1} }}C_{s - 1} 
\end{equation*}
and
\begin{equation*}
\binom{{s + 1/2}}{s} = \frac{{2s + 1}}{{2^{2s} }}\binom{{2s}}{s}.
\end{equation*}
Setting $y=n$ and $x=1/2$ in the following identity~\cite[Equation (7.14), p.35]{quaintance10}:
\begin{equation*}
\sum_{k = 0}^n {( - 1)^k \binom{{x}}{k}\binom{{y}}{{n - k}}}  = \sum_{k = 0}^{\left\lfloor {n/2} \right\rfloor } {( - 1)^k \binom{x}{k}\binom{{y - x}}{{n - 2k}}} ,
\end{equation*}
where $x$ and $y$ are complex numbers, gives
\begin{equation*}
\sum_{k = 1}^n {\binom{{n}}{k}\frac{{C_{k - 1} }}{{2^{2k} }}}  = \frac{1}{2} - \frac{1}{2}\binom{{n - 1/2}}{n} + \sum_{k = 1}^{\left\lfloor {n/2} \right\rfloor } {\frac{{C_{k - 1} }}{{2^{2k} }}\binom{{n - 1/2}}{{n - 2k}}} ,
\end{equation*}
and hence~\eqref{nhdk9ch} since
\begin{equation*}
\binom{{s - 1/2}}{s} = 2^{ - 2s} \binom{{2s}}{s}
\end{equation*}
and
\begin{equation*}
\binom{{r - 1/2}}{{r - 2s}} = 2^{4s - 2r} \binom{{2r}}{r}\binom{{r}}{{2s}}\binom{{4s}}{{2s}}^{ - 1} .
\end{equation*}
Identities~\eqref{ali4ba2} and~\eqref{m8z59d8} follow from~\eqref{bdaobq2} and~\eqref{nhdk9ch} by application of Theorem~\ref{thm.u5nhiva}.
\end{proof}

\begin{proposition}
If $n$ is a non-negative integer, then
\begin{align}
\sum_{k = 0}^n {( - 1)^k \binom{{n}}{k}\frac{{C_k }}{{2^{2k} }}}  &= \frac{{2n + 1}}{{2^{2n} }}C_n ,\label{jxih6gu}\\
\sum_{k = 0}^n {( - 1)^k \binom{{n}}{k}\frac{{kC_k }}{{2^{2k} }}}  &= -\frac{{nC_n }}{{2^{2n} }},\label{por6dir}\\
\sum_{k = 0}^n {( - 1)^k \binom{{n}}{k}\frac{{C_{k + 1} }}{{2^{2k} }}}  &= \frac{{C_{n + 1} }}{{2^{2n} }},\label{k5ply5k}
\end{align}
and
\begin{equation}\label{kkkr7kf}
\sum_{k = 0}^n {( - 1)^k \binom{{n}}{k}2^{ - 2k} \binom{{2k}}{k}}  = 2^{ - 2n} \binom{{2n}}{n}.
\end{equation}
\end{proposition}

\begin{proof}
Set $x=1/2$ and $z=1$ in~\eqref{b0s5n16} to obtain~\eqref{jxih6gu} since
\begin{equation*}
\binom{{1/2}}{{s + 1}} = \frac{{( - 1)^s }}{{2^{2s + 1} }}C_s 
\end{equation*}
and
\begin{equation*}
\binom{{s + 1/2}}{{s + 1}} = \frac{{2s + 1}}{{2^{2s + 1} }}C_s .
\end{equation*}
Identities~\eqref{por6dir} and~\eqref{kkkr7kf} follow from~\eqref{jxih6gu} by Theorem~\ref{thm.hdq3q7r}. Identity~\eqref{k5ply5k} is obtained from~\eqref{por6dir} by Theorem~\ref{thm.u5nhiva}.
\end{proof}

\begin{remark}
Identity~\eqref{k5ply5k} was also reported in~\cite{donaghey76}.
\end{remark}
\begin{lemma}
If $x$ and $z$ are complex numbers that are not negative integers and $n$ is a non-negative integer, then
\begin{equation}\label{eeutcar}
\sum_{k = 0}^n {\binom{{n}}{k}\binom{{x}}{{k + z}}H_{x - k - z} }  = \binom{{n + x}}{{n + z}}\left( {H_x  - H_{n + x}  + H_{x - z} } \right)
\end{equation}
and
\begin{equation}\label{ulehqvz}
\sum_{k = 0}^n {\binom{{n}}{k}\binom{{x}}{{k + z}}H_{k + z} }  = \binom{{n + x}}{{n + z}}\left( {H_x  - H_{n + x}  + H_{n + z} } \right).
\end{equation}
\end{lemma}
\begin{proof}
Differentiate~\eqref{b0s5n16} with respect to $x$ to obtain
\begin{equation}\label{jnuo5az}
\sum_{k = 0}^n {\binom{{n}}{k}\binom{{x}}{{k + z}}\left( {H_x  - H_{x - k - z} } \right)}  = \binom{{n + x}}{{n + z}}\left( {H_{n + x}  - H_{x - z} } \right)
\end{equation}
and hence~\eqref{eeutcar} after using~\eqref{b0s5n16} again. Differentiating~\eqref{b0s5n16} with respect to $z$ gives
\begin{equation*}
\sum_{k = 0}^n {\binom{{n}}{k}\binom{{x}}{{k + z}}\left( {H_{x - k - z}  - H_{k + z} } \right)}  = \binom{{n + x}}{{n + z}}\left( {H_{x - z}  - H_{n + z} } \right),
\end{equation*}
which in conjunction with~\eqref{eeutcar} gives~\eqref{ulehqvz}.
\end{proof}

\begin{proposition}
If $n$ is a non-negative integer, then
\begin{equation}\label{tmw4arp}
\sum_{k = 0}^n {( - 1)^k \binom{{n}}{k}2^{ - 2k} \binom{{2k}}{k}O_k }  =  - 2^{ - 2n} \binom{{2n}}{n}O_n 
\end{equation}
and
\begin{equation}\label{i97oolq}
\sum_{k = 0}^n {( - 1)^k \binom{{n}}{k}2^{ - 2k} \frac{{\left( {2k + 1} \right)}}{{k + 1}}C_k O_{k + 1} }  = 2^{ - 2n} \frac{{\left( {2n + 1} \right)}}{{n + 1}}C_n O_{n + 1} .
\end{equation}
\end{proposition}

\begin{proof}
Set $x=-1/2$ and $z=0$ in~\eqref{jnuo5az} to obtain~\eqref{tmw4arp}. Identity~\eqref{i97oolq} is obtained from~\eqref{tmw4arp} by applying Theorem~\ref{thm.u5nhiva}.
\end{proof}
\begin{remark}
Identity~\eqref{tmw4arp} was also derived in~\cite{rathie22}.
\end{remark}
\begin{proposition}
If $n$ is a non-negative integer, then
\begin{align}
\sum_{k = 0}^n {( - 1)^k \binom{{n}}{k}\frac{{C_k O_k }}{{2^{2k} }}}  = 2n\frac{{C_n }}{{2^{2n} }} - \left( {2n + 1} \right)\frac{{C_n O_n }}{{2^{2n} }},\label{p975h0e}\\
\sum_{k = 0}^n {( - 1)^k \binom{{n}}{k}\frac{{kC_k O_k }}{{2^{2k} }}}  = \frac{{nC_n }}{{2^{2n} }}\left( {O_n  - 2} \right),\label{zxuhe9e}\\
\sum_{k = 0}^n {( - 1)^k \binom{{n}}{k}\frac{{C_{k + 1} O_{k + 1} }}{{2^{2k} \left( {k + 1} \right)}}}  = \left( {2n + 3} \right)\frac{{C_{n + 1} O_{n + 1} }}{{2^{2n} \left( {n + 1} \right)}} - \frac{{2C_{n + 1} }}{{2^{2n} }}\label{k7apiin},
\end{align}
and
\begin{equation}\label{blzgasq}
\sum_{k = 0}^n {( - 1)^k \binom{{n}}{k}\frac{{C_{k + 1} O_{k + 1} }}{{2^{2k} }}}  = \frac{{C_{n + 1} }}{{2^{2n} }}\left( {2 - O_{n + 1} } \right).
\end{equation}
\end{proposition}

\begin{proof}
Set $z=1$ and $x=1/2$ in~\eqref{jnuo5az} to obtain~\eqref{p975h0e} after some straightforward manipulation. Identity~\eqref{zxuhe9e} follows from~\eqref{p975h0e} by the application of Theorem~\ref{thm.xl9sjqb}. Identities~\eqref{k7apiin} and~\eqref{blzgasq} come respectively from~\eqref{p975h0e} and~\eqref{zxuhe9e} by Theorem~\ref{thm.u5nhiva}.
\end{proof}

\begin{proposition}
If $m$ and $n$ are positive integers such that $m\ge n$, then
\begin{equation}
\sum_{k = 0}^n {\binom{{n}}{k}\binom{{m}}{k}\binom{{2\left( {m - k} \right)}}{{m - k}}^{ - 1} \binom{{2k}}{k}^{ - 1} }  = \frac{{m }}{{m + n}}2^{2n}\binom{{m + n}}{n}\binom{{2m}}{m}^{ - 1} \binom{{2n}}{n}^{ - 1} 
\end{equation}
and
\begin{align}
&\sum_{k = 0}^n {\binom{{n}}{k}\binom{{m}}{k}\binom{{2\left( {m - k} \right)}}{{m - k}}^{ - 1} \binom{{2k}}{k}^{ - 1} O_k }\nonumber\\
&\qquad  = \frac{m}{{m + n}}2^{2n - 1} \binom{{m + n}}{n}\binom{{2m}}{m}^{ - 1} \binom{{2n}}{n}^{ - 1} \left( {H_{m - 1}  - H_{n + m - 1}  + 2O_n } \right).
\end{align}
\end{proposition}

\begin{proof}
Set $x=m-1$ and $z=-1/2$ in~\eqref{ulehqvz} and equate the respective rational and irrational parts of the resulting identity.
\end{proof}

\begin{proposition}
If $n$ is a non-negative integer, then
\begin{equation}\label{y633vtc}
\sum_{k = 0}^n {( - 1)^k \binom{{n}}{k}\frac{{C_k H_{k + 1} }}{{2^{2k} }}}  = \frac{{2n + 1}}{{2^{2n} }}C_n \left( {H_{n + 1}  - 2O_{n + 1}  + 2} \right)
\end{equation}
and
\begin{equation}\label{lkvngwd}
\sum_{k = 0}^n {( - 1)^k \binom{{n}}{k}\left( {2k + 1} \right)\frac{{C_k H_{k + 1} }}{{2^{2k} }}}  = \frac{1}{{2^{2n} }}C_n \left( {H_{n + 1}  - 2O_n } \right).
\end{equation}
\end{proposition}
\begin{proof}
Using $x=1/2$ and $z=1$ in~\eqref{ulehqvz} gives~\eqref{y633vtc} while $x=-1/2$ and $z=1$ yields~\eqref{lkvngwd}.
\end{proof}

\begin{proposition}
If $n$ is a non-negative iinteger, then
\begin{equation}\label{y5i4f4n}
\sum_{k = 0}^n {( - 1)^k \binom{{n}}{k}\frac{{2^{2k} }}{{C_k }}}  = \frac{3}{{\left( {2n - 1} \right)\left( {2n - 3} \right)}}
\end{equation}
and
\begin{equation}\label{gbhdfrg}
\sum_{k = 0}^n {( - 1)^k \binom{{n}}{k}\frac{{2^{2k} }}{{\left( {2k + 1} \right)C_k }}}  =  - \frac{1}{{\left( {2n + 1} \right)\left( {2n - 1} \right)}}.
\end{equation}
\end{proposition}

\begin{proof}
Consider the following identity~\cite{adegoke1}:
\begin{equation}\label{qtctp70}
\sum_{k = 0}^n {( - 1)^k \binom{{n}}{k}\binom{{k + u + v + 1}}{{u + 1}}^{ - 1} }  = \frac{{u + 1}}{{v + 1}}\binom{{n + u + v + 1}}{{v + 1}}^{ - 1} .
\end{equation}
Setting $v=1$ and $u=-5/2$ gives~\eqref{y5i4f4n} while $v=1$, $u=-3/2$ yields~\eqref{gbhdfrg} since
\begin{equation*}
\binom{{k - 1/2}}{{k + 1}} =  - \frac{{C_k }}{{2^{2k + 1} }}
\end{equation*}
and
\begin{equation*}
\binom{{k + 1/2}}{{k + 1}} = \frac{{2k + 1}}{{2^{2k + 1} }}C_k .
\end{equation*}
\end{proof}

\begin{proposition}
If $n$ is a non-negative integer, then
\begin{align}
\sum_{k = 0}^n {( - 1)^k \binom{{n}}{k}2^{2k} \binom{{2k}}{k}^{ - 1} }  &=  - \frac{1}{{2n - 1}},\label{pfkhcjx}\\
\sum_{k = 0}^n {( - 1)^k \binom{{n}}{k}\frac{{2^{2k} }}{{2k + 1}}\binom{{2k}}{k}^{ - 1} }  &= \frac{1}{{2n + 1}}\label{t3p282d},
\end{align}
and
\begin{equation}\label{u81fnua}
\sum_{k = 0}^n {( - 1)^k \binom{{n}}{k}\frac{1}{{k + 1}}}  = \frac{1}{{n + 1}}.
\end{equation}

\end{proposition}
\begin{proof}
Set $v=0$ and write $u-1$ for $u$ in~\eqref{qtctp70} to get
\begin{equation}\label{cbfpx5e}
\sum_{k = 0}^n {( - 1)^k \binom{{n}}{k}\binom{{k + u}}{k}^{ - 1} }  = \frac{u}{{n + u}}.
\end{equation}
Evaluate~\eqref{cbfpx5e} at $u=-1/2$,~$1/2$,~$1$ to obain~\eqref{pfkhcjx},~\eqref{t3p282d} and~\eqref{u81fnua}.
\end{proof}
\begin{lemma}
If $u$ and $v$ are complex numbers that are not negative integers, then
\begin{align}\label{em7996p}
\sum_{k = 0}^n {( - 1)^k \binom{{n}}{k}\binom{{k + v}}{u}^{ - 1} H_{k + v} }  &= \frac{u}{{v - u + 1}}\binom{{v + n}}{{v - u + 1}}^{ - 1} \left( {H_u  - H_{n + u - 1}  + H_{n + v} } \right)\nonumber\\
&\qquad - \frac{1}{{v - u + 1}}\,\binom{{v + n}}{{v - u + 1}}^{ - 1} 
\end{align}
and
\begin{align}\label{x7ia91s}
\sum_{k = 0}^n {( - 1)^k \binom{{n}}{k}\binom{{k + u + v}}{u}^{ - 1} } H_{k+v} &= \frac{u}{{v + 1}}\binom{{n + u + v}}{{v + 1}}^{ - 1} \left( {H_u  + H_{v + 1}  - H_{n + u - 1} } \right)\nonumber\\
&\qquad - \frac{{u + v + 1}}{{\left( {v + 1} \right)^2 }}\,\binom{{n + u + v}}{{v + 1}}^{ - 1} .
\end{align}

\end{lemma}

\begin{proof}
Differentiate~\eqref{qtctp70} with respect to $u$ and use~\eqref{qtctp70} again to obtain
\begin{align}\label{knccfjh}
&\sum_{k = 0}^n {( - 1)^k \binom{{n}}{k}\binom{{k + u + v + 1}}{{u + 1}}^{ - 1} H_{k + u + v + 1} }\nonumber \\
&\qquad= \frac{{u + 1}}{{v + 1}}\binom{{n + u + v + 1}}{{v + 1}}^{ - 1} \left( {H_{u + 1}  - H_{n + u}  + H_{n + u + v + 1} } \right)\nonumber\\
&\qquad\qquad - \frac{1}{{v + 1}}\binom{{n + u + v + 1}}{{v + 1}}^{ - 1}  .
\end{align}
Now replace $u$ with $u-1$ and $v$ with $v-u$; this gives~\eqref{em7996p}. Next differentiate~\eqref{qtctp70} to obtain
\begin{align}\label{f3uy3vj}
&\sum_{k = 0}^n {( - 1)^k \binom{{n}}{k}\binom{{k + u + v + 1}}{{u + 1}}^{ - 1} \left( {H_{k + u + v + 1}  - H_{k + v} } \right)}\nonumber\\
&\qquad  = \frac{{u + 1}}{{\left( {v + 1} \right)^2 }}\binom{{n + u + v + 1}}{{v + 1}}^{ - 1}  + \frac{{u + 1}}{{v + 1}}\binom{{n + u + v + 1}}{{v + 1}}^{ - 1} \left( {H_{n + u + v + 1}  - H_{v + 1} } \right).
\end{align}
Identity~\eqref{x7ia91s} now follows from~\eqref{knccfjh} and~\eqref{f3uy3vj} after replacing $u$ with $u-1$.
\end{proof}

\begin{proposition}
If $n$ is a non-negative integer, then
\begin{align}
\sum_{k = 0}^n {( - 1)^k \binom{{n}}{k}2^{ - 2k} \binom{{2k}}{k}H_k }  &= 2^{ - 2n} \binom{{2n}}{n}\left( {H_n  - 2O_n } \right)\label{vrl8z12},\\
\sum_{k = 0}^n {( - 1)^k \binom{{n}}{k}2^{ - 2k} \binom{{2k}}{k}\left( {H_k  - O_k } \right)}  &= 2^{ - 2n} \binom{{2n}}{n}\left( {H_n  - O_n } \right)\label{mdn83mz},
\end{align}
and
\begin{equation}\label{pqgy1pp}
\sum_{k = 0}^n {( - 1)^k \binom{{n}}{k}2^{ - 2k} \frac{{2k + 1}}{{k + 1}}C_k \left( {H_{k + 1}  - O_{k + 1} } \right)}  =  - 2^{ - 2n} \frac{{2n + 1}}{{n + 1}}C_n \left( {H_{n + 1}  - O_{n + 1} } \right).
\end{equation}
\end{proposition}

\begin{proof}
Set $v=0$, $u=1/2$ in~\eqref{em7996p} to obtain~\eqref{vrl8z12}. Identity~\eqref{mdn83mz} follows from~\eqref{vrl8z12} by applying Theorem~\ref{thm.hdq3q7r}. Identity~\eqref{pqgy1pp} is obtained from~\eqref{mdn83mz} through the application of Theorem~\ref{thm.u5nhiva}.
\end{proof}

\begin{proposition}
If $n$ is a non-negative integer, then
\begin{equation}\label{al1u5c8}
\sum_{k = 0}^n {( - 1)^k \binom{{n}}{k}2^{2k} \frac{{O_k }}{{C_k }}}  = \frac{{4n\left( {4n - 5} \right)}}{{\left( {2n - 1} \right)^2 \left( {2n - 3} \right)^2 }}
\end{equation}
and
\begin{equation}\label{vje1o4r}
\sum_{k = 0}^n {( - 1)^k \binom{{n}}{k}\frac{{2^{2k} }}{{k + 1}}\frac{{O_{k + 1} }}{{C_{k + 1} }}}  =  - \frac{{\left( {4n - 1} \right)}}{{\left( {2n + 1} \right)^2 \left( {2n - 1} \right)^2 }}.
\end{equation}
\end{proposition}
\begin{proof}
Set $v=-1/2$, $u=-3/2$ in~\eqref{em7996p} to obtain~\eqref{al1u5c8}, from which identity~\eqref{vje1o4r} then follows by Theorem~\ref{thm.u5nhiva}.
\end{proof}

\begin{proposition}
If $n$ is a non-negative integer, then
\begin{equation}\label{bwaf516}
\sum_{k = 0}^n {( - 1)^k \binom{{n}}{k}2^{2k} \binom{{2k}}{k}^{ - 1} H_k }  = \frac{{2\left( {O_{n - 1}  - 1} \right)}}{{2n - 1}}
\end{equation}
and
\begin{equation}\label{nnbxet1}
\sum_{k = 0}^n {( - 1)^k \binom{{n}}{k}2^{2k} \binom{{2k}}{k}^{ - 1} \frac{{H_{k + 1} }}{{2k + 1}}}  =  - \frac{{\left( {O_n  - 1} \right)}}{{\left( {n + 1} \right)\left( {2n + 1} \right)}}.
\end{equation}
\end{proposition}
\begin{proof}
Put $v=0$, $u=-1/2$ in~\eqref{x7ia91s}; this gives~\eqref{bwaf516} from which~\eqref{nnbxet1} follows by Theorem~\ref{thm.u5nhiva}.
\end{proof}

\begin{proposition}
If $n$ is a non-negative integer and $u$ is a complex number, then
\begin{equation}
\sum_{k = 0}^n {( - 1)^k \binom{{n}}{k}\binom{{2k}}{k}\binom{{2\left( {k + u} \right)}}{{k + u}}^{ - 1} \binom{{k + u}}{u}^{ - 1} }  = \frac{u\,2^{2n}}{{n + u}} \binom{{2\left( {n + u} \right)}}{{n + u}}^{-1}
\end{equation}
and
\begin{align}
&\sum_{k = 0}^n {( - 1)^k \binom{{n}}{k}\binom{{2k}}{k}\binom{{2\left( {k + u} \right)}}{{k + u}}^{ - 1} \binom{{k + u}}{u}^{ - 1} O_k }\nonumber\\
&\qquad  = \frac{{u\,2^{2n - 1} }}{{n + u}}\binom{{2\left( {n + u} \right)}}{{n + u}}^{ - 1} \left( {H_u  - H_{n + u - 1} } \right) - \frac{{2^{2n - 1} }}{{n + u}}\binom{{2\left( {n + u} \right)}}{{n + u}}^{ - 1} .
\end{align}
\end{proposition}

\begin{proof}
Set $v=-1/2$ in~\eqref{x7ia91s} and use
\begin{equation*}
\binom{{k + u - 1/2}}{u} = \binom{{2\left( {k + u} \right)}}{{k + u}}\binom{{k + u}}{u}2^{ - 2u} \binom{{2k}}{k}^{ - 1} .
\end{equation*}
Now equate respective rational and irrational parts in the resulting expression.
\end{proof}

\begin{proposition}
If $n$ is a non-negative integer, then
\begin{equation}\label{fh1t4t8}
\sum_{k = 0}^n {( - 1)^k \binom{{n}}{k}2^{2k} \frac{{H_k }}{{C_k }}}  = \frac{{2n + 8 - 6O_n }}{{\left( {2n - 1} \right)\left( {2n - 3} \right)}} + \frac{{24\left( {n - 1} \right)}}{{\left( {2n - 1} \right)^2 \left( {2n - 3} \right)^2 }}
\end{equation}
and
\begin{equation}\label{qt6wbom}
\sum_{k = 0}^n {( - 1)^k \binom{{n}}{k}\frac{{2^{2k} }}{{k + 1}}\frac{{H_{k + 1} }}{{C_{k + 1} }}}  = \frac{{3O_{n + 1}  - n - 5}}{{2\left(n+1\right)\left( {2n - 1} \right)\left( {2n + 1} \right)}} - \frac{{6n}}{{\left( {n + 1} \right)\left( {2n - 1} \right)^2 \left( {2n + 1} \right)^2 }}.
\end{equation}
\end{proposition}
\begin{proof}
Set $v=1$ and $u=-3/2$ in~\eqref{x7ia91s} to obtain~\eqref{fh1t4t8} from which~\eqref{qt6wbom} follows by Theorem~\ref{thm.u5nhiva}.
\end{proof}

\begin{proposition}
If $n$ is a non-negative integer and $r$ is a complex number that is not a negative integer, then
\begin{equation}\label{ww5fa8g}
\sum_{k = 0}^n {( - 1)^k \binom{{n}}{k}\binom{{k}}{{k + r}}}  
= \begin{cases}
 \dfrac{{\sin (\pi r)}}{{\pi \left( {r + n} \right)}},&\text{if $r\ne 0$;} \\ 
 0,&\text{if $r= 0$, $n\ne 0$;}  
 \end{cases}
\end{equation}
and
\begin{equation}\label{mdaqu57}
\sum_{k = 0}^n {( - 1)^k \binom{{n}}{k}\binom{{k}}{{k + r}}H_{k + r} }  
= \begin{cases}
 \dfrac{{\sin (\pi r)}}{{\pi \left( {r + n} \right)}}\left( {H_{r - 1}  + \dfrac{1}{{r + n}}} \right),&\text{if $r\ne 0$;} \\ 
 -1/n,&\text{if $r=0$, $n\ne 0$.}  
 \end{cases}
\end{equation}
\end{proposition}
\begin{proof}
Write the identity~\cite[Entry 1.5]{gould}:
\begin{equation}\label{lvdleh1}
\sum_{k = 0}^n {( - 1)^k \binom{{p}}{k}}  = ( - 1)^n \binom{{p - 1}}{n}
\end{equation}
as
\begin{equation*}
\sum_{k = 0}^n {( - 1)^k \binom{{p}}{{n - k}}}  = \binom{{p - 1}}{n}
\end{equation*}
and set $p=r+n$ to obtain
\begin{equation}
\sum_{k = 0}^n {( - 1)^k \binom{{r + n}}{{r + k}}}  = \binom{{r + n - 1}}{n},
\end{equation}
from which identity~\eqref{ww5fa8g} then follows upon using the identity
\begin{equation*}
\binom{{n+r}}{{k+r}}\binom{{n}}{{n + r}} = \binom{{n}}{k}\binom{{k}}{{k + r}},
\end{equation*}
and the fact that
\begin{equation*}
\binom{{a}}{b}\binom{{b}}{a} 
=\begin{cases}
 \dfrac{{\sin (\pi \left( {b - a} \right))}}{{\pi \left( {b - a} \right)}},&\text{if $b\ne a$;} \\ 
 1,&\text{if $b=a$.} \\ 
 \end{cases}
\end{equation*}
Identity~\eqref{mdaqu57} is obtained by differentiating~\eqref{ww5fa8g} with respect to $r$.
\end{proof}

\begin{proposition}\label{prop.ij2foq8}
If $n$ is a positive integer, then
\begin{align}
\sum_{k = 0}^n {( - 1)^k \binom{{n}}{k}H_k^2 }  &= \frac{{H_n }}{n} - \frac{2}{{n^2 }}\label{w9y5cwi},\\
\sum_{k = 0}^n {( - 1)^k \binom{{n}}{k}O_k }  &=  - \frac{{2^{2n - 1} }}{n}\binom{{2n}}{n}^{ - 1}\label{tj0vcug},
\end{align}
and
\begin{equation}
\sum_{k = 0}^n {( - 1)^k \binom{{n}}{k}O_k^2 }  = \frac{{2^{2n - 1} }}{n}\binom{{2n}}{n}^{ - 1} \left( {H_{n - 1}  - O_n } \right)\label{pj6sue2} .
\end{equation}

\end{proposition}

\begin{proof}
Setting $\ell=0$ in the identity~\cite{sofo14}:
\begin{equation}\label{aqvpl55}
\sum_{k = 0}^n {( - 1)^k \binom nkH_{k + \ell}^2 }  = \frac1n\binom{n + \ell}\ell^{ - 1} \left( {2H_{n - 1}  - H_{n + l}  - H_l } \right),\quad \ell\in\mathbb C\setminus\mathbb Z^{-},
\end{equation}
gives~\eqref{w9y5cwi}. Identities~\eqref{tj0vcug} and~\eqref{pj6sue2} are obtained by setting $\ell=-1/2$ in~\eqref{aqvpl55} and equating respective rational and irrational parts in the resulting expression.
\end{proof}
\begin{proposition}
If $n$ is a non-negative integer, then
\begin{align}
\sum_{k = 0}^n {( - 1)^k \binom{{n}}{k}\frac{{H_{k + 1}^2 }}{{k + 1}}}  &=  - \frac{{H_{n + 1} }}{{\left( {n + 1} \right)^2 }} + \frac{2}{{\left( {n + 1} \right)^3 }},\\
\sum_{k = 0}^n {( - 1)^k \binom{{n}}{k}\frac{{O_{k + 1} }}{{k + 1}}}  &= \frac{{2^{2n} }}{{\left( {n + 1} \right)\left( {2n + 1} \right)}}\binom{{2n}}{n}^{ - 1} ,
\end{align}
and
\begin{equation}
\sum_{k = 0}^n {( - 1)^k \binom{{n}}{k}\frac{{O_{k + 1}^2 }}{{k + 1}}}  =  - \frac{{2^{2n} }}{{\left( {n + 1} \right)\left( {2n + 1} \right)}}\binom{{2n}}{n}^{ - 1} \left( {H_n  - O_{n + 1} } \right).
\end{equation}
\end{proposition}
\begin{proof}
These identities follow from those in Proposition~\ref{prop.ij2foq8} by the application of Theorem~\ref{thm.u5nhiva}.
\end{proof}
\begin{proposition}
If $n$ is a non-negative integer and $m$ and $p$ are complex numbers, then
\begin{equation}\label{w7spap8}
\sum_{k = 0}^n {\binom{{n}}{k}\binom{{p}}{{k + m}}^{ - 1} }  = \frac{{p + 1}}{{p + 1 - n}}\binom{{p - n}}{m}^{ - 1} 
\end{equation}
and
\begin{align}\label{bxekmi9}
\sum_{k = 0}^n {\binom{{n}}{k}\binom{{p}}{{k + m}}^{ - 1} H_{k + m} }  &= \frac{{p + 1}}{{p + 1 - n}}\binom{{p - n}}{m}^{ - 1} \left( {H_p  - H_{p - n}  + H_m } \right)\nonumber\\
&\qquad - \frac{n}{{\left( {p +1- n} \right)^2 }}\binom{{p - n}}{m}^{ - 1} .
\end{align}
\end{proposition}
\begin{proof}
Identity~\eqref{w7spap8} is the same as~\eqref{foz9j01} on page~\pageref{foz9j01}, with $m$ and $p$ interchanged. From~\eqref{cd4s9t9} and~\eqref{mvuo316} we find
\begin{align*}
\sum_{k = 0}^n {\binom{{n}}{k}\binom{{m}}{{p + n - k}}^{ - 1} H_{m - p - n + k} }  &= \frac{{m + 1}}{{m + 1 - n}}\binom{{m - n}}{p}^{ - 1} \left( {H_m  - H_{m - n}  + H_{m - n - p} } \right)\\
&\qquad - \frac{n}{{\left( {m - n + 1} \right)^2 }}\binom{{m - n}}{p}^{ - 1} ,
\end{align*}
from which~\eqref{bxekmi9} follows after first replacing $m$ with $m+p+n$, and then replacing $p$ with $p-m-n$.
\end{proof}
\begin{proposition}
If $m$ is a non-negative integer and $n$ is a positive integer, then
\begin{align}
\sum_{k = 0}^n {( - 1)^k \binom{{n}}{k}B_m (k + 1)}  &= ( - 1)^n (n - 1)!m\braces{{ m}}{n}\label{czm1sgx},\\
\sum_{k = 0}^n {( - 1)^k \binom{{n}}{k}\frac{{B_m (k + 1)}}{{k + 1}}}  &= \frac{{B_m }}{{n + 1}} + \frac{{( - 1)^n n!}}{{n + 1}}m\braces{{ m - 1}}{n}\label{fsl6xqx},
\end{align}
and
\begin{equation}\label{kbx6493}
\sum_{k = 0}^n {( - 1)^k \binom{{n}}{k}\frac{{B_m (k + 2)}}{{\left( {k + 1} \right)\left( {k + 2} \right)}}}  = \frac{{B_m }}{{n + 2}} + \frac{{( - 1)^n n!}}{{n + 2}}m\braces{{ m - 1}}{{n + 1}}.
\end{equation}
\end{proposition}

\begin{proof}
By~\eqref{ib2djg2} we identify the binomial-transform pair $(t_k)$ and $(\tau_k)$, $k=0,1,2,\ldots$, where
\begin{equation}\label{ykzec1x}
t_k=k^m\text{ and }\tau_k=(-1)^kk!\braces mk.
\end{equation}
Now, consider the identity~\cite{apostol08}:
\begin{equation}
\sum_{j = 1}^k {j^m }  = \frac{{B_{m + 1} (k + 1) - B_{m + 1} }}{{m + 1}}.
\end{equation}
By Theorem~\ref{thm.ov7pa39}, we recognize the binomial-transform pair of the first kind $(a_k)$ and $(\alpha_k)$, $k=0,1,2,\ldots$, where
\begin{equation}
a_k  = \sum_{j = 1}^k {j^m }  = \frac{{B_{m + 1} (k + 1) - B_{m + 1} }}{{m + 1}},
\end{equation}
and
\begin{align}
\alpha _k  = \tau _k  - \tau _{k - 1} & = ( - 1)^k k!\braces{{ m}}{k} - ( - 1)^{k - 1} (k - 1)!\braces{{ m}}{{k - 1}}\nonumber\\
 &= ( - 1)^k (k - 1)!\left( {k\braces{{ m}}{k} + \braces{{ m}}{{k - 1}}} \right)\nonumber\\
 &= ( - 1)^k (k - 1)!\braces{{ m + 1}}{k},\quad\text{by~\eqref{qkejzsh}}.
\end{align}
Thus
\begin{equation}
\sum_{k = 0}^n {( - 1)^k \binom{{n}}{k}\frac{{B_{m + 1} (k + 1) - B_{m + 1} }}{{m + 1}}}  = ( - 1)^n (n - 1)!\braces{{ m + 1}}{n},
\end{equation}
and hence~\eqref{czm1sgx}. By Theorem~\ref{thm.ph6iklv} we have that
\begin{equation*}
\frac{{\sum_{j = 0}^k {t_j } }}{{k + 1}}\text{ and }\frac{{\tau _k }}{{k + 1}}
\end{equation*}
are a binomial-transform pair of the first kind, with $(t_k)$ and $(\tau_k)$ given in~\eqref{ykzec1x} and hence~\eqref{fsl6xqx}. Identity~\eqref{kbx6493} follows from~\eqref{fsl6xqx} by the application of Theorem~\ref{thm.u5nhiva}.
\end{proof}

\begin{proposition}
If $n$ is a non-negative integer, then
\begin{align}
\sum_{k = 0}^n {\binom{{n}}{k}\frac{kB_{k + 1} }{{k + 1}}}  &= \frac{1}{{n + 1}} - (-1)^nB_n  + \frac{nB_{n+1}}{{n + 1}}\label{sgj101q},\\
\sum_{k = 0}^n {( - 1)^k \binom{{n}}{k}\frac{{kB_{k + 1} }}{{k + 1}}}  &=  - \frac{1}{{n + 1}} + ( - 1)^n B_n  + ( - 1)^n \frac{{nB_{n + 1} }}{{n + 1}}\label{oz8lx8g}, \\
\sum_{k = 0}^n {\binom{{n}}{k}\frac{{B_{k + 2} }}{{k + 2}}}  &= \frac{1}{{\left( {n + 1} \right)\left( {n + 2} \right)}} + (-1)^n\frac{{B_{n + 1} }}{{n + 1}} + \frac{{B_{n + 2} }}{{n + 2}}\label{dvts2f1},
\end{align}
and
\begin{equation}\label{gnq7c91}
\sum_{k = 0}^n {(-1)^k\binom{{n}}{k}\frac{{B_{k + 2} }}{{k + 2}}}  = \frac{1}{{\left( {n + 1} \right)\left( {n + 2} \right)}} + (-1)^n\frac{{B_{n + 1} }}{{n + 1}} +(-1)^n \frac{{B_{n + 2} }}{{n + 2}}.
\end{equation}

\end{proposition}
\begin{proof}
Apostol~\cite{apostol08} derived an identity that can be written in the following equivalent form:
\begin{equation*}
\sum_{k = 2}^n {\binom{{n}}{{k - 2}}\frac{{B_k }}{k}}  = \frac{1}{{\left( {n + 1} \right)\left( {n + 2} \right)}} - (-1)^{n+1}B_{n + 1},
\end{equation*}
which produces~\eqref{sgj101q} by a shift of the summation index. Identity~\eqref{oz8lx8g} is obtained from~\eqref{sgj101q} by acting on both sides with the operator $\mathcal N_n$ defined in Lemma~\ref{btransform} on page~\pageref{btransform}. Identities~\eqref{dvts2f1} and~\eqref{gnq7c91} are obtained from~\eqref{sgj101q} and~\eqref{oz8lx8g} by Theorem~\ref{thm.u5nhiva}.
\end{proof}

\begin{proposition}
If $n$ is a non-negative integer, then
\begin{equation}\label{i2jitcy}
\sum_{k = 0}^n {\binom{{n}}{k}\frac{{B_{k + 1} }}{{k + 1}}}  = ( - 1)^{n + 1} \frac{{B_{n + 1} }}{{n + 1}} - \frac{1}{{n + 1}}.
\end{equation}

\end{proposition}

\begin{proof}
Operating on~\eqref{gnq7c91} with $\mathcal N_n$ gives
\begin{equation*}
\sum_{k = 0}^n {\binom{{n}}{k}\frac{{B_{k + 1} }}{{k + 1}}}  = -\sum_{k = 0}^n {\binom{{n}}{k}\frac{{B_{k + 2} }}{{k + 2}}}  - \frac{1}{{n + 2}} + ( - 1)^n \frac{{B_{n + 2} }}{{n + 2}},
\end{equation*}
from which~\eqref{i2jitcy} follows after using~\eqref{dvts2f1}.
\end{proof}

\bigskip
\hrule
\bigskip





\end{document}